\documentclass{amsart}
\usepackage{url}
\usepackage{graphicx}
\usepackage{color}
\usepackage{times}
\usepackage{cite}

\usepackage{enumerate,latexsym}
\usepackage{enumitem}
\usepackage{amsmath,amssymb}
\usepackage{amsthm}
\usepackage{verbatim}
\usepackage{mathrsfs}
\usepackage{tabularx}
\usepackage[dvipsnames]{xcolor}
\usepackage{hyperref}
\usepackage[capitalise]{cleveref} 
\crefname{equation}{}{}
\crefformat{enumi}{(#2#1#3)}

\newtheorem{thm}{Theorem}[section]
\newtheorem*{theorem*}{Theorem}
\newtheorem*{acknowledgement*}{Acknowledgement}
\newtheorem{cor}[thm]{Corollary}

\newtheorem{lem}[thm]{Lemma}
\newtheorem{prop}[thm]{Proposition}
\theoremstyle{definition}
\newtheorem{defn}[thm]{Definition}
\theoremstyle{remark}
\newtheorem{rem}[thm]{Remark}

\numberwithin{equation}{section}
\newcommand{\norm}[1]{\left\Vert#1\right\Vert}
\newcommand{\abs}[1]{\left\vert#1\right\vert}
\newcommand{\set}[1]{\left\{#1\right\}}
\newcommand{\Real}{\mathbb R}
\newcommand{\eps}{\varepsilon}

\newcommand{\func}[1]{\ensuremath{\mathop{\mathrm{#1}}} }

\newcommand{\Div}[0]{\func{div}}

\newcommand{\dist}[0]{\mathrm{dist}}

\newcommand{\spt}[0]{\func{spt}}

\newcommand{\xX}[0]{\mathbf{x}}

\newcommand{\fF}[0]{\mathbf{f}}

\newcommand{\nN}[0]{\mathbf{n}}
\newcommand{\cC}[0]{\mathcal{C}}

\newcommand{\R}{\mathbb{R}}

\title{Existence of monotone Morse flow lines of the expander functional}

\author{Jacob Bernstein}
\address{Department of Mathematics, Johns Hopkins University, 3400 N. Charles Street, Baltimore, MD 21218}
\email{jberns15@jhu.edu}

\author{Letian Chen}
\address{Department of Mathematics, University of Tennessee, 1403 Circle Drive, Knoxville, TN 37996}
\email{lchen63@utk.edu}

\author{Lu Wang}
\address{Department of Mathematics, Yale University, 219 Prospect Street, New Haven, CT 06511}
\email{lu.wang@yale.edu}

\thanks{The first-named author was partially supported by the NSF Grants DMS-1904674 and DMS-2203132. The third-named author was partially supported by the NSF Grant DMS-2146997.}

\begin{document}
\begin{abstract}
  Given a  smooth asymptotically conical self-expander that is strictly unstable we construct a (singular) Morse flow line of the expander functional that  connects it to a stable self-expander.  This flow is monotone in a suitable sense and has small singular set.
\end{abstract}
	\maketitle
	
\section{Introduction} \label{IntroSec}
A family $t\mapsto\Sigma_t\subseteq \mathbb{R}^{n+1}$ of hypersurfaces is an \emph{expander mean curvature flow} (EMCF) if  
\begin{equation}\label{RSMCF}
	\left(\frac{\partial\mathbf{x}}{\partial t}\right)^\perp=\mathbf{H}_{\Sigma_t} -\frac{\mathbf{x}^\perp}{2} 
\end{equation}
where 
$$
\mathbf{H}_{\Sigma_t}=\Delta_{\Sigma_t}\mathbf{x}=-H_{\Sigma_t}\mathbf{n}_{\Sigma_t}=-\mathrm{div}_{\Sigma_t}(\mathbf{n}_{\Sigma_t})\mathbf{n}_{\Sigma_t}
$$
is the mean curvature vector, $\mathbf{n}_{\Sigma_t}$ is the unit normal, and $\mathbf{x}^\perp$ is the normal component of the position vector $\mathbf{x}$. Formally, this is the negative gradient flow of the expander functional
\begin{equation}
\label{expander-functional}
E[\Sigma]=\int_{\Sigma} e^{\frac{|\mathbf{x}|^2}{4}} d\mathcal{H}^n
\end{equation}
where $\mathcal{H}^n$ is the $n$-dimensional Hausdorff measure. 

The static points of the flow \cref{RSMCF} satisfy
$$
\mathbf{H}_{\Sigma} -\frac{\mathbf{x}^\perp}{2} =\mathbf{0}
$$
and are called \emph{self-expanders}. Indeed, $\Sigma$ is a self-expander if and only if the family of homothetic hypersurfaces
$$
t\mapsto \tilde{\Sigma}_t=\sqrt{t}\, \Sigma, t>0
$$
is a \emph{mean curvature flow} (MCF), that is, a solution to the flow
$$
\left(\frac{\partial \mathbf{x}}{\partial t}\right)^\perp=\mathbf{H}_{\tilde{\Sigma}_t}.
$$
Self-expanders model the behavior of a MCF as it emerges from a conical singularity \cite{AngenentIlmanenChopp} as well as the long time behavior of solutions of the flow \cite{EckerHuisken}. 

Consider a cone, $\mathcal{C}\subseteq \mathbb{R}^{n+1}$, with vertex we always fix at the origin,  $\mathbf{0}$.    For $k\in\mathbb{N}$ and $\alpha\in [0,1)$, the cone $\mathcal{C}$ is \emph{of $C^{k,\alpha}$ regularity} if $\mathcal{C}\setminus\{\mathbf{0}\}$ is a hypersurface in $\mathbb{R}^{n+1}$ of regularity $C^{k,\alpha}$. A hypersurface $\Sigma\subseteq\mathbb{R}^{n+1}$ is \emph{$C^{k,\alpha}$-asymptotically conical} if, as $\rho\to 0^+$, $\rho\Sigma \to \mathcal{C}$ in $C^{k,\alpha}_{loc}(\mathbb{R}^{n+1}\setminus\{\mathbf{0}\})$\footnote{ A sequence of closed sets $A_i$ converges to $A$ in $C^{k,\alpha}_{loc}(U)$, for $U\subseteq \mathbb{R}^{n+1}$ open, when the $A_i$ converge as closed sets to $A$, and there is a smooth hypersurface $\Upsilon\subseteq U$ so that in each $V\subset\subset U$, for sufficiently large $i$,  $\partial A_i$ and $\partial A$ are,  respectively, the normal graphs of functions $u_i$ and $u$ on $\Upsilon$ and $u_i\to u$ in the $C^{k,\alpha}$ topology.} for $\mathcal{C}$ a $C^{k,\alpha}$-regular cone. In this case, $\mathcal{C}$ is called the \emph{asymptotic cone of $\Sigma$} and denoted by $\mathcal{C}(\Sigma)$. By \cite[Proposition 3.3]{BWSmoothCompactness}, if $\mathcal{C}$ is a $C^{k,\alpha}$-regular cone in $\mathbb{R}^{n+1}$ with $k+\alpha>2$, and if $\Sigma$ is a self-expander $C^1$-asymptotic to $\mathcal{C}$, then $\Sigma$ is $C^{k',\alpha'}$-asymptotic to $\mathcal{C}$ for all $k'+\alpha'<k+\alpha$. Thus, we simply say the self-expander $\Sigma$ is asymptotic to the cone $\mathcal{C}$.  

Given a cone, $\mathcal{C}$, there may be more than one self-expanding solution emerging from it. That is, there may exist distinct self-expanders, $\Sigma$ and $\Sigma'$, both asymptotic to $\mathcal{C}$. In this case, there may also be other, non-self-similar flows coming out of the cone. For instance, there are non-self-similar solutions that may be interpreted as Morse flow lines of the expander functional interpolating between two different critical points, i.e., eternal solutions of \eqref{RSMCF} that limit to one self-expander as $t\to -\infty$ and to another as $t\to \infty$. In general,  these flows must be realized in a weak sense due to possible singularities.

We are particularly interested in studying weak formulations of \eqref{RSMCF} in the set- and measure-theoretic sense. Namely, we will consider \emph{expander weak flow} of closed sets -- see \cref{WeakFlowSec} for the precise definition and its various generalizations, and \emph{expander Brakke flow} of Radon measures -- see \cref{ProofWeakThmSec} and \cite[Section 13]{HershkovitsWhite} for the definition.
We restrict to the smaller class of \emph{unit-regular} expander Brakke flows, i.e., at spacetime points of Gaussian density $1$, the the flow is smooth in a spacetime neighborhood -- cf. \cite{WhiteRegularity}.

Our main result is the construction of suitable weak Morse flow lines asymptotic at time $-\infty$ to unstable self-expanders. As the precise properties of these flow lines will be important in application, we first summarize them in the following definition:
\begin{defn}
	\label{strictly-monotone-flow}
	Let $\Omega$ be an expander weak flow in $\mathbb{R}^{n+1}$ with starting time $T_0\in\mathbb{R}$ and let $M(t)=\partial\Omega(t)$. We say $\Omega$ is \emph{monotone} if 
	\begin{enumerate}
	\item \label{monotoneCond} $\Omega(t_2)\subseteq \Omega(t_1)$ for $T_0\leq t_1\leq t_2<\infty$.
	\end{enumerate}
	The flow is \emph{strictly monotone} if Item \cref{monotoneCond} can be strengthened to 
	\begin{enumerate}[resume]
		\item $\Omega(t_2)\subseteq \mathrm{int}(\Omega(t_1))$ for $T_0\leq t_1<t_2<\infty$.
	\end{enumerate}
	Given a $C^3$-regular cone $\mathcal{C}\subseteq\mathbb{R}^{n+1}$, the flow is \emph{asymptotic to $\mathcal{C}$} if
	\begin{enumerate}[resume]
		 \item Given $\eps>0$ there is a radius $R_0>1$ so that for $t\in [T_0,\infty)$ there is a $C^2$ function $u(\cdot, t)\colon\mathcal{C}\setminus B_{R_0}\to \mathbb{R}$ satisfying
                          \[
    			\sup_{p\in\mathcal{C}\setminus B_{R_0}}\sum_{i=0}^2 |\mathbf{x}(p)|^{i-1}|\nabla_\mathcal{C}^iu(p,t)|\leq \eps \mbox{ and }
    			\]
    			\[
    			M(t)\setminus B_{2R_0}\subseteq\set{\mathbf{x}(p)+u(p,t)\mathbf{n}_\mathcal{C}(p)\mid p\in\mathcal{C}\setminus B_{R_0}}\subseteq M(t).
    			\]
	\end{enumerate}
	The flow is \emph{regular} if:
	\begin{enumerate}[resume]
		\item For $[a,b]\subseteq [T_0,\infty)$, $t\in [a,b]\mapsto M(t)$ forms a partition of $\Omega(a)\setminus\mathrm{int}(\Omega(b))$; 
		\item $t\in[T_0, \infty)\mapsto \mathcal{H}^n\llcorner M(t)$ is a unit-regular expander Brakke flow that is smooth away from a set of parabolic Hausdorff dimension $(n-1)$ in spacetime.
	\end{enumerate} 
	The flow is \emph{strongly regular} if, in addition,
	\begin{enumerate}[resume]
			\item Any limit flow is \emph{convex} in the sense that it satisfies (1)-(3) of Theorem \ref{StrongRegularThm}.     In particular, for any tangent flow $\tilde{\Omega}$, $t\in\mathbb{R}\mapsto \tilde{M}(t)=\partial \tilde{\Omega}(t)$ is either a static $\mathbb{R}^n$ or a self-similarly shrinking $\mathbb{S}^k\times\mathbb{R}^{n-k}$ for some $1\leq k\leq n$. 
	\end{enumerate}
\end{defn}
We refer to \cref{BlowupDefn} for the definitions of blow-up sequence and tangent flow and to Theorem \ref{StrongRegularThm} for additional details. Using this definition,  we introduce a notion of strictly monotone Morse flow line of the expander functional.
\begin{defn} \label{MorseFlowDefn}
	Let $\mathcal{C}\subseteq\mathbb{R}^{n+1}$ be a $C^3$-regular cone, and let $\Sigma_\pm$ be both smooth self-expanders asymptotic to $\mathcal{C}$. A \emph{strictly monotone expander Morse flow line from $\Sigma_-$ to $\Sigma_+$} is a closed subset $\Omega\subseteq\mathbb{R}^{n+1}\times\mathbb{R}$ with $M(t)=\partial\Omega(t)$ that satisfies: 
	\begin{enumerate}
		\item For any $T_0\in\mathbb{R}$ the restriction of $\Omega$ onto $[T_0,\infty)$ is a strongly regular strictly monotone expander weak flow asymptotic to $\mathcal{C}$ with starting time $T_0$;
		\item For $\Omega_-=\mathrm{cl}(\bigcup_{t\in\mathbb{R}} \Omega(t))$, $\partial\Omega_-=\Sigma_-$ and $\lim_{t\to -\infty} M(t)= \Sigma_-$ in $C^\infty_{loc}(\Real^{n+1})$;
		\item For $\Omega_+=\bigcap_{t\in\mathbb{R}} \Omega(t)$, $\partial\Omega_+=\Sigma_+$ and $\lim_{t\to \infty}  M(t)=  \Sigma_+$ in $C^\infty_{loc}(\Real^{n+1})$.
	\end{enumerate}
\end{defn}

The main result of this paper is the existence of such a strictly monotone expander Morse flow line asymptotic at time $-\infty$ to any given unstable self-expander.
\begin{thm}
	\label{main-theorem}
	For $2\leq n \leq 6$, let $\mathcal{C}$ be a $C^3$-regular cone in $\Real^{n+1}$. Suppose $\Omega_-\subseteq\mathbb{R}^{n+1}$ is a closed set with $\Sigma_-=\partial\Omega_-$ a smooth self-expander asymptotic to $\mathcal{C}$. If $\Sigma_-$ is strictly unstable, i.e.,  each component is unstable, then there exists a stable smooth self-expander $\Sigma_+$ asymptotic to $\mathcal{C}$ and a strictly monotone expander Morse flow line $\Omega$ from $\Sigma_-$ to $\Sigma_+$ with $\Omega_-=\mathrm{cl}(\bigcup_{t\in\mathbb{R}}\Omega(t))$. Moreover, if $\Omega'\subseteq \Real^{n+1}$ is a closed set so $\partial \Omega'$ is a smooth self-expander asymptotic to a $C^3$-regular cone $\mathcal{C}'$, and $\Omega'\subseteq \mathrm{int}(\Omega_-)$, then $\Omega'\subseteq \bigcap_{t\in\mathbb{R}}\Omega(t)$.
\end{thm}
\begin{rem}
	By replacing $\Omega_-$ with $\Real^{n+1}\setminus \mathrm{int}(\Omega_-)$, Theorem \ref{main-theorem} produces a second strictly monotone Morse flow line on the other side of $\Sigma_-$.  
\end{rem}

To prove this result we will need an existence and regularity result for strictly expander mean convex flows that are well-behaved at infinity. A closed set $U\subseteq \Real^{n+1}$ is \emph{strictly expander mean convex} if $\partial U$ is smooth and $\mathbf{H}_{\partial U}-\frac{\mathbf{x}^\perp}{2}$ is nowhere vanishing on $\partial U$ and points into the (non-empty) interior of $U$, $\mathrm{int}(U)$.
Theorem \ref{main-theorem} is a consequence of the following existence and regularity result.
\begin{thm}
	\label{weak-theorem}
	Let $\mathcal{C}$ be a $C^3$-regular cone in $\Real^{n+1}$.  Suppose $\Omega_0\subseteq \Real^{n+1}$ is a strictly expander mean convex closed set so that $\partial \Omega_0$ is a smooth hypersurface that is $C^3$-asymptotic to $\mathcal{C}$. There is a regular strictly monotone expander weak flow $\Omega$ asymptotic to $\mathcal{C}$ that starts from $\Omega_0$ at time $0$ and so that as $t\to\infty$ the $\partial\Omega(t)$ converge smoothly away from a codimension-$7$ closed set to a stable (possibly singular) self-expander asymptotic to $\mathcal{C}$. Moreover, when $n \leq 6$, this flow is strongly regular.	 
	
  In addition, if there is a closed set $\Omega'\subseteq \mathrm{int}(\Omega_0)$ with $\partial \Omega'$ smooth self-expander that is asymptotic to a $C^3$-regular cone, $\mathcal{C}'$, then one can construct the flow so $\Omega'\subseteq \bigcap_{t\geq 0}\Omega(t)$.
\end{thm}
This generalizes \cite[Proposition 5.1]{BWTopologicalUniqueness},  where stronger hypotheses on entropy and the asymptotic decay of the expander mean curvature are used to produce a globally smooth strictly monotone flow using classical mean curvature flow.  

A difficulty in the construction of the flow in Theorem \ref{weak-theorem} comes from the lack of a general avoidance principle for non-compact weak flows -- see for instance \cite[Example 7.3]{IlmanenIndiana}. In particular, while the PDE methods used in \cite{BWTopologicalUniqueness} allow one to show that the initial expander mean convexity condition is preserved as long as the flow remains smooth, it is not immediately clear that an appropriate weak notion of expander mean convexity is preserved for the biggest flow, which is the ``hull" of all weak flows of the given initial set. To handle this, we construct the flow as an exhaustion limit of compact expander mean convex flows. We then show that the flow is what we call \emph{ample}, a localized notion of being a biggest flow -- see Section \ref{AmpleSec}. This allows one to adapt the strategy of White \cite{WhiteMCSize,WhiteMCNature} for compact mean convex flows to address the regularity of the noncompact flow constructed above. Recently, White \cite{WhiteNonCompactAvoidance} has established a quite general avoidance principle for non-compact hypersurfaces that are a positive distance apart.  This does not seem to suffice for our application as the flows we consider are all asymptotic to a fixed cone and so different time slices have distance zero.

Finally, we note a related construction of a (singular) monotone Morse flow line of the shrinker functional in \cite{CCMSGeneric} coming out of an asymptotically conical self-shrinker.  This is a crucial ingredient used there to study generic mean curvature flows.   As in this paper, the initial construction is based off a compact exhaustion and the asymptotic regularity of the flow, though there are several differences related to the different geometric properties of the shrinker and expander functionals. Most notably, unlike the situation considered here, the geometry of the shrinker equation does not allow one to take these compact flows to be themselves monotone. Our approach stresses the preservation of expander mean convexity, and the corresponding regularity, under appropriate limiting process.


This paper is organized as follows. In \cref{PrelimSec} we introduce weak set flows generated by a spacetime set and prove some basic properties for weak flows. Indeed, we work in an ambient Riemannian manifold and consider a more general set-theoreric weak flow associated to MCF with a transport term $X$. Some of the proofs are deferred to \cref{ProvePrelim} because they are straightforward modifications of the proofs of \cite{HershkovitsWhite}. \cref{BiggestFlowSec} focuses on a special class of weak flows called biggest flows. In \cref{MonotoneXMeanConvexSec} we define a weak notion of mean convexity (with respect to mean curvature with a transport term) for spacetime sets. We also show strict monotonicity and closedness under limiting process for biggest flows generated by mean convex spacetime sets, which we call mean convex flows. In \cref{AmpleSec} we define ample flows and construct strictly monotone ample flows using suitable approximations by compact mean convex flows. \cref{ProofWeakThmSec} is devoted to the proof of \cref{weak-theorem} by applying the results of previous sections with the transport term corresponding to the expander weak flow. In \cref{MonotoneMorseFlowSec} we construct an appropriate ancient expander flows, which combined with \cref{weak-theorem} implies \cref{main-theorem}. In \cref{RegularXMCSec} we show how to smooth certain singular $X$-mean-convex domains while preserving the $X$-mean-convexity, which is used in the construction in \cref{ProofWeakThmSec}.

\section{Preliminaries} \label{PrelimSec}
In this section we introduce a notion of weak $X$-flow generated by a spacetime set. This is a generalization of a classical $X$-flow with prescribed heat boundary and will be a basic tool used in the sequel. We establish some properties of weak $X$-flows generated by a spacetime set, including an avoidance principle. We also study biggest $X$-flows generated by a spacetime set that contains all weak $X$-flows generated by the set.  This synthesizes ideas from \cite{WhiteTopology, WhiteSubsequent, HershkovitsWhite}.

We fix some notation and conventions used throughout the paper. First of all, let $N$ be a Riemannian manifold and $X$ be a smooth vector field on $N$.  We say $(N,X)$ is \emph{tame} if $N$ is complete, has Ricci curvature uniformly bounded from below and $|\nabla X|$ is uniformly bounded from above.
When $A\subseteq B\subseteq N$ or $A\subseteq B \subseteq N\times \mathbb{R}$, we denote by 
\[
\mathrm{int}_B(A), \mathrm{cl}_B(A), \mbox{ and } \partial_B A
\]
respectively the interior of the set $A$, the closure of the set $A$, and the boundary of the set $A$ all in the subspace topology of $B$. We omit the subscript $B$ when it is equal to $N$ or $N\times \mathbb{R}$.  Given a spacetime subset $A\subseteq N\times \Real$ and $t\in \Real$ the time $t$-slice of $A$ is denoted
$$
A(t)=\set{x\in N \mid (x,t)\in A}\subseteq N.
$$
Likewise, for an interval $I\subseteq \Real$ we denote the time restriction
$$
A|_I=\set{(x,t)\in A \mid t\in I}\subseteq N\times \Real.
$$
We denote by $B_r(p)$ and $\bar{B}_r(p)$ the open and closed balls in $N$ of radius $r$ centered at $p$;  when $N = \R^{n+1}$, omit the center $p$ if it is the origin. For $T_0\in\mathbb{R}$ and subset $U\subseteq N$, let
\[
C_{U,T_0}=U\times [T_0,\infty)\subseteq N \times \mathbb{R}
\]
be the spacetime cylinder over $U$ starting from $T_0$.  Finally, a closed set $U\subseteq N$ is \emph{smooth} if it has smooth (topologyical) boundary.  Such $U$ is \emph{strictly $X$-mean-convex} if the $X$-mean-curvature vector of $\partial U$ does not vanish and points into $\mathrm{int}(U)$ -- and so $U=\mathrm{cl}(\mathrm{int}(U))$. 

\subsection{Barriers} \label{BarrierSec}
Following Hershkovits--White \cite{HershkovitsWhite}, a family $t\in [a,b]\mapsto K(t)$ of closed subsets of $N$ is a \emph{smooth barrier} in $N$ if it is a smooth family of closed smooth regions. The barrier is \emph{compact} if $\bigcup_{t\in [a,b]} K(t)$ is compact. We think of the barrier $K$ as a subset of spacetime, so $K(t)$ is the time $t$-slice in $N$.

For a smooth barrier $K$, there exists a smooth function $f\colon N\times [a,b]\to \mathbb{R}$ such that $K(t)=\{x \mid f(x,t)\leq 0\}$ and that $|\nabla f(x,t)|\neq 0$ on $\partial K(t)$ for $t\in [a,b]$, where $\nabla$ is the gradient on $N$. For $x\in\partial K(t)$, let $\mathbf{n}_K(x,t)$ be the unit normal to $\partial K(t)$ that points outside $K(t)$, let $v_K(x,t)$ be the normal velocity of $s\mapsto \partial K(s)$ in the direction of $\mathbf{n}_K$, and let $H_K^X(x,t)$ be the scalar $X$-mean-curvature of $\partial K(t)$ given by 
\[
H_K^X(x,t)=-{\Div}_{\partial K(t)}\mathbf{n}_K(x,t)+X\cdot\mathbf{n}_K(x,t).
\]
It is readily checked that
\begin{align*}
\mathbf{n}_K =\frac{\nabla f}{|\nabla f|},  \; \; \; v_K  =-\frac{1}{|\nabla f|}\frac{\partial f}{\partial t},\\
H_K^X  =-\Div\left(\frac{\nabla f}{|\nabla f|}\right)+X\cdot \frac{\nabla f}{|\nabla f|} , 
\end{align*}
where $\Div$ is the divergence on $N$. For $x\in\partial K(t)$, define
\[
\Phi_K^X(x,t)=v_K(x,t)-H_K^X(x,t).
\]
In particular, $\Phi_K^X\equiv 0$ if and only if $\{\partial K(t)\}_{t\in [a,b]}$ is a classical $X$-flow. Finally, a \emph{strong $X$-flow barrier} is a smooth compact barrier in $N$, $t\in [a,b]\mapsto K(t)$ such that $\Phi_K^X(x,t)<0$ for all $(x,t)$ with $x\in\partial K(t)$ and $t\in [a,b]$. For the sake of simplicity, we will use the term \emph{strong barrier} when it is clear from context.

Crucially, shrinking spheres of sufficiently small radii are strong $X$-flow barriers.
\begin{prop}
	\label{spherical-barrier}
	Given $p\in N$, there exist positive constants $c_0=c_0(X,p,n)$ and $\delta_0=\delta_0(p)$ such that for every $c>c_0$, $0<\delta < \delta_0$ and $0<\tau<\delta^2/c$, 
	\[
	t\in [0,\tau] \mapsto K(t)=\bar{B}_{\sqrt{\delta^2 - ct}}(p)
	\]
	is a strong $X$-flow barrier.
\end{prop}
\begin{proof}
	Let $f(x,t)=\dist(x,p)^2-(\delta^2-ct)$, so $K(t)=\{x\mid f(x,t)\leq 0\}$. Choose $\delta_0=\delta_0(p)>0$ so that $B_{\delta_0}(p)$ is a normal ball in $N$. If $0<\delta<\delta_0$ and $0<\tau<\delta^2/c$, then $f$ is smooth in a neighborhood of $K$ and $|\nabla f(x,t)|\neq 0$ on $\partial K(t)$ for $t\in [0,\tau]$. One computes, for $x\in \partial K(t)$ and $t\in [0,\tau]$,
	\[
	v_K(x,t)=-\frac{c}{2\sqrt{\delta^2-ct}}
	\]
	and, shrinking $\delta_0$ if necessary, 
	\[
	H_K^X(x,t)\geq -\frac{2n}{\sqrt{\delta^2-ct}}-\sup_{B_{\delta_0}(p)}|X|.
	\]
	Thus, letting 
	\[
	c_0=4n+2\delta_0\sup_{B_{\delta_0}(p)}|X|
	\]
	if $c>c_0$ then $\Phi_K^X(x,t)<0$ for $x\in \partial K(t)$ and $t\in [0,\tau]$.
\end{proof}

\subsection{Weak $X$-flows generated by a spacetime set} \label{WeakFlowSec}
We now introduce weak $X$-flows generated by a natural class of closed subsets of spacetime.
\begin{defn} 
	\label{weak-x-flow-defn}
	Fix a closed set $\Gamma \subseteq C_{N,T_0}$. A closed set $Z \subseteq C_{N,T_0}$ is a \emph{weak $X$-flow generated by $\Gamma$ with starting time $T_0$} provided the following hold:
	\begin{itemize}
		\item $Z(T_0) = \Gamma(T_0)$ and $\Gamma\subseteq Z$;
		\item For any smooth compact barrier 	
		\[
		t\in [a,b]\mapsto K(t)
		\]
		such that $a\geq T_0$ and $K(t)\cap Z(t)=\emptyset$ for $t\in [a,b)$, if $p \in K(b)\cap Z(b)$, then either $p\in \partial K(b)$ and $\Phi_K^X(p,b)\geq 0$, or $p\in \Gamma(b)$.
	\end{itemize}
\end{defn} 
We will omit the starting time when it is taken to be $0$. We will also omit the vector field $X$ when it is the zero vector field. When $X \equiv \mathbf{0}$, White \cite{WhiteTopology} defines a weak flow by mean curvature generated by a closed subset of spacetime that includes hypersurfaces with boundary. In this case, \cref{weak-x-flow-defn} seemingly differs from that of White, but it is shown by \cref{equivalent-defn} below that the two definitions are indeed equivalent. We also remark that when $\Gamma\subseteq N\times\{T_0\}$, a weak $X$-flow generated by $\Gamma$ with starting time $T_0$ is a weak $X$-flow with starting time $T_0$ in Hershkovits--White's definition (see Section 11 of \cite{HershkovitsWhite}). If, in addition, $X=-\frac{\mathbf{x}}{2}$ on $N=\mathbb{R}^{n+1}$, this is an \emph{expander weak flow}.

We make some simple observations.
When $t\in [0,T]\mapsto M(t)$ is a smooth family of smooth hypersurfaces (possibly with boundary) and 
\[
\Gamma=\set{(x,0) \mid x\in M(0)} \cup \set{(x,t) \mid x\in \partial M(t), \, t\in [0,T]},
\]
then $M$ is a weak $X$-flow generated by $\Gamma$ if and only if it is a classical $X$-flow (possibly with boundary).  Suppose $Z$ is a weak $X$-flow generated by $\Gamma$ with starting time $T_0$. For $T\geq T_0$, the restriction $Z|_{[T,\infty)}$ of $Z$ to $[T,\infty)$ is a weak $X$-flow generated by $Z(T)\cup \Gamma|_{[T,\infty)}$ with starting time $T$, and $Z'=Z|_{[T_0,T]} \cup \Gamma$ is a weak $X$-flow generated by $\Gamma$ with starting time $T_0$. 

With minor modification to the arguments of Theorem 16 in \cite{HershkovitsWhite}, we establish the following avoidance principle for weak $X$-flows generated by spacetime sets -- see \cref{ProveAvoidanceSec} for the proof.
\begin{thm} \label{AvoidanceThm}
	For $i=1,2$, let $Z_i$ be a weak $X$-flow generated by a closed set $\Gamma_i\subseteq C_{N,T_0}$ with starting time $T_0$. Suppose $T\geq T_0$ and $\bigcup_{t\in [T_0,T]} Z_1(t)$ is compact. If  
	\[
	\Gamma_1(t)\cap Z_2(t)=\Gamma_2(t)\cap Z_1(t)=\emptyset
	\]
	for all $t\in [T_0,T]$, then $Z_1(t)\cap Z_2(t)=\emptyset$ for every $t\in [T_0,T]$.
\end{thm}

We give the following criterion that guarantees the compactness of the time slices of the weak $X$-flows generated by spacetime sets considered above. This is analogous to \cite[Theoerm 29]{HershkovitsWhite}, and we defer its proof to \cref{ProveCompactFlowSec}.
\begin{prop} \label{CompactFlowProp}
Suppose $(N,X)$ is tame.
	Let $Z$ be a weak $X$-flow generated by a closed set $\Gamma\subseteq C_{N,T_0}$ with starting time $T_0$. If $T\geq T_0$ and $\bigcup_{t\in [T_0,T]} \Gamma(t)$ is compact, then $\bigcup_{t\in [T_0,T]} Z(t)$ is compact.
\end{prop}

It is convenient to present several equivalent definitions for weak $X$-flows generated by a spacetime set. In particular, \cref{weak-x-flow-defn} is equivalent to the one of White \cite{WhiteTopology} when $X\equiv \mathbf{0}$. The following theorem is analogous to \cite[Theorem 17]{HershkovitsWhite}, and we refer the interested reader to \cref{ProveEquivDefnSec} for a detailed proof.
\begin{thm}
	\label{equivalent-defn}
	Let $\Gamma, Z$ be closed subsets of $C_{N,T_0}$ with  $Z(T_0) = \Gamma(T_0)$ and $\Gamma\subseteq Z$. The following are equivalent:
	\begin{enumerate}
		\item $Z$ is a weak $X$-flow generated by $\Gamma$ with starting time $T_0$; 
		\item If $t\in [a,b]\mapsto K(t)$ is a strong $X$-flow barrier in $N$ such that $a\geq T_0$, $K(a) \cap Z(a) = \emptyset$, and $K(t)\cap \Gamma(t)=\emptyset$ for all $t\in [a,b]$, then $K(t)\cap Z(t)=\emptyset$ for all $t\in [a,b]$;
		\item If $t\in [a,b]\mapsto K(t)$ is a strong $X$-flow barrier in $N$ such that $a\geq T_0$, $K(a) \cap Z(a) = \emptyset$, and $K(t)\cap \Gamma(t)=\emptyset$ for all $t\in [a,b]$, then $K(t)\cap Z(t)=\emptyset$ for all $t\in [a,b)$;
		\item If $t\in [a,b]\mapsto M(t)$ is a smooth $X$-flow of closed embedded hypersurfaces in $N$ such that $a\geq T_0$, $M(a)\cap Z(a)=\emptyset$, and $M(t)\cap \Gamma(t)=\emptyset$ for all $t\in [a,b]$, then $M(t)\cap Z(t)=\emptyset$ for all $t\in [a,b]$. 
	\end{enumerate}
\end{thm}

We observe the following useful localization property for weak $X$-flows in cylinders. 
\begin{prop}\label{WeakFlowLocalization}
	Suppose $U\subseteq N$ is a closed set. If $Z$ is a weak $X$-flow generated by a closed set $\Gamma\subseteq C_{N,T_0}$ with starting time $T_0$, then for any $T_1\geq T_0$
	\[
	Z'=Z \cap C_{U, T_1}
	\]
	is a weak $X$-flow generated by $\Gamma'=(Z\cap \partial C_{U,T_1})\cup (\Gamma\cap C_{U,T_1})$ with starting time $T_1$.
\end{prop}
\begin{proof}
	We argue by contradiction. Suppose $Z'$ was not a weak $X$-flow generated by $\Gamma'$ with starting time $T_1$. Then there is a smooth compact barrier 
	\[
	t\in [a,b] \mapsto K(t)
	\]
	with $b>a\geq T_1$ such that $K(t)\cap Z'(t)=\emptyset$ for all $t\in [a,b)$, but there is $p\in K(b)\cap (Z'(b)\setminus \Gamma'(b))$ such that either $p\in\mathrm{int}(K(b))$ or $p\in\partial K(b)$ with $\Phi_K^X(p,b)<0$. By \cref{ContCor}, as $p\in Z'(b)\setminus\Gamma'(b)\subseteq Z(b)\setminus\Gamma(b)$ there is a sequence of points $p_i\in Z(t_i)\setminus\Gamma(t_i)$ with $t_i<b$ and $(p_i,t_i)\to (p,b)$. If $p \in \mathrm{int}(K(b))$, then $K(t_i)\cap Z'(t_i)\neq\emptyset$ for large $i$, which is a contradiction and so we must have $p \in \partial K(b)$ with $\Phi_K^X(p,b)<0$. By adapting \cite[Lemma 7]{HershkovitsWhite} to the $X$-flow setting (see Section 11 of \cite{HershkovitsWhite}), as $p\in\mathrm{int}(U)\cap\partial K(b)$ there is $\hat{a}\in [a,b)$ and a smooth compact barrier in $\mathrm{int}(U)$
	\[
	t\in [\hat{a},b]\mapsto \hat{K}(t)
	\]
	such that $\hat{K}(t)\subseteq\mathrm{int}(U)\cap \mathrm{int}(K(t))$ for $t\in [\hat{a},b)$, $\hat{K}(b)\cap \partial K(b)=\set{p}$ and $\Phi^X_{\hat{K}}(p,b)=\Phi_K^X(p,b)<0$. Thus, $\hat{K}$ is a smooth compact barrier in $N$ so that $\hat{K}(t)\cap Z(t)=\emptyset$ for $t\in [\hat{a},b)$ and $\Phi_{\hat{K}}^X(p,b)<0$ for some $p\in (Z(b)\setminus\Gamma(b))\cap \partial\hat{K}(b)$. This contradicts $Z$ being a weak $X$-flow generated by $\Gamma$ with starting time $T_0$ and so completes the proof.    
\end{proof} 
We conclude this subsection with the existence of certain weak $X$-flows that locally ``push" in a piece of a smooth and strictly $X$-mean-convex compact set.
\begin{lem}\label{XMCBoundaryWeakFlowLem}
	Let $U\subseteq N$ be a smooth and  strictly $X$-mean-convex compact set.   There is an $\eps_0=\eps_0(U, N,X)>0$ so that for any $T_1\in \Real$,  $p\in \partial U$ and $\eps\in (0, \eps_0)$, there is a closed subset $W\subseteq N\times [T_1, T_1+\eps]$ so that:
	\begin{enumerate}
		\item $W(T_1)=\bar{B}_\eps(p)\cap \partial U$;
		\item $\Gamma=W\cap C_{\partial U, T_1} = (\bar{B}_\eps(p)\cap \partial U)\times [T_1, T_1+\eps]$;
		\item For $t\in [T_1,T_1+ \eps]$, $B_{100^{-1} \eps(t-T_1)}(p)\cap U \subseteq W(t)$;
		\item $W$ is a weak $X$-flow generated by $\Gamma$ with starting time $T_1$ .
	\end{enumerate}
\end{lem}
\begin{proof}
	Without loss of generality we may take $T_1=0$. Pick $\eps_0>0$ small enough so that for any $p\in \partial U$, the injectivity radius of $N$ at $p$ is bigger than $\eps_0$ and $\bar{B}_{\eps_0}(p)\cap \partial U$ is a small graph over $T_p\partial U$ (within normal coordinates around $p$).  Moreover, as $U$ is strictly $X$-mean-convex, there is a uniform lower bound $h_0$ on the scalar $X$-mean-curvature of $U$ (with respect to the outward unit normal), and we choose $\eps_0 \le 100^{-1} h_0$.
	
	Fix $p\in \partial U$ and $\eps\in (0, \eps_0)$ and let $\mathbf{n}$ be the inward unit normal to $\partial U$ defined on $B'=\bar{B}_{\eps}(p)\cap \partial U$. Pick a smooth non-negative and compactly supported function,  $\phi $,  on $B'$ with the property that $\Vert \phi \Vert_{C^2(B')}\leq 10^{-1} \eps$ and on $B_{100^{-1} \eps}(p)\cap B'$, $\phi(q')\geq  100^{-1}\eps$.
	For $t\in [0, \eps]$, set
	$$
	W(t)=\set{\exp_N( s \phi( q') \mathbf{n}(q')) \mid q'\in B', s\in [0, t]}
	$$
	The construction immediately ensures that $W$ satisfies (1)-(3) with $T_1=0$.  
	
	The construction also implies that, for any point $q' \in B'$ with $\phi(q') > 0$
	\begin{align*}
	q = \exp_N( t \phi( q') \nN(q'))\in \partial W(t)\cap \mathrm{int}(U).
	\end{align*} 
	Thus, up to shrinking $\eps_0$, the unit normal at $q$ pointing into $W(t)$ is compatible with the outward normal to $U$ at $q'$, $-\mathbf{n}(q'),$ in the sense that, after parallel transporting along the minimizing geodesic connecting $q$ and $q'$, the two differ by less than $\frac{1}{10}$.  In particular, the $X$-mean-curvature of $\partial W(t)$ with respect to this unit normal, $H_W^X(q,t)$, satisfies
	$$
	H_W^X(q,t)\geq h_0- c_1 \eps t \ge (100 - c_1\eps) \eps
	$$
	where $c_1>0$ depends on $\partial U$ and $X$.  Hence, up to shrinking $\eps_0$, $H_W^X(q,t)\geq 2\eps$.  
	
	To check that $W$ is a weak $X$-flow generated by $\Gamma$, it  suffices to consider a smooth compact barrier $K\colon [a,b]\to N$ with $K(t)\cap W(t)=\emptyset$ for $t\in [a,b)$ so that there is a $q\in K(b)\cap \mathrm{int}(U)\cap W(b)$ where $b\in (0, \eps]$ .  Clearly,  $q=\exp_N(b\phi(q') \mathbf{n}(q'))$ for $q'\in B'$ satisfying $\phi(q')>0$.  One readily checks that the normal velocity of $\partial W$ in the direction out of $W$, i.e., into $K$, must satisfy
	$$
	v_W(q,b)\leq 2\phi(q')\leq \eps.
	$$
	Combining this with the lower bound on $H_W^X(q,b)$ from above, the relationship between $K$ and $W$ forces
	$$
	\Phi_K^X(q,b) =v_K(q,b) -H_K^X(q,b)\geq-v_W(q,b)+H_W^X(q,b) \geq \eps>0
	$$
	This is what is required in the definition of weak $X$-flow and so  (4) holds.
\end{proof}

\subsection{Biggest $X$-flows generated by a spacetime set} \label{BiggestFlowSec}
In what follows we discuss a notion of biggest $X$-flow generated by a spacetime set. For classical mean curvature flow there are several equivalent definitions for biggest flow -- see, e.g., \cite{IlmanenIndiana}, \cite{IlmanenPSPM}, \cite{IlmanenMAMS}, \cite{WhiteTopology}, \cite{HershkovitsWhite}. We adopt a constructive definition inspired by \cite{HershkovitsWhite}.
\begin{thm} \label{ExistBiggestFlow}
	Given a closed set $\Gamma\subseteq C_{N,T_0}$, there exists a unique weak $X$-flow $F^X(\Gamma;T_0)$ generated by $\Gamma$ with starting time $T_0$ such that, for any closed set $\Gamma'\subseteq F^X(\Gamma;T_0) \cap C_{N,T'_0}$ with $T'_0\geq T_0$, all weak $X$-flows $Z'$ generated by $\Gamma'$ with starting time $T'_0$ are contained in $F^X(\Gamma;T_0)$. Moreover, if $\hat{\Gamma}\subseteq C_{N,T_0}$ is a closed subset with $\hat{\Gamma}|_{[T_0,T_1]}=\Gamma|_{[T_0,T_1]}$, then 
	\[
	F^X(\hat{\Gamma};T_0)|_{[T_0,T_1]}=F^X(\Gamma;T_0)|_{[T_0,T_1]}.
	\]
\end{thm}
We call $F^X(\Gamma;T_0)$ the \emph{biggest $X$-flow generated by $\Gamma$ with starting time $T_0$}. We will omit the vector field $X$ when it is the zero vector field. The proof of \cref{ExistBiggestFlow} is again a minor modification of the proof of Theorem 19 in \cite{HershkovitsWhite} -- see \cref{ProveExistBiggestFlowSec} for details. 

We prove the following semigroup property for biggest flows. Let
\[
F_t^X(\Gamma;T_0)=F^X(\Gamma; T_0)(t)=\set{x\in N \mid (x,t)\in F^X(\Gamma; T_0)}\subseteq N.
\]
\begin{prop} \label{semigroup}
	Given a closed set $\Gamma\subseteq C_{N,T_0}$ and $T\geq T_0$, if 
	\[
	\Gamma^T=(F^X_T(\Gamma;T_0)\times\set{T})\cup \Gamma|_{[T,\infty)}
	\]
	then, for all $h\geq 0$,
	\[
	F^X_{T+h}(\Gamma;T_0)=F^X_{T+h}(\Gamma^T;T).
	\]
\end{prop}
\begin{proof}
	By \cref{ExistBiggestFlow}, as $\Gamma^T\subseteq F^X(\Gamma;T_0)$ and $T\geq T_0$, one has $F^X(\Gamma^T;T)\subseteq F^X(\Gamma;T_0)$ which immediately implies $F^X_{T+h}(\Gamma^T;T)\subseteq F^X_{T+h}(\Gamma;T_0)$ for all $h\geq 0$. Likewise, as $F^X(\Gamma;T_0)|_{[T,\infty)}$ is a weak $X$-flow generated by $\Gamma^T$ with starting time $T$, appealing to \cref{ExistBiggestFlow} again gives $F^X(\Gamma;T_0)|_{[T,\infty)}\subseteq F^X(\Gamma^T;T)$ and so $F^X_{T+h}(\Gamma;T_0)\subseteq F^X_{T+h}(\Gamma^T;T)$ for all $h\geq 0$.
\end{proof}

\begin{prop}\label{CylinderFlow}
Suppose $(N,X)$ is tame and $U\subseteq N$ is a smooth and strictly $X$-mean-convex compact set. Then one has
	\begin{enumerate}
		\item $F^X(\partial C_{U,T_0}; T_0)=C_{U, T_0}$;
		\item If $\Gamma\subseteq \partial C_{U,T_0}$ is closed, $F^X_t(\Gamma; T_0)\cap \partial U=\Gamma(t)$ for all $t>T_0$;
		\item If $\Gamma\subseteq \partial C_{U,T_0}$ is closed, then
		$$
		\mathrm{int}_{\partial C_{U, T_0}}(\Gamma)=\partial C_{U, T_0}\cap \mathrm{int}_{C_{U,T_0}} (F^X(\Gamma; T_0)).
		$$
	\end{enumerate}
\end{prop}
\begin{proof}
	To prove Item (1), we observe that it follows directly from the definition that $C_{U,T_0}$ is a weak $X$-flow generated by $\partial C_{U,T_0}$ with starting time $T_0$. In particular, $C_{U,T_0}\subseteq F^X(\partial C_{U,T_0}; T_0)$.  To show the reverse inclusion, first note that, as $U$ is compact and strictly $X$-mean-convex, there is $\eps_0 = \eps_0(U)>0$ and a family $\{V_\eps\}_{0<\eps<\eps_0}$ of smooth and  strictly $X$-mean-convex compact sets in $N$ satisfying $U \subseteq V_{\eps'}\subseteq  \mathrm{int}(V_\eps) \subseteq \mathcal{T}_{\eps_0}(U)$ for $0<\eps'<\eps<\eps_0$, $\bigcup_{0<\eps<\eps_0} \partial V_\eps=\mathcal{T}_{\eps_0}(U)\setminus U$ and, as $\eps\to 0^+$, $\partial V_\eps \to \partial U$ in $C^2_{loc}(N)$. Here $\mathcal{T}_{\eps_0}(U)$ is the $\eps_0$-tubular neighborhood of $U\subseteq N$. 
	
	Let 
	\[
	s=\inf\set{t\geq T_0 \mid F_t^X(\partial C_{U,T_0};T_0)\setminus U\neq\emptyset}.
	\]
	Suppose $s<\infty$. Define $\tau\colon N \to \mathbb{R}$ by $\tau(x)=0$ if $x\in U$,  $\tau(x)=\eps$ if $x\in\partial V_\eps$ for some $\eps\in (0,\eps_0)$, and $\tau(x)=\eps_0$ otherwise. By \cref{CompactFlowProp}, for any $T\geq T_0$, $F^X(\partial C_{U,T_0};T_0)|_{[T_0,T]}$ is compact. Moreover, 
	\[
	\lim_{t\to s^+} \max\set{\tau(x) \mid x\in F_t^X(\partial C_{U,T_0};T_0)\setminus U}=0.
	\]
	Otherwise, there is a sequence with subsequential limit $p\in F^X_s(\partial C_{U,T_0};T_0)\setminus U$.  By \cref{spherical-barrier}, the family of shrinking spheres centered at $p$ with sufficiently small radii is a strong barrier which is disjoint from $(F^X_{s'}(\partial C_{U,T_0};T_0)\times\{s'\})\cup \partial C_{U,T_0}$ for some $s'\in [T_0,s)$, but intersects $F^X(\partial C_{U,T_0};T_0)$ at $(p,s)$, giving a contradiction. Thus, there is $t_1>s$ and $\eps_1\in (0,\eps_0)$ such that $F^X_{t_1}(\partial C_{U,T_0};T_0)\subseteq V_{\eps_1}$, $F^X_{t_1}(\partial C_{U,T_0};T_0)\cap\partial V_{\eps_1}\neq\emptyset$, and $F^X_t(\partial C_{U,T_0};T_0)\subseteq \mathrm{int}(V_{\eps_1})$ for $s\leq t<t_1$. 
	
	Let $B_R(x_0)\subseteq N$ be a large ball containing $V_{\eps_1}$. For $t\in [s,t_1]$, let $K(t)=\bar{B}_R(x_0)\setminus \mathrm{int}(V_{\eps_1})$. As $V_{\eps_1}$ is strictly $X$-mean-convex, $t\in [s,t_1]\mapsto K(t)$ is a compact smooth barrier such that $K\cap \partial C_{U,T_0}=\emptyset$ and $K(t)\cap F_t^X(\partial C_{U,T_0};T_0)=\emptyset$ for $t\in [s,t_1)$, but there is a $p_1\in \partial K(t_1)\cap F^X_{t_1}(\partial C_{U,T_0};T_0)$ so $\Phi_K^X(p_1,t_1)<0$. This contradicts $F^X(\partial C_{U,T_0};T_0)$ being a weak $X$-flow generated by $\partial C_{U,T_0}$ with starting time $T_0$. That is, $s=\infty$. Thus, $F^X_t(\partial C_{U,T_0};T_0)\subseteq U$ for all $t\geq T_0$,  and so $F^X(\partial C_{U,T_0};T_0)\subseteq C_{U,T_0}$ proving Item (1).
	
	To prove Item (2), first note that, by the hypothesis, $\Gamma(t)\subseteq F^X_t(\Gamma;T_0)\cap \partial U$ for $t>T_0$. For the reverse inclusion, assume for a contradiction that there is $t'_1>T_0$ and $p'_1\in (F_{t'_1}^X(\Gamma;T_0)\cap\partial U)\setminus\Gamma(t'_1)$. By the adaption of Lemma 7 in \cite{HershkovitsWhite} to $X$-flow (see Section 11 of \cite{HershkovitsWhite}), as $U$ is strictly $X$-mean-convex, there is a compact smooth barrier $t\in [a',t'_1]\mapsto K'(t)$ with $a'\geq T_0$ such that $K'(t)\subseteq N\setminus U$ for $t\in [a',t'_1)$, $K'(t_1')\cap \partial U=\{p'_1\}$ and $\Phi^X_{K'}(p'_1,t'_1)<0$. By \cref{ExistBiggestFlow} and Item (1), our hypothesis ensures $F^X(\Gamma;T_0)\subseteq F^X(\partial C_{U,T_0};T_0)=C_{U,T_0}$. Thus, $K'(t)\cap F_t^X(\Gamma;T_0)=\emptyset$ for $t\in [a',t'_1)$ and $p'_1\in \partial K'(t'_1)\cap F^X_{t'_1}(\Gamma;T_0)$ with $\Phi^X_{K'}(p'_1,t'_1)<0$. This contradicts $F^X(\Gamma;T_0)$ being a weak $X$-flow generated by $\Gamma$. Hence, $F^X_t(\Gamma;T_0)\cap\partial U\subseteq\Gamma(t)$ for all $t>T_0$ showing Item (2).   
	
	To establish Item (3), we need only show
	$$
	\mathrm{int}_{\partial C_{U,T_0}}(\Gamma)\subseteq \mathrm{int}_{C_{U,T_0}}(F^X(\Gamma; T_0)).
	$$
	This suffices as when $(p,t)\in \partial C_{U,T_0}\cap \mathrm{int}_{C_{U,T_0}}(F^X(\Gamma; T_0))$ has $t=T_0$, by construction $p\in \Gamma(T_0)$ and, when  $t>T_0$, Item (2) implies $p\in \Gamma(t)$.  In either case, basic topology implies $(p,t)\in \mathrm{int}_{\partial C_{U,T_0}}(\Gamma)$.
	
	Now suppose $(p,t)\in  \mathrm{int}_{\partial C_{U,T_0}}(\Gamma)$.  We treat three cases.  First suppose $t=T_0$ and $p\in\mathrm{int}(U)$.
	In this case, there is a small radius $\eps>0$ so $\bar{B}_\eps(p)\subseteq \Gamma(T_0)$.  As strong $X$-flow barriers are also weak $X$-flows it follows from  \cref{spherical-barrier} and \cref{ExistBiggestFlow} that $(p,t)=(p,T_0)\in \mathrm{int}_{C_{U,T_0}}(F^X(\Gamma;T_0))$. 
	Likewise, if $t=T_0$ and  $p\in \partial U$, by Propositions \ref{spherical-barrier} and \ref{WeakFlowLocalization} there exist sufficiently small $\eps,\delta>0$ and sufficiently large $c$ such that
	\[
	t\in [T_0,T_0+\eps]\mapsto \Bar{B}_{\sqrt{\delta^2-c(t-T_0)}}(x)\cap U=\mathcal{B}(t)
	\]
	is a weak $X$-flow with starting time $T_0$ generated by 
	\[
	\partial C_{U,T_0}\cap \bigcup_{t\in [T_0,T_0+\eps]} \mathcal{B}(t)\times\set{t} \subseteq \Gamma.
	\]
	Thus, it follows from \cref{ExistBiggestFlow} that $(p,t)=(p,T_0)\in \mathrm{int}_{C_{U,T_0}}(F^X(\Gamma;T_0))$.
	
	Finally, suppose $t>T_0$.  Using $U$, let $\eps_0>0$ be given by Lemma \ref{XMCBoundaryWeakFlowLem}. As $(p,t)$ is an interior point we may choose $\eps\in (0,\frac{1}{2}\eps_0)$ so that $t-\eps>T_0$ and
	$$ \Gamma'=\partial C_{ U,T_0}\cap \left(\bar{B}_\eps(p)\times[t-\eps, t+\eps]\right) \subseteq \Gamma.
	$$
	Set $T_1=t-\eps$.  By \cref{XMCBoundaryWeakFlowLem} and time translation, there is a weak $X$-flow, $W'$ generated by $\Gamma'$ with starting time $T_0$. Moreover, $(p,t)$ is an interior point of $W'\cap C_{U,T_0}$.  Hence, by \cref{ExistBiggestFlow}, $W'\subseteq F^X(\Gamma; T_0)$ and so $(p,t)$ is an interior point of  $F^X(\Gamma; T_0)$.
\end{proof}

We show that the boundary of the biggest $X$-flow generated by appropriate spacetime sets is itself a weak $X$-flow of a related spacetime set. 
\begin{prop} \label{boundary-flow}
	Suppose $(N,X)$ is tame and $U\subseteq N$ is a smooth and  strictly $X$-mean-convex compact  set. If $\Gamma\subseteq \partial C_{U,T_0}$ is closed 
	then $\partial_{C_{U,T_0}} F^X(\Gamma;T_0)$ is a weak $X$-flow generated by $\partial_{\partial C_{U,T_0}} \Gamma$ with starting time $T_0$.
\end{prop}
\begin{proof}
	Set $Z=F^X(\Gamma;T_0)$, $Y=\partial_{C_{U,T_0}} F^X(\Gamma;T_0)$ and $\Upsilon=\partial_{\partial C_{U,T_0}} \Gamma$.  We first establish 
	\begin{equation}\label{UpsilonYEqn}
	\Upsilon=Y\cap \partial C_{U, T_0}.
	\end{equation}
	Indeed, it follows from the hypotheses that $\Gamma\subseteq Z\cap \partial C_{U, T_0}$. Conversely, suppose $(p,t)\in Z\cap \partial C_{U,T_0}$. When $t=t_0$,  $p\in Z(T_0)=\Gamma(T_0)$, while when $t>T_0$, $p\in \Gamma(t)$ by Item (2) of Proposition \ref{CylinderFlow}. Hence, $\Gamma=Z\cap \partial C_{U, T_0}$.  Finally, as both $Z$ and $\Gamma$ are closed, it follows from Item (3) of  Proposition \ref{CylinderFlow} that
	$$
	\Upsilon =\Gamma\setminus \mathrm{int}_{\partial C_{U, T_0}} (\Gamma)=(Z\setminus\mathrm{int}_{C_{U,T_0}} (Z))\cap \partial C_{U,T_0}=Y\cap \partial C_{U, T_0}.
	$$
	
	By \cref{equivalent-defn}, it remains only to show that if $K\colon t\in [a,b]\mapsto K(t)$ is a strong barrier with $a\geq T_0$ such that $K(a)\cap Y(a)=\emptyset$ and $K(t)\cap\Upsilon(t)=\emptyset$ for all $t\in [a,b]$, then $K(t)\cap Y(t)=\emptyset$ for all $t\in [a,b]$. We argue by contradiction. Without loss of generality we may assume $b$ is the first time that $Y$ intersects $K$. We claim $\partial K(b)\cap Y(b)\neq\emptyset$. Indeed, the hypotheses and \cref{CylinderFlow} ensure $Y(b)\setminus\Upsilon(b)\subseteq Z(b)\setminus\Gamma(b) \subseteq \mathrm{int}(U)$. Thus, by \cref{ContCor} applied to $u(\cdot, t)=\dist(\cdot, Z(t))$ and \cref{spherical-barrier}, any point in $Y(b)\setminus\Upsilon(b)$ is a subsequential limit of points in $Y(t_i)$ with $t_i\uparrow b$. As $K(t)\cap Y(b)=\emptyset$ for $t\in [a,b)$, one has $\mathrm{int}(K(b))\cap Y(b)=\emptyset$ and so $\partial K(b)\cap Y(b)\neq\emptyset$.
	
	Let $p\in\partial K(b)\cap Y(b)$. As observed before $p\in\mathrm{int}(U)$. By the adaption of Theorem 8 in \cite{HershkovitsWhite} to $X$-flow (see Section 11 of \cite{HershkovitsWhite}) there exists a strong barrier $\hat{K}\colon t\in [\hat{a},b]\mapsto \hat{K}(t)\subseteq\mathrm{int}(U)$ with $\hat{a}\in (a,b)$ such that $\hat{K}(t)\subseteq\mathrm{int}(K(t))$ for $t\in [\hat{a},b)$ and such that $\hat{K}(b)\cap \partial K(b)=\{p\}$. Furthermore, we may assume $\hat{K}$ is connected as otherwise we can replace $\hat{K}$ by its connected component containing $(p,b)$. As $\hat{K}(t)\cap Y(t)=\emptyset$ for $t\in [\hat{a},b)$, either (i) $\hat{K}|_{[\hat{a},b)}\cap Z=\emptyset$ or (ii) $\hat{K}|_{[\hat{a},b)}\subseteq\mathrm{int}(Z)$. In Case (i), by \cref{CylinderFlow} one has $\hat{K}(t)\cap \Gamma(t)=\emptyset$ for all $t\in [\hat{a},b]$ and so $\hat{K}(b)\cap Y(b)\subseteq \hat{K}(b)\cap Z(b)=\emptyset$, a contradiction. In Case (ii), let $\tilde{K}$ be a prolongation in time of $\hat{K}$ and fattening in space of $\hat{K}$ so that $\tilde{K}$ is also a strong barrier with $\tilde{K} (\hat{a})\subseteq Z(\hat{a})$. By \cref{ExistBiggestFlow}, as strong barriers are automatically weak flows, $\tilde{K}\subseteq Z$ and so $\hat{K}\subseteq\mathrm{int}(Z)$. In particular, $\hat{K}(b)\cap Y(b)=\emptyset$. This is a contradiction and completes the proof.
\end{proof}

\section{Monotone $X$-flows and $X$-mean-convexity} \label{MonotoneXMeanConvexSec}

\subsection{Monotone $X$-flows generated by spacetime sets} \label{MonotoneFlowSec}
The boundary of a smooth and strictly $X$-mean-convex compact set moves inwards under the smooth $X$-flow. In this subsection, we define a notion in terms of weak $X$-flows that captures the phenomenon. First of all, we formalize what it means for a set to ``move inwards". 
\begin{defn}
	\label{monotone-def}
	Fix a set $U\subseteq N$ and $T_0\in\mathbb{R}$. A set $Z\subseteq C_{U,T_0}$ is \emph{monotone (in $C_{U,T_0}$)} if, for $T_0\leq t_1\leq t_2$,
	\[
	Z(t_2)\subseteq Z(t_1).
	\]
	It is \emph{strictly monotone (in $C_{U,T_0}$)} if, for $T_0\leq t_1<t_2$,
	\[
	Z(t_2)\subseteq\mathrm{int}_U(Z(t_1)).
	\]
	Here we use the convention that the interior of the empty set is the empty set and omit the cylinder $C_{U, T_0}$ when it is clear from context.
\end{defn}
We record some basic topological properties of strictly monotone sets that will be used in the next section. 
\begin{lem}
	\label{strictly-monotone-lem}
	Fix a closed set $U\subseteq N$ and $T_0\in\mathbb{R}$. Suppose $Z\subseteq C_{U,T_0}$ is closed and strictly monotone in $ C_{U,T_0}$. The following hold:
	\begin{enumerate}
		\item $Z|_{(T_0,\infty)}=\mathrm{cl}(\mathrm{int}_{C_{U,T_0}}(Z))|_{(T_0,\infty)}$;
		\item Suppose $V\subseteq U$ and $\Gamma\subseteq Z$ are closed. If $\Gamma\subseteq C_{\partial V,T_0}$ is strictly monotone in $C_{\partial V,T_0}$, then $\mathrm{int}_{C_{\partial V,T_0}}(\Gamma)|_{(T_0,\infty)}\subseteq\mathrm{int}_{C_{U,T_0}}(Z)$.
	\end{enumerate}
\end{lem}
\begin{proof}
	As $Z$ is closed and $\mathrm{int}_{C_{U,T_0}}(Z)\subseteq Z$ we have 
	$ \mathrm{cl}(\mathrm{int}_{C_{U,T_0}}(Z))\subseteq Z.$
	Hence,
	\[
	\mathrm{cl}(\mathrm{int}_{C_{U,T_0}}(Z))|_{(T_0,\infty)} \subseteq Z|_{(T_0, \infty)}.
	\]
	For the reverse inclusion, suppose $P=(p,T_1)\in \partial_{C_{U,T_0}} Z$ with $T_1>T_0$. By the strict monotonicity of $Z$, for each $T_0<t<T_1$, there is a subset $W_t$ with $p\in W_t\subseteq Z(t)$ so that $W_t$ is open in $U$. The monotonicity ensures that
	\[
	\hat{W}_t=[T_0, t)\times W_t \subseteq \mathrm{int}_{C_{U,T_0}}(Z)
	\]
	is an open subset of $C_{U,T_0}$. Hence, 
	\[
	\hat{W}=\bigcup_{T_0<t<T_1} \hat{W}_t\subseteq \mathrm{int}_{C_{U,T_0}}(Z)
	\]
	is an open subset of $C_{U,T_0}$ and $P\in \mathrm{cl}(\hat{W})$.  It follows that $P\in \mathrm{cl}(\mathrm{int}_{C_{U,T_0}}(Z))$ and so
	\[
	Z|_{(T_0, \infty)}\subseteq \mathrm{cl}(\mathrm{int}_{C_{U,T_0}}(Z))|_{(T_0,\infty)},
	\]
	which completes the proof of Item (1). 
	
	Similarly, suppose $(p,T_1)\in\mathrm{int}_{C_{\partial V,T_0}}(\Gamma)$ with $T_1>T_0$. There is an $\eps>0$ so that $p\in \Gamma(T_1+2\eps)\subseteq Z(T_1+2\eps)$. Strict monotonicity of $Z$ implies $p\in \mathrm{int}_U(Z(T_1+\eps))$ and so there is an open subset $W$ of $U$ such that $p\in W\subseteq Z(T_1+\eps)$. Monotonicity ensures
	\[
	(p,T_1)\in [T_0,T_1+\eps)\times W\subseteq \mathrm{int}_{C_{U,T_0}}(Z).
	\]
	As $(p,T_1)$ is arbitrary, this proves Item (2). 
\end{proof}

\begin{defn}
	\label{weak-mean-convexity-def}
	Fix a closed set $U\subseteq N$, a $T_0\in \Real$ and a closed set $\Gamma\subseteq\partial C_{U,T_0}$.  A weak $X$-flow, $Z\subseteq C_{U,T_0}$, generated by $\Gamma$ with starting time $T_0$ is \emph{(strictly) monotone} if $\Gamma\cap C_{\partial U,T_0}$ is (strictly) monotone in $C_{\partial U,T_0}$ and $Z$ is a (strictly) monotone in $C_{U,T_0}$.
\end{defn}

We record the following topological property for strictly monotone biggest flows.
\begin{lem}
	\label{BoundaryPropMonotoneLem}
	Fix a smooth closed set $U\subseteq N$ and $T_0\in\mathbb{R}$. Suppose $\Gamma\subseteq\partial C_{U,T_0}$ is closed and $Z=F^X(\Gamma; T_0)\subseteq C_{U,T_0}$ is strictly monotone.  If $\Gamma(T_0)=\mathrm{cl}(\mathrm{int}_{U}(\Gamma(T_0))$, then 
	\[
	Z=\mathrm{cl}(\mathrm{int}_{C_{U,T_0}}(Z)).  
	\]
\end{lem}
\begin{proof}
	For $p\in Z(T_0)= \Gamma(T_0)= \mathrm{cl}(\mathrm{int}_U(\Gamma(T_0)))$, there is a sequence of points $p_i\in \mathrm{int}(U)\cap \Gamma(T_0)$ converging to $p$.  Using the sphere barriers of \cref{spherical-barrier} as in the proof of \cref{CylinderFlow}, one sees that $(p_i,T_0)\in \mathrm{int}_{C_{U,T_0}}(Z)$ and so $(p,T_0)\in \mathrm{cl}(\mathrm{int}_{C_{U,T_0}}(Z)$.
	Combined with Item (1) of \cref{strictly-monotone-lem}, it follows that $Z\subseteq\mathrm{cl}(\mathrm{int}_{C_{U,T_0}}(Z))$. The reverse inclusion follows as $Z$ is closed and proves the claim.    
\end{proof}

\subsection{$X$-mean-convex spacetime sets and their flows } \label{XMCSec}
We introduce a notion of $X$-mean-convexity for, not necessarily regular, subsets of spacetime and prove some basic properties of the biggest flows generated by such sets. This should be thought of as a spacetime and low regularity analogue of strict $X$-mean-convexity of smooth sets. For related results on $X$-mean-convexity of spatial  and spacetime sets, see \cite{WhiteMCNature}, \cite{HershkovitsWhite} and \cite{WhiteSubsequent}.

\begin{defn}
	\label{MCSpaceTimeDef}
	Fix a closed set $U\subseteq N$ and suppose $\Gamma\subseteq \partial C_{U,T_0}$ is closed. The set $\Gamma$ is \emph{spacetime $X$-mean-convex} if
	\begin{enumerate}
		\item $\Gamma\cap C_{\partial U, T_0}$ is a strictly monotone set in $C_{\partial U,T_0}$; 
		\item There is an $\eps>0$ so that $F_t^X(\Gamma;T_0)\subseteq \mathrm{int}_{U} (\Gamma(T_0))$ for all $T_0<t<T_0+\eps$;
		\item $\mathrm{int}_{\partial C_{U,T_0}}(\Gamma)|_{(T_0,\infty)}\subseteq\mathrm{int}_{C_{U,T_0}}(F^X(\Gamma;T_0))$.
	\end{enumerate}
	When  $U$ is a smooth and strictly $X$-mean-convex compact set, Proposition \ref{CylinderFlow} implies Item (3) holds automatically. For such a $U$, $\Gamma=U\times \set{T_0}$ is spacetime $X$-mean-convex. 
\end{defn}
The main result in this subsection is the following:
\begin{prop}
	\label{MCGenFlowProp}
	Suppose $(N,X)$ is tame.
	Let $U \subseteq N$ be a smooth and strictly $X$-mean-convex compact set and suppose that $\Gamma\subseteq \partial C_{U,T_0}$ is spacetime $X$-mean-convex. The following hold:
	\begin{enumerate}
		\item $F_{t+h}^X(\Gamma; T_0) \subseteq \mathrm{int}_U(F_t^X(\Gamma; T_0))$ for all $t\geq T_0$ and $h>0$;
		\item $\partial_{C_{U,T_0}} F^X(\Gamma;T_0)=F^X(\partial_{\partial C_{U,T_0}} \Gamma; T_0)$.
	\end{enumerate}
	In particular,  the following \emph{partition property} holds
	\begin{enumerate}[resume]
		\item \label{PartitionPropItem} In $\mathrm{int}(U)$ the family $\{\partial_U F^X_t(\Gamma;T_0)\}_{t\in [T_1,T_2)}$ forms a partition of $F^X_{T_1}(\Gamma;T_0)\setminus F^X_{T_2}(\Gamma;T_0)$ for each $T_0<T_1<T_2$.
	\end{enumerate}
\end{prop}
\begin{rem}
	The biggest flow $F^X(\Gamma;T_0)$ of a spacetime $X$-mean-convex set $\Gamma\subseteq\partial C_{U,T_0}$ is called an \emph{$X$-mean-convex flow} and is a strictly monotone flow.  
\end{rem}
The proof below is in part inspired by the proof of \cite[Theorem 32]{HershkovitsWhite}.
\begin{proof}[Proof of \cref{MCGenFlowProp}]
	In this proof we assume, without loss of generality, that $T_0=0$ and write $F^X(\Gamma)=F^X(\Gamma;0)$. For the first item, let $h\in (0,\eps)$, where $\eps$ is from Item (2) of \cref{MCSpaceTimeDef}. Since $\Gamma$ is spacetime $X$-mean-convex, the inclusion holds trivially by \cref{MCSpaceTimeDef} for $t=0$. By \cref{semigroup} we have, for $t\geq h$,
	\[
	F_{t}^X(\Gamma)=F^X_{t}(\Gamma^h;h)
	\]
	where 
	\[
	\Gamma^h=(F^X_h(\Gamma)\times\set{h})\cup \Gamma|_{[h,\infty)}.
	\]
	Thus, if $\tilde{\Gamma}^h$ is the translation in time of $\Gamma^h$ given by
	\[
	\tilde{\Gamma}^h=(F^X_h(\Gamma)\times\set{0})\cup \left(\bigcup_{t\ge h} (\Gamma(t)\times\set{t-h})\right)
	\]
	then 
	\[
	F^X(\tilde{\Gamma}^h)=\bigcup_{t\ge h} (F^X_t(\Gamma^h;h)\times\set{t-h})=\bigcup_{t\ge h} (F_t^X(\Gamma)\times\set{t-h}).
	\]
	As $\Gamma$ is strictly monotone, $\tilde{\Gamma}^h\subseteq\Gamma\subseteq F^X(\Gamma)$. Therefore, by \cref{ExistBiggestFlow}, 
	\begin{equation}
	\label{weak-inclusion}
	F^X(\tilde{\Gamma}^h)\subseteq F^X(\Gamma).
	\end{equation}
	
	By \cref{CylinderFlow} and definition of $\tilde{\Gamma}^h$, one has for all $t>0$
	\[
	F^X_t(\tilde{\Gamma}^h)\cap \partial U=\tilde{\Gamma}^h(t)=\Gamma(t+h).
	\]
	Combined with the strict monotonicity of $\Gamma$ and of $F^X(\Gamma)$ at $t=0$, it follows that
	\begin{equation}
	\label{eq-4-2}
	\tilde{\Gamma}^h\subseteq \mathrm{int}_{\partial C_{U,0}} (\Gamma)
	\end{equation}
	which implies
	\[
	F^X(\tilde{\Gamma}^h)\cap \partial_{\partial C_{U,0}}\Gamma=\tilde{\Gamma}^h \cap \partial_{\partial C_{U,0}}\Gamma\subseteq\mathrm{int}_{\partial C_{U,0}}(\Gamma) \cap \partial_{\partial C_{U,0}}\Gamma=\emptyset.
	\]
	Likewise, by \cref{CylinderFlow} and \eqref{eq-4-2}, 
	\[
	\tilde{\Gamma}^h\cap F^X(\partial_{\partial C_{U,0}}\Gamma)\subseteq \tilde{\Gamma}^h\cap\partial_{\partial C_{U,0}} \Gamma \subseteq \mathrm{int}_{\partial C_{U,0}} (\Gamma) \cap \partial_{\partial C_{U,0}}\Gamma=\emptyset.
	\]
	Thus, by \cref{AvoidanceThm}, 
	\begin{equation} \label{disjoint-boundary}
	F^X(\tilde{\Gamma}^h)\cap F^X(\partial_{\partial C_{U,0}}\Gamma)=\emptyset.
	\end{equation}
	
	As $U$ is a smooth and strictly $X$-mean-convex compact set,  \cref{boundary-flow} ensures $\partial_{C_{U,0}} F^X(\Gamma)$ is a weak $X$-flow generated by $\partial_{\partial C_{U,0}}\Gamma$ and so, by \cref{ExistBiggestFlow}, 
	\begin{equation} \label{inclusion-boundary}
	\partial_{C_{U,0}} F^X(\Gamma)\subseteq F^X(\partial_{\partial C_{U,0}}\Gamma). 
	\end{equation}
	Thus, combined with \eqref{disjoint-boundary}, it follows that 
	\[
	F^X(\tilde{\Gamma}^h)\cap \partial_{C_{U,0}} F^X(\Gamma)=\emptyset
	\]
	and so, by \eqref{weak-inclusion},
	\[
	F^X(\tilde{\Gamma}^h)\subseteq\mathrm{int}_{C_{U,0}}(F^X(\Gamma)).
	\]
	It is a basic topological fact that $(p,t)\in \mathrm{int}_{C_{U,0}}(F^X(\Gamma))$ implies $p\in \mathrm{int}_U(F_t^X(\Gamma))$ and so we conclude that, for all $t\geq 0$ and $0<h<\eps$,
	\[
	F_{t+h}^X(\Gamma)=F^X_t(\tilde{\Gamma}^h)\subseteq \mathrm{int}_U(F_t^X(\Gamma)).
	\]
	Iterating this inclusion gives the first conclusion.  
	
	For the second item, first observe by \eqref{disjoint-boundary} and definition of $F^X(\tilde{\Gamma}^h)$
	\[
	F^X_{t+h}(\Gamma)\cap F^X_t(\partial_{\partial C_{U,0}}\Gamma)=F^X_t(\tilde{\Gamma}^h)\cap F^X_t(\partial_{\partial C_{U,0}}\Gamma)=\emptyset
	\]
	for all $t\geq 0$ and $0<h<\eps$. By the first item
	\[
	F^X_T(\Gamma) \subseteq F^X_{t+h}(\Gamma) 
	\]
	for all $T\geq t+h$, and thus 
	\begin{equation} \label{disjoint-slice}
	F^X_T(\Gamma)\cap F^X_t(\partial_{\partial C_{U,0}}\Gamma)=\emptyset
	\end{equation}
	for all $T>t$. By \cref{ExistBiggestFlow}, as $\partial_{\partial C_{U,0}}\Gamma\subseteq\Gamma$ one has $F^X_T(\partial_{\partial C_{U,0}}\Gamma)\subseteq F^X_T(\Gamma)$ and so
	\begin{equation} \label{disjoint-boundary-slice}
	F^X_T(\partial_{\partial C_{U,0}}\Gamma)\cap F_t^X(\partial_{\partial C_{U,0}}\Gamma)=\emptyset
	\end{equation}
	for all $T>t$.
	
	Now define the arrival-time function $u\colon \Gamma(0)\to [0, \infty]$ by
	\[
	u(x)=\left\{\begin{array}{ll} t & \mbox{if $x\in F^X_t(\partial_{\partial C_{U,0}}\Gamma)$} \\
	\infty & \mbox{if } x\in \Gamma(0)\setminus \bigcup_{t\geq 0} F^X_t(\partial_{\partial C_{U,0}}\Gamma). \end{array} \right.
	\]
	By \eqref{disjoint-boundary-slice}, the $F^X_t(\partial_{\partial C_{U,0}}\Gamma)$ are disjoint for distinct $t$, so $u$ is well-defined.  If $x\in \Gamma(0)\setminus F_t^X(\Gamma)$, then $(x,0)\in F^X(\Gamma)$, but $(x,t)\not\in F^X(\Gamma)$. This means there is a $\tau\in [0,t)$ so $(x,\tau)\in \partial_{C_{U,0}} F^X(\Gamma)$. Combined with \eqref{inclusion-boundary}, it follows that $x\in F^X_\tau(\partial_{\partial C_{U,0}}\Gamma)$ and so $u(x)<t$. 
	Conversely, by \eqref{disjoint-slice}, if $u(x)<t$ then $F^X_{u(x)}(\partial_{\partial C_{U,0}}\Gamma)$ is disjoint from $F_t^X(\Gamma)$ and so $x\not\in F_t^X(\Gamma)$.  Thus, we have proven that $F_t^X(\Gamma)=\{u\geq t\}$.
	Hence, 
	$$\Gamma(\infty)=\bigcap_{t\geq 0} F_t^X(\Gamma)=\bigcap_{t\geq 0} \set{u\geq t}=\set{u=\infty}$$ 
	is closed.  As $F^X(\partial_{\partial C_{U,0}}\Gamma)$ is also a closed subset of spacetime, $u$ is continuous.

	By Item \eqref{DistItem} of \cref{ContCor} applied to $F^X(\partial_{\partial C_{U,0}}\Gamma)$, every point in $\{u=t\}\setminus (\partial_{\partial C_{U,0}}\Gamma)(t)$ is a limit of points in $\set{u<t}$ when $t>0$. Hence, for $t > 0$,
	\[
	\set{u=t} \setminus (\partial_{\partial C_{U,0}} \Gamma)(t)= \partial_U \set{u \ge t} \setminus (\partial_{\partial C_{U,0}} \Gamma)(t)= \partial_U F^X_t(\Gamma) \setminus (\partial_{\partial C_{U,0}} \Gamma)(t).
	\]
	By general topology, if $p\in\partial_U F^X_t(\Gamma)$ then $(p,t)\in \partial_{C_{U,0}} F^X(\Gamma)$, and thus 
	\[
	F^X(\partial_{\partial C_{U,0}}\Gamma)\subseteq \partial_{C_{U,0}} F^X(\Gamma).
	\]
	Combined with \eqref{inclusion-boundary}, this proves the second claim. 
\end{proof}

We have the following closedness property of $X$-mean-convex flows:
\begin{prop}
	\label{MCLimitProp}
	Suppose $(N,X)$ is tame.  Let $U \subseteq N$ be a smooth and strictly $X$-mean-convex set and suppose $\Gamma, \Gamma_i\subseteq \partial  C_{U,T_0}$, for $i\in \mathbb{N}$, are spacetime $X$-mean-convex. If
	\begin{enumerate}
		\item \label{MCLimitHypo1} $\Gamma_i\subseteq \mathrm{int}_{\partial C_{U,T_0}}(\Gamma_{i+1})$;
		\item \label{MCLimitHypo2} $ \mathrm{int}_{\partial C_{U,T_0}} (\Gamma)= \bigcup_{i=1}^\infty\Gamma_i$,
	\end{enumerate}  
then $ \mathrm{int}_{C_{U,T_0}}(F^X(\Gamma;T_0))=Z_\infty$ where $Z_\infty=\bigcup_{i=1}^\infty F^X(\Gamma_i; T_0).$
	If, in addition, 
	\begin{enumerate}[resume]
		\item $\Gamma(T_0)=\mathrm{cl}(\mathrm{int}_U (\Gamma(T_0)))$,
	\end{enumerate}
 then $  F^X(\Gamma; T_0)=\mathrm{cl}(Z_\infty)$.
\end{prop}
\begin{rem}
	Hypothesis (3) is needed as can be seen by considering two disjoint balls connected by a line segment -- the line segment disappears immediately under the flow. Likewise, the result fails if instead of \cref{MCLimitHypo2}, one assumed that
	\[
	\Gamma =\mathrm{cl}\left( \bigcup_{i=1}^\infty\Gamma_i \right).
	\]
	Indeed, let $X =\mathbf{0}$ and $\mathbb{D} \subseteq \mathbb{R}^2$ be the closed unit disk. Consider $\Gamma_i = \Gamma^+_i \cup \Gamma_i^- \subseteq \mathbb{D}$, where $\Gamma_i^\pm$ are the sets
	\[
	\Gamma^{\pm}_i=\set{(x,y): \pm y\geq i^{-1}} \cap \mathbb{D},\mbox{ so }\Gamma=\mathrm{cl}\left(\bigcup_{i=1}^\infty \Gamma_i\right)=\mathbb{D}.
	\]
	However, as $\mathbf{0}\not\in \Gamma^\pm_i$ and the $\Gamma^\pm_i$ are mean convex, by Proposition \ref{RoundingProp}, for every $t>0$, there is a radius $r>0$ so that $B_r(\mathbf{0})\cap F_t^X(\Gamma_i;0)=\emptyset$ for all $i$. Hence, for $t>0$ small,
	\[
	(0,t)\in F^X(\Gamma; 0)\setminus \mathrm{cl}\left(\bigcup_{i=1}^\infty F^X(\Gamma_i; 0)\right).
	\]
\end{rem}
\begin{proof}[Proof of \cref{MCLimitProp}]
	Without loss of generality we take $T_0=0$ and write $F^X(\Gamma_i)=F^X(\Gamma_i;0)$ and $F^X(\Gamma)=F^X(\Gamma;0)$, omitting the starting time when it is $0$. For $h>0$, let 
	\[
	\tilde{\Gamma}^h=(F^X_h(\Gamma)\times\set{0}) \cup \left(\bigcup_{t\geq h} \Gamma(t)\times\set{t-h}\right).
	\]
	First we show that 
	\begin{equation} \label{interior-eq}
	\mathrm{int}_{C_{U,0}}(F^X(\Gamma))=\bigcup_{h>0} F^X(\tilde{\Gamma}^h).
	\end{equation}
	To see this, observe that by \cref{semigroup}
	\[
	F^X(\Tilde{\Gamma}^h)=\bigcup_{t\geq h} F^X_t(\Gamma)\times\set{t-h}.
	\]
	Thus, if $(p,t)\in\mathrm{int}_{C_{U,0}}(F^X(\Gamma))$, there is an $h>0$ so that 
	\[
	p\in F^X_{t+h}(\Gamma)=F^X_t(\tilde{\Gamma}^h).
	\]
	That is, $(p,t)\in F^X(\tilde{\Gamma}^h)$ and so 
	\[
	\mathrm{int}_{C_{U,0}}(F^X(\Gamma))\subseteq\bigcup_{h>0} F^X(\tilde{\Gamma}^h).
	\]
	
	We next establish the reverse inclusion.  Fix $h>0$.  As $\Gamma$ is spacetime $X$-mean-convex,  it is strictly monotone and, moreover, by \cref{MCGenFlowProp} so is $F^X(\Gamma)$. Hence, $\tilde{\Gamma}^h\subseteq\mathrm{int}_{\partial C_{U,0}}(\Gamma)$ and so, by \cref{ExistBiggestFlow},  $F^X(\tilde{\Gamma}^h)\subseteq F^X(\Gamma)$, and by \cref{CylinderFlow} 
	\[
	F^X(\tilde{\Gamma}^h)\cap \partial_{\partial C_{U,0}}\Gamma=\tilde{\Gamma}^h \cap F^X(\partial_{\partial C_{U,0}}\Gamma)=\tilde{\Gamma}^h\cap\partial_{\partial C_{U,0}}\Gamma=\emptyset.
	\]
	As $\Gamma$ is spacetime $X$-mean-convex, it follows from \cref{MCGenFlowProp} and \cref{AvoidanceThm} that 
	\[
	F^X(\tilde{\Gamma}^h)\cap\partial_{C_{U,0}}F^X(\Gamma)=F^X(\tilde{\Gamma}^h)\cap F^X(\partial_{\partial C_{U,0}}\Gamma)=\emptyset.
	\]
	Hence, $F^X(\tilde{\Gamma}^h)\subseteq\mathrm{int}_{C_{U,0}}(F^X(\Gamma))$ and so
	\[
	\bigcup_{h>0} F^X(\tilde{\Gamma}^h)\subseteq\mathrm{int}_{C_{U,0}}(F^X(\Gamma))
	\]
	since $h>0$ is arbitrary. This proves the claim.
	
	We now establish,  for any $T_1>0$ and $h>0$, the existence of $i_0=i_0(T_1,h)$ so that
	\[
	\tilde{\Gamma}^h|_{[0,T_1]}\subseteq\Gamma_{i_0}.
	\]
	Indeed,  strict  monotonicity of $\Gamma$ and  Hypothesis (2) ensure
	\[
	\tilde{\Gamma}^h|_{[0,T_1]}\subseteq\mathrm{int}_{\partial C_{U,0}}(\Gamma)=\bigcup_{i=1}^\infty \Gamma_i.
	\]
	That is, $\{\mathrm{int}_{\partial C_{U,0}} (\Gamma_i)\}_{i=1}^\infty$ is an open cover of the compact set $\tilde{\Gamma}^h|_{[0,T_1]}$ and so there is a finite subcover. Thus, one may choose $i_0$ sufficiently large so that
	\[
	\tilde{\Gamma}^h|_{[0,T_1]}\subseteq \mathrm{int}_{\partial C_{U,0}} (\Gamma_{i_0})\subseteq \Gamma_{i_0}.
	\] 
	
	Now fix $T_1 > 0$ and $h > 0$, and assume in the following that $i\geq i_0$. Set $\hat{\Gamma}_i^h=\Gamma_i \cap \tilde{\Gamma}^h$. By \cref{ExistBiggestFlow}, as $\hat{\Gamma}^h_i\subseteq\Gamma_i\subseteq F^X(\Gamma_i)$, 
	\[
	F^X(\hat{\Gamma}^h_i)\subseteq F^X(\Gamma_i).
	\]
	Moreover, the choice of $i$ ensures
	\[
	\hat{\Gamma}^h_i|_{[0,T_1]}=\tilde{\Gamma}^h|_{[0,T_1]}.
	\]
	Hence, by uniqueness, \cref{ExistBiggestFlow}, 
	\[
	F^X(\tilde{\Gamma}^h)|_{[0,T_1]}=F^X(\hat{\Gamma}^h_i)|_{[0,T_1]}\subseteq F^X(\Gamma_i)\subseteq\bigcup_{i=1}^\infty F^X(\Gamma_i)=Z_\infty.
	\]
	As $T_1$ is arbitrary, we conclude that, for all $h>0$,
	$
	F^X(\tilde{\Gamma}^h)\subseteq Z_\infty.
	$
	Hence, by \eqref{interior-eq}, 
	\begin{equation} \label{interior-include-union}
	\mathrm{int}_{C_{U,0}}(F^X(\Gamma))=\bigcup_{h>0} F^X(\tilde{\Gamma}^h)\subseteq Z_\infty=\bigcup_{i=1}^\infty F^X(\Gamma_i).
	\end{equation}
	
	Next, for any $i$ fixed, by the hypotheses
	\[
	\Gamma_i\subseteq\mathrm{int}_{\partial C_{U,0}} (\Gamma_{i+1})\subseteq\mathrm{int}_{\partial C_{U,0}}(\Gamma)
	\]
	and thus, by \cref{CylinderFlow}, 
	\[
	F^X(\Gamma_i)\cap \partial_{\partial C_{U,0}}\Gamma=\Gamma_i\cap F^X(\partial_{\partial C_{U,0}}\Gamma)=\Gamma_i\cap\partial_{\partial C_{U,0}}\Gamma=\emptyset.
	\]
	Thus, \cref{AvoidanceThm} applied to $F^X(\Gamma_i)$ and $F^X(\partial_{\partial C_{U,0}}\Gamma)$ implies that 
	\[
	F^X(\Gamma_i)\cap F^X(\partial_{\partial C_{U,0}}\Gamma)=\emptyset.
	\]
	As $i$ is arbitrary and  $\Gamma$ is spacetime $X$-mean-convex  it follows from \cref{MCGenFlowProp} that
	\[
	Z_\infty\cap \partial_{C_{U,0}}F^X(\Gamma)=Z_\infty\cap F^X(\partial_{\partial C_{U,0}}\Gamma)=\emptyset.
	\]
	By \cref{ExistBiggestFlow}, as $\Gamma_i\subseteq\Gamma$ for all $i$, one has
	\[
	\bigcup_{i=1}^\infty F^X(\Gamma_i)=Z_\infty \subseteq F^X(\Gamma) \Rightarrow     \bigcup_{i=1}^\infty F^X(\Gamma_i)\subseteq \mathrm{int}_{C_{U,0}}(F^X(\Gamma)).
	\]
	This together with \eqref{interior-include-union} gives the first conclusion.
	
	To see the second conclusion observe that, as $ F^X(\Gamma)$ is strictly monotone and is the biggest flow generated by $\Gamma$, \cref{BoundaryPropMonotoneLem} implies
	\[
	F^X(\Gamma)=\mathrm{cl} (\mathrm{int}_{C_{U,0}} ( F^X(\Gamma)))=\mathrm{cl}\left(Z_\infty \right)
	\]
	completing the proof.
\end{proof}

\section{Ample flows} \label{AmpleSec}
We establish regularity for a restricted class of flows coming out of smooth $X$-mean-convex sets. When the initial set is compact the class of flows is provided by the biggest flow.  Importantly, such flows  are strictly monotone in the sense of Definition \ref{monotone-def} and satisfy the partition property -- i.e., Item \eqref{PartitionPropItem} of \cref{MCGenFlowProp}. To treat the unbounded case we establish these properties for a class of flows that are ``biggest" in a localized sense.

\begin{defn}
	\label{ample-flow}
	Let $Z$ be a weak $X$-flow in $N$ starting from $T_0$, that is, a weak $X$-flow generated by an initial set in $N\times\{T_0\}$.  If  $U\subseteq N$ is a smooth and strictly $X$-mean-convex compact set, then we say $Z$ is \emph{ample in $U$} provided
	\[
	Z\cap C_{U, T_0}=F^X(Z\cap \partial C_{U,T_0}; T_0)	
	\]
	The flow $Z$ is \emph{ample} if it is ample for all possible choices of $U$. 
\end{defn}

We now show ampleness of a weak $X$-flow that arises as an appropriate limit of a sequence of $X$-mean-convex flows of compact sets.
\begin{thm}
	\label{exhaustion-ample-flow}
Suppose $(N,X)$ is tame. Let $Z$ be a weak $X$-flow in $N$ starting from $T_0$. Suppose there is a sequence of $X$-mean-convex flows $Z_i$ of compact sets in $N$ starting from $T_0$, a smooth and strictly $X$-mean-convex compact set $U_0 \subseteq N$, and an $\eps_0>0$ so that the following hold:
	\begin{enumerate}
		\item $Z_i\subseteq \mathrm{int}_{C_{N,T_0}}(Z_{i+1})$ and $Z=\mathrm{cl}(\bigcup_{i=1}^\infty Z_i)$;
		\item $\mathrm{int}_{C_{N,T_0}}(Z)\setminus C_{U_0,T_0+\eps_0} =(\bigcup_{i=1}^\infty Z_i) \setminus C_{U_0, T_0+\eps_0}$;
		\item For all $T_0\leq t_1<t_2$,
		\[
		Z(t_2) \setminus U_0 \subseteq \mathrm{int} (Z(t_1));
		\]
		\item The biggest $X$-flow of $Z(T_0)$ agrees with $Z$ in $N\times[T_0, T_0+\eps_0]$ and $Z(t)\subseteq \mathrm{int}(Z(T_0))$ for $T_0<t<T_0+\eps_0$;
	\end{enumerate} 
	then $Z$ is strictly monotone and ample in $U_0$. Moreover,
	\[
	\mathrm{int}_{C_{N,T_0}}(Z)=\bigcup_{i=1}^\infty Z_i.
	\]
	If, in addition, $N$ has the property that for every compact set $K$, there is is smooth and strictly $X$-mean-convex compact set, $U$, so $K\subseteq U$, then $Z$ is ample.
\end{thm}

To prove \cref{exhaustion-ample-flow}, we start with the following simple observation. 
\begin{prop}
	\label{compact-ample-flow}
Suppose $(N,X)$ is tame.  Fix $U\subseteq N$ closed and suppose and $\Gamma\subseteq\partial C_{U,T_0}$ is closed with $F^X(\Gamma;T_0)\subseteq C_{U,T_0}$. If $V\subseteq \mathrm{int}(U)$ is a smooth and strictly $X$-mean-convex compact set th boundary, then $Z'=F^X(\Gamma;T_0)\cap C_{V,T_0}$ is the biggest $X$-flow generated by $\Gamma'=F^X(\Gamma;T_0)\cap \partial C_{V,T_0}$ with starting time $T_0$. In particular, the biggest $X$-flow of any closed set in $N$ is ample. 
\end{prop}  
\begin{proof}
	By \cref{WeakFlowLocalization}, $Z'$ is a weak $X$-flow generated by $\Gamma'$ with starting time $T_0$ and so
	$Z' \subseteq F^X(\Gamma'; T_0).$
	Conversely, as
	$\Gamma' \subseteq F^X(\Gamma; T_0)$,
	\cref{ExistBiggestFlow} implies $F^X(\Gamma'; T_0)\subseteq F^X(\Gamma;T_0)$. Likewise, as 
	$\Gamma'\subseteq\partial C_{V,T_0},$
	 \cref{ExistBiggestFlow} and \cref{CylinderFlow} imply that
	\[
	F^X(\Gamma'; T_0)\subseteq F^X(\partial C_{V,T_0}; T_0)=C_{V,T_0}.
	\]
	Thus,
	\[
	F^X(\Gamma'; T_0)\subseteq F^X(\Gamma;T_0)\cap C_{V,T_0}=Z'.
	\]
	That is, $Z'=F^X(\Gamma';T_0)$. In particular, if $U=N$ and $\Gamma\subseteq N\times\{T_0\}$, it follows that $F^X(\Gamma;T_0)$ is ample as $V$ is arbitrary. 
\end{proof}

We are now ready to prove \cref{exhaustion-ample-flow}.
\begin{proof}[Proof of \cref{exhaustion-ample-flow}]
	First observe a straightforward modification of the proof of \cite[Theorem 24]{HershkovitsWhite} shows $Z$ is a weak $X$-flow starting from $T_0$. As $U_0$ has smooth boundary and is strictly $X$-mean-convex, a small tubular neighborhood of $U_0$ yields a subset $U\subseteq N$ that is a smooth and strictly $X$-mean-convex compact set satisfying $U_0\subseteq \mathrm{int}(U)$. Let 
	\begin{align*}
	\Gamma & =Z \cap \partial C_{U,T_0}, \\
	\Gamma_i & =Z_i \cap \partial C_{U,T_0}.
	\end{align*}
	We claim that the $\Gamma_i$ are spacetime $X$-mean-convex. 
	Indeed, by \cref{MCGenFlowProp} each $Z_i\subseteq C_{N,T_0}$ is strictly monotone and, by general topology, so are $Z_i\cap C_{U,T_0}\subseteq C_{U,T_0}$ and $\Gamma_i\cap C_{\partial U,T_0}\subseteq C_{\partial U,T_0}$. 
	The hypotheses ensure that $Z_i=F^X(Z_i(T_0)\times \set{T_0}; T_0)$.  Hence, by \cref{compact-ample-flow}, the $Z_i$ are ample in $U$; that is, 
	\begin{equation} \label{seq-ample}
	Z_i\cap C_{U,T_0}=F^X(\Gamma_i;T_0).
	\end{equation}
	As $U$ is a smooth and strictly $X$-mean-convex compact set, the $\Gamma_i$ are $X$-mean-convex spacetime sets.
	
	Next we claim that $\Gamma$ is a spacetime $X$-mean-convex. To see this, by \cref{compact-ample-flow}, Hypothesis (4) ensures that
	\begin{equation} \label{short-time-ample}
	\begin{split}
	(Z\cap C_{U,T_0})|_{[T_0,T_0+\eps_0]} & =(F^X(Z(T_0)\times\{T_0\};T_0)\cap C_{U,T_0})|_{[T_0,T_0+\eps_0]} \\
	& =F^X(\Gamma;T_0)|_{[T_0,T_0+\eps_0]}
	\end{split}
	\end{equation}
	and so, for $T_0<t<T_0+\eps_0$,
	\[
	F^X_t(\Gamma;T_0)=Z(t)\cap U \subseteq\mathrm{int}(Z(T_0))\cap U\subseteq \mathrm{int}_U(\Gamma(T_0)).
	\]
	That is, Item (2) of \cref{MCSpaceTimeDef} holds. 
	Likewise, as $U_0\subseteq\mathrm{int}(U)$, Hypothesis (3) ensures for $T_0\leq t_1<t_2$
	\[
	\Gamma(t_2)\cap \partial U=(Z(t_2)\setminus U_0)\cap \partial U \subseteq\mathrm{int}(Z(t_1))\cap \partial U\subseteq \mathrm{int}_{\partial U}(\Gamma(t_1)\cap \partial U).
	\]
	That is, $\Gamma\cap C_{\partial U,T_0}\subseteq C_{\partial U,T_0}$ is strictly monotone, verifying Item (1) of \cref{MCSpaceTimeDef}.
	Hence, as $U$ a smooth and strictly $X$-mean-convex compact set the claim is verified.

	%
	
	We next show
	\[
	\mathrm{int}_{\partial C_{U,T_0}} (\Gamma) = \bigcup_{i=1}^\infty \Gamma_{i}.
	\]
	To that end, by Hypothesis (1) and general topology
	\[
	\Gamma_i=Z_i\cap  \partial C_{U,T_0} \subseteq\mathrm{int}_{C_{N,T_0}}(Z_{i+1})\cap \partial C_{U,T_0}\subseteq \mathrm{int}_{C_{\partial U,T_0}}(Z_{i+1}\cap \partial C_{U,T_0}).
	\] 
	Hence,   $\Gamma_i\subseteq \mathrm{int}_{\partial C_{U,T_0}}(\Gamma_{i+1}).$
	As such, by Hypothesis (1) ,
	\[
	\bigcup_{i=1}^\infty \Gamma_i\subseteq\bigcup_{i=1}^\infty \mathrm{int}_{\partial C_{U,T_0}}(\Gamma_i)\subseteq \bigcup_{i=1}^\infty \Gamma_i=\partial C_{U,T_0}\cap \bigcup_{i=1}^\infty Z_i\subseteq \partial C_{U,T_0}\cap Z= \Gamma.
	\]
	As $\mathrm{int}_{\partial C_{U,T_0}}(\Gamma_i)$ are open subsets of $\partial C_{U,T_0}$, their union is open and so 
	$$
	\bigcup_{i=1}^\infty \Gamma_i\subseteq \bigcup_{i=1}^\infty \mathrm{int}_{\partial C_{U,T_0}}(\Gamma_i)\subseteq \mathrm{int}_{\partial C_{U,T_0}} ( \Gamma).
	$$ 
	
	We now show  the reverse inclusion.  Pick $P=(p,t)\in \mathrm{int}_{\partial C_{U,T_0}}(\Gamma)$.   When $t>T_0$, by Hypothesis (3), $Z\setminus C_{U_0,T_0}$ is closed and strictly monotone in $C_{N\setminus U_0,T_0}$. As $U\setminus U_0$ is a closed subset of the manifold $N\setminus U_0$ and $\partial_{N\setminus U_0} U=\partial U$, \cref{strictly-monotone-lem} implies
	\begin{equation*} 
	P\in	\mathrm{int}_{\partial C_{U,T_0}}(\Gamma)|_{(T_0,\infty)}\subseteq \mathrm{int}_{C_{N\setminus U_0,T_0}}(Z\setminus C_{U_0,T_0})|_{(T_0,\infty)}\subseteq\mathrm{int}_{C_{N, T_0}}(Z).
	\end{equation*}    
	When $t=T_0$ and $p\in \mathrm{int}(U)$, the construction implies $p\in\mathrm{int}(Z(T_0))$. Likewise, when $t=T_0$ and $p\in \partial U$,  $(p, T_0+\delta)\in \Gamma\subseteq Z$ for $\delta>0$ small and so Hypothesis (3) implies $p\in \mathrm{int}(Z(T_0))$.  By Hypothesis (4), $Z$ agrees with the biggest flow of $Z(T_0)$ in $N\times [T_0,T_0+\eps_0]$ and so, arguing as in the proof of \cref{CylinderFlow},  $P\in\mathrm{int}_{C_{N,T_0}}(Z)$. 
	Hence,  as $P\in \mathrm{int}_{C_{N,T_0}}(Z)\setminus C_{U_0, T_0+\eps}$, Hypothesis (2) implies $P\in Z_i$ for some $i$.  When $t=T_0$, this immediately implies $P\in \Gamma_i$.  When $t>T_0$, as $U$ is a smooth strictly $X$-mean-convex compact set,  Item (2) of Proposition \ref{CylinderFlow} and \eqref{seq-ample} imply $P\in \Gamma_i$.   This establishes the reverse inclusion and the claimed equality.

	By what we have shown above, it follows from \cref{MCLimitProp} and \eqref{seq-ample} that 
	\begin{equation} \label{InteriorEqn}
	\mathrm{int}_{C_{U,T_0}}(F^X(\Gamma;T_0))=\bigcup_{i=1}^\infty F^X(\Gamma_i;T_0)=\bigcup_{i=1}^\infty Z_i\cap C_{U,T_0}.
	\end{equation}
	By \cref{MCGenFlowProp}, as $\Gamma$ is spacetime $X$-mean-convex, $F^X(\Gamma;T_0)\subseteq C_{U,T_0}$ is strictly monotone. Taking the closure of \cref{InteriorEqn} and appealing to \cref{strictly-monotone-lem} yield 
	\[
	F^X(\Gamma;T_0)|_{(T_0,\infty)}=\mathrm{cl}(\mathrm{int}_{C_{U,T_0}}(F^X(\Gamma;T_0)))|_{(T_0,\infty)}=\mathrm{cl}\left(\bigcup_{i=1}^\infty Z_i\cap C_{U,T_0}\right)|_{(T_0,\infty)}.
	\]
	Hence, by Hypothesis (1) and \eqref{short-time-ample},
	\[
	F^X(\Gamma;T_0)=Z\cap C_{U,T_0}
	\]
	and so $Z$ is ample in $U$. Thus, by \eqref{InteriorEqn},
	\[
	\mathrm{int}_{C_{U,T_0}}(Z\cap C_{U,T_0})=\bigcup_{i=1}^\infty Z_i\cap C_{U,T_0}.
	\]
	As $U_0\subseteq\mathrm{int}(U)$, combining the equation above with Hypothesis (2) gives
	\[
	\mathrm{int}_{C_{N,T_0}}(Z)=\bigcup_{i=1}^\infty Z_i.
	\]
	Moreover, the strict monotonicity of $Z$ follows immediately from \cref{MCGenFlowProp} together with Hypothesis (3) and the fact $U_0\subseteq\mathrm{int}(U)$.
	
	As we have shown $Z$ is ample in $U$, it follows from \cref{compact-ample-flow} that $Z$ is ample in $U_0$. In general, when  $V\subseteq N$ is smooth and strictly $X$-mean-convex compact set, then the additional hypothesis on $N$ allows us to choose $U$ so $V\cup U_0\subseteq U$ and so the above argument implies $Z$ is ample in $V$ and hence is an ample flow.
\end{proof}

We conclude with two useful properties of strictly monotone ample flows. 
\begin{lem} \label{SliceLem}
	When $(N, X)$ is tame and $Z$ is a strictly monotone ample $X$-flow in $N$ starting from $T_0$,  
	\[
	\mathrm{int}_{C_{N,T_0}} (Z)=\bigcup_{t\geq T_0} \mathrm{int} (Z(t))\times\set{t} \mbox{ and }   \partial_{C_{N,T_0}} Z=\bigcup_{t\geq T_0} \partial Z(t) \times\set{t}.
	\]
\end{lem}
\begin{proof}
	As the projection onto $N$, $(p,t)\in C_{N,T_0}\mapsto p$ is an open map, one has 
	\[
	\mathrm{int}_{C_{N,T_0}}(Z)\subseteq\bigcup_{t\geq T_0} \mathrm{int}(Z(t))\times\set{t}.
	\]
	Let $t\geq T_0$ and $p\in\mathrm{int}(Z(t))$. By \cref{spherical-barrier}, there are small $\delta,\eps>0$ and large $c>0$ so that the family of shrinking spheres
	\[
	s\in [t,t+\eps)\mapsto \bar{B}_{\sqrt{\delta^2-c(s-t)}}(p)
	\]
	is a strong $X$-flow barrier so it is a weak $X$-flow, and $\bar{B}_\delta(p)\subseteq Z(t)$ is a smooth and strictly $X$-mean-convex compact set. The ampleness of $Z$ gives $Z\cap C_{\bar{B}_{\delta}(p),T_0}$ is the biggest flow generated by $Z\cap \partial C_{\bar{B}_{\delta}(p),T_0}$. Thus,  \cref{ExistBiggestFlow} implies that, for $s\in [t,t+\eps)$,
	\[
	B_{\sqrt{\delta^2-c\eps}}(p)\subseteq\bar{B}_{\sqrt{\delta^2-c(s-t)}}(p)\subseteq Z(s)
	\]
	As $Z$ is strictly monotone, one has $B_{\sqrt{\delta^2-c\eps}}(p)\subseteq Z(s)$ for all $s\in [T_0,t]$. Hence, $B_{\sqrt{\delta^2-c\eps}}(p)\times [T_0,t+\eps)\subseteq Z$ and so $(p,t)\in\mathrm{int}_{C_{N,T_0}}(Z)$. As $p,t$ are arbitrary, 
	\[
	\mathrm{int}_{C_{N,T_0}}(Z)\supseteq\bigcup_{t\geq T_0} \mathrm{int}(Z(t))\times\set{t}.
	\]
	This proves the first identity and the second follows immediately. 
\end{proof}

\section{Proof of \cref{weak-theorem}} \label{ProofWeakThmSec}
This section is devoted to the proof of \cref{weak-theorem}. We restrict attention to the expander setting and take  $X=-\frac{\mathbf{x}}{2}$ and $N=\mathbb{R}^{n+1}$. In this case, the $X$-mean-curvature flow is precisely the EMCF, \cref{RSMCF}, and static solutions of the flow are self-expanders. A fact we use throughout is that in this setting there is a correspondence between the biggest $X$-flows (resp. Brakke $X$-flows) and the usual level set flow associated to mean curvature flow (resp. Brakke flows) via the change of variables 
\begin{align}\label{ChangeOfVariables}
	(\mathbf{x},t)\mapsto (e^{\frac{t}{2}}\mathbf{x}, e^t)
\end{align}
-- see \cite[Section 13]{HershkovitsWhite} and \cite[6.3 and 10.3]{IlmanenMAMS} for additional discussion.

Fix a smooth and strictly $X$-mean-convex closed set $\Omega_0\subseteq \Real^{n+1}$ so the smooth hypersurface, $\partial \Omega_0$, is $C^3$-asymptotic to a $C^3$-regular cone $\mathcal{C}$.  We will construct the flow of \cref{weak-theorem} by taking an appropriate exhaustion of $\Omega_0$ by  compact subsets and proving regularity of their limits.  
To that end, pick a sequence $R_i\to\infty$ and smooth the corners of $\Omega_0\cap \bar{B}_{R_i}$ as in \cref{RoundingProp}.  This yields a sequence of smooth and strictly $X$-mean-convex compact sets $\Omega^i_{0}\subseteq\mathrm{int}(\Omega_0)$ such that $\partial\Omega^i_{0}\to\partial\Omega_0$ in $C^\infty_{loc}(\mathbb{R}^{n+1})$ and $\Omega^i_{0}\subseteq \mathrm{int}(\Omega^{i+1}_{0})$. Let $\Omega_i=F^X(\Omega^i_{0}\times\{0\};0)$ be the biggest $X$-flow of $\Omega^i_{0}$ with starting time $0$ and $M_i=\partial_{C_{\mathbb{R}^{n+1},0}}\Omega_i$. Thus,  by \cite[Theorems 26 and 30]{HershkovitsWhite},
\begin{equation} \label{MonotoneExhaustEqn}
	\Omega_i\subseteq \mathrm{int}_{C_{\mathbb{R}^{n+1},0}}(\Omega_{i+1})
\end{equation} 
and, by \cite[Theorem 32]{HershkovitsWhite} and \cref{MCGenFlowProp}, $\Omega_i$ is strictly monotone with
\begin{equation} \label{NonFattenEqn}
	F^X(\partial\Omega^i_{0}\times\{0\};0)=\partial_{C_{\mathbb{R}^{n+1},0}}F^X(\Omega^i_{0}\times\{0\};0)=M_i. 
\end{equation} 
Using this sequence of closed spacetime sets, define
\begin{equation} \label{LimitFlowEqn}
\Omega=\mathrm{cl}(\bigcup_{i=1}^\infty \Omega_i) \mbox{ and } 	M=\partial_{C_{\mathbb{R}^{n+1},0}}\Omega.
\end{equation}
By \cite[Theorem 24]{HershkovitsWhite}, $\Omega$ is a weak $X$-flow with $\Omega(0)=\Omega_0$ and $M(0)=\partial \Omega_0$ a smooth hypersurface. The goal of this section is to establish regularity of $M$ as a spacetime set.

As in \cite{WhiteMCSize}, we analyze the flows of associated boundary measures. By \eqref{NonFattenEqn}, the flows $M_i$ do not fatten and so the dictionary with the usual level set flow means \cite[11.4]{IlmanenMAMS} and \cite{WhiteRegularity} directly imply that
\[
\mathcal{M}_i = \{\mathcal{H}^n \llcorner M_i(t)\}_{t \in [0,\infty)}
\]
are unit-regular integral Brakke $X$-flows. The compactness theorem of Brakke \cite[Chapter 4]{Brakke} (cf. \cite[7.1]{IlmanenMAMS}) implies that, up to passing to a subsequence, the $\mathcal{M}_i$ converge to an integral Brakke $X$-flow, $\mathcal{M}$. By White's local regularity theorem, \cite{WhiteRegularity}, as each $\mathcal{M}_i$ is unit regular, so is $\mathcal{M}$. Ultimately, we will establish that $\mathcal{M}(t)=\mathcal{H}^n\llcorner M(t)$.

%

We first establish some short time and asymptotic regularity of $M$.
\begin{prop} \label{infinity-regularity}
	There are $R_0 > 1$ and $\delta_0>0$ so that $M \setminus C_{\bar{B}_{R_0},\delta_0}$ is the spacetime track of a smooth $X$-mean curvature flow and away from $C_{\bar{B}_{R_0}, \delta_0}$, $M_i$ converge locally smoothly to $M$.  In this region, this flow agrees with the flow $\mathcal{M}$ in that
	$$
	 \mathcal{H}^n\llcorner M(t)= \mathcal{M}(t), t\in [0, \delta_0] \mbox{ and }
	 $$
	 $$	\mathcal{H}^n\llcorner (M(t)\setminus\bar{B}_{R_0}) = \mathcal{M}(t)\llcorner(\mathbb{R}^{n+1}\setminus\bar{B}_{R_0}), t\in [0, \infty).
	 $$
	Moreover, given $\eps>0$ there is a radius $R_\eps\geq R_0$ so that for , $M(t)\setminus \bar{B}_{R_0}$ can be written as a normal graph of a  $C^{2}$ function $w(\cdot, t)$ over some subset of $\mathcal{C}$ with 
	\begin{align*}
		\sum_{j=0}^2 |\mathbf{x}(p)|^{j-1}|\nabla^j_\mathcal{C} w(p,t)|\leq \eps.
	\end{align*}
\end{prop}
\begin{proof}
   As $\Omega_0$ is asymptotically conical and smooth, the construction of the $\Omega_0^i$ and $M_0^i=\partial \Omega_0^i$ ensure that for any $\eta>0$ there is a radius $R_\eta>1$ and scale $r_\eta>0$ so the following holds:  When $R>R_\eps$, there is $i_0=i_0(R)$, so that, for $i>i_0$,
    \begin{enumerate}
    	\item When $p\in B_R\cap \partial \Omega_0^i$, $M_0^i\cap B_{r_0}(p)$ is a $C^{3}$ $\mathbf{n}_{M_0^i}(p)$-graph of size $\eta$ on scale $r_\eta$ -- see \cite[Definition 3.1]{BWSmoothCompactness};
    	\item $M_0^i\cap (B_{2R}\setminus \bar{B}_R)$ can be written as a normal graph of a $C^{2}$-function $w^{i}_0$ over (some subset of) $\mathcal{C}$ with $\sum_{j=0}^2 R^{j-1}|\nabla^j_\mathcal{C} w^{i}_0|<\eta$. 
    \end{enumerate}
 Using the dictionary between expander flows and standard mean curvature flows, we may use the pseudo-locality result of \cite{IlmanenNevesSchulze} to argue as in \cite[Proposition 3.3]{BWSmoothCompactness}. In particular,  for the given $\eps>0$ we may choose $\eta>0$ and $\delta=\delta(\eps)>0$ sufficiently small and obtain the following consequences for $R=R_\eps\geq R_\eta$, and $i\geq i_0(R)$:
     \begin{enumerate}
 	\item When $p\in B_R\cap M_i(0)$ and $t\in [0,\delta]$, $M_i(t)\cap B_{r_0}(p)$ is a $C^{2,\alpha}$ $\mathbf{n}_{M_i(0)}(p)$-graph of size $\eps$ on scale $r_\eta$;
 	\item For all $t\geq 0$, $M_i(t)\cap (B_{2R}\setminus \bar{B}_R)$ can be written as a normal graph of a $C^{2}$-function $w^{i}_0$ over (some subset of) $\mathcal{C}$ with $\sum_{j=0}^2 R^{j-1}|\nabla^j_\mathcal{C} w^{i}_0|<\eps$. 
 \end{enumerate}
Hence, up to passing to a subsequence and choosing $R_0$ appropriately large and $\delta_0$ appropriately small, the  Arzel\`{a}-Ascoli theorem and standard parabolic regularity theory theory imply the $M_i$ converge in $C^\infty_{loc}(\Real^{n+1}\times \Real \setminus C_{\bar{B}_{R_0}, \eps_0})$ to a smooth $X$-mean curvature flow which, by construction, must be
$
M\setminus C_{\bar{B}_{R_0}, \eps_0}.
$
This also implies the claimed equality of the measures and asymptotic regularity.
\end{proof}
We now apply \cref{exhaustion-ample-flow} with $Z_i=\Omega_i$ and $Z=\Omega$ to prove the following:
\begin{prop} \label{monotone-prop}
	The flow $\Omega$ is strictly monotone and ample. Furthermore, 
	\begin{enumerate}
		\item $M(t)=\partial\Omega(t)$ for each $t\geq 0$ and, for $0\leq a<b<\infty$, $\set{M(t)}_{t\in [a,b]}$ forms a partition of $\Omega(a)\setminus\mathrm{int}(\Omega(b))$;
		\item If there is a smooth closed set $\Omega'\subseteq \mathrm{int}(\Omega_0)$ so that $\partial \Omega'$ is self-expander asymptotic to a $C^3$-regular cone, $\mathcal{C}'$, then $\Omega'\subseteq \mathrm{int}(\Omega(t))$ for all $t\geq 0$.
	\end{enumerate}	
\end{prop}
\begin{proof}
	For any $R>0$, in the expander setting $\bar{B}_R$ is strictly $X$-mean-convex.   We check that the remaining conditions of \cref{exhaustion-ample-flow} hold.  Hypothesis (1) of \cref{exhaustion-ample-flow} follows from \eqref{MonotoneExhaustEqn}. By \cref{infinity-regularity}, $M\setminus C_{\bar{B}_{R_0}, \delta_0}$ is the locally smooth limit of the $M_i$, and so
	\begin{align} \label{ExhaustEndEqn}
		&\phantom{{}={}}\mathrm{int}_{C_{\mathbb{R}^{n+1},0}}(\Omega)\setminus C_{\bar{B}_{R_0},\delta_0} = (\Omega \setminus M) \setminus C_{\bar{B}_{R_0},\delta_0} \nonumber  = \left(\bigcup_{i=1}^\infty \Omega_i\right)\setminus C_{\bar{B}_{R_0},\delta_0},
	\end{align}
	That is, Hypothesis (2) of \cref{exhaustion-ample-flow} holds with $U_0 = \bar{B}_{R_0}$.

	By construction, $\Omega$ is monotone. By \cref{infinity-regularity}, $\partial \Omega(t)\setminus \bar{B}_{R_0}={M}(t)\setminus \bar{B}_{R_0}$, and thus the $X$-mean-curvature of ${M}(t)\setminus \bar{B}_{R_0}$ points into $\Omega(t)$ for each $t\geq 0$. As $\Omega(0)=\Omega_0$ is strictly $X$-mean-convex, the strong maximum principle implies the $X$-mean-curvature of ${M}(t)\setminus \bar{B}_{R_0}$ vanishes nowhere for all $t\geq 0$. It follows that, for $t_2>t_1\geq 0$,
	\[
	\Omega(t_2)\setminus \bar{B}_{R_0}\subseteq\mathrm{int}(\Omega(t_1))
	\]
	verifying Hypothesis (3) of \cref{exhaustion-ample-flow}.
	
	Likewise, by \cref{BigFlowProp}, as $\partial \Omega_0$ is smooth and asymptotically conical,
	\[
	\Omega|_{[0,\delta_0]}=F^X(\Omega_0\times\{0\};0)|_{[0,\delta_0]}.
	\]
	By the strong maximum principle, as $\Omega(0)=\Omega_0$ is strictly $X$-mean-convex, so is $\Omega(t)$ for each $t\in [0,\delta_0]$. Thus $\Omega(t_2)\subseteq \mathrm{int}(\Omega(t_1))$ for all $0\leq t_1<t_2\leq\delta_0$. That is, Hypothesis (4) of \cref{exhaustion-ample-flow} holds. 
	
	As such, $Z_i=\Omega_i$ and $Z=\Omega$ satisfy all the hypotheses of \cref{exhaustion-ample-flow}, and so we conclude that $\Omega$ is strictly monotone and ample with 
	\begin{equation} \label{IntClosureEqn}
		\mathrm{int}_{C_{\mathbb{R}^{n+1},0}}(\Omega)=\bigcup_{i=1}^\infty \Omega_i.
	\end{equation}
By \cref{SliceLem}, $M(t)=\partial\Omega(t)$ for each $t\geq 0$ and so, as $\Omega$ is ample,  \cref{MCGenFlowProp} and the asymptotic properties of $M$ ensure, $\{M(t)\}_{t \in [a,b]}$ forms a partition of $\Omega(a)\setminus\mathrm{int}(\Omega(b))$. \par 
	
	To show Item (2), observe that there is a sequence of compact sets $\Omega'_j\subseteq\Omega'\subseteq\mathrm{int}(\Omega_0)$ so that $\Omega'_{j}\subseteq\Omega'_{j+1}$ and $\partial\Omega'_j\to\partial\Omega'$ in $C^\infty_{loc}(\mathbb{R}^{n+1})$. By \cite[Theorem 31]{HershkovitsWhite}, $F^X(\Omega'_j\times\{0\};0)$ agrees with the compact region bounded by the spacetime track of the smooth $X$-mean-curvature flow of $\partial\Omega'_i$ as long as the latter exists.  Arguing as in the proof of \cref{BigFlowProp}, one obtains
	\[
\mathrm{cl}\left(	\bigcup_{j=1}^\infty F^X(\Omega'_j\times\{0\};0)\right)=\Omega'\times [0,\infty).
	\]
	As the $\Omega'_j$ are compact and contained in $\mathrm{int}(\Omega_0)$ it follows from the construction of the $\Omega_0^i$ that for any fixed $j$, there is a sufficiently large $i$ so that $\Omega_j'\subseteq \Omega_0^i$. Thus, by \cite[Theorem 19]{HershkovitsWhite} and \eqref{IntClosureEqn},
	\[
	\Omega'\times [0,\infty)=\mathrm{cl}\left(\bigcup_{j=1}^\infty F^X(\Omega'_j\times\{0\};0)\right) \subseteq\mathrm{cl}\left(\bigcup_{i=1}^\infty \Omega_i\right)=\Omega.
	\]
	Finally, the hypothesis ensures $\Omega'\subseteq \mathrm{int}(\Omega(0))$.  Moreover, as $\Omega$ is strictly monotone and ample, \cref{SliceLem} implies $\Omega'\subseteq\mathrm{int}(\Omega(t))$ for $t> 0$, proving the claim.
\end{proof}
Next we adapt the arguments of the main theorems of \cite{WhiteMCSize, WhiteMCNature} to establish the regularity of $M$. A subtle, but important, issue is that, \emph{a priori}, it is not known that $t\in [0,\infty)\mapsto \mathcal{H}^n\llcorner M(t)$ is an integral Brakke $X$-flow. In \cite[Theorem 5.1]{WhiteMCSize}, this is established in the closed setting using the fact that the flow does not fatten and a result of Ilmanen \cite[11.4]{IlmanenMAMS}.   It is likely the elliptic regularization procedure for flows with boundary as sketched in \cite[Section 5]{WhiteSubsequent} could be used to this end.  However, in order to present a more self-contained argument we instead show the arguments of \cite{WhiteMCSize} apply with minor modifications to the flow $\mathcal{M}$.

We first show $M$ coincides with the support of the Brakke $X$-flow $\mathcal{M}$. For a family $t\mapsto\mathcal{N}(t)$ of Radon measures on $\mathbb{R}^{n+1}$, the \emph{support of $\mathcal{N}$} is the smallest closed spacetime set $\spt(\mathcal{N})$ such that for each $t$, $\mathcal{N}(t)$ is supported in $\{p\in\mathbb{R}^{n+1}\mid (p,t)\in\spt(\mathcal{N})\}$. Recall that, for a spacetime point $(p,t)$, the Gaussian density of $\mathcal{N}$ at $(p,t)$ is:
\[
\Theta(\mathcal{N},(p,t))=\lim_{r\downarrow 0} (4\pi r^2)^{-\frac{n}{2}} \int e^{-\frac{|\mathbf{x}-\mathbf{x}(p)|^2}{4r^2}} \, d\mathcal{N}(t-r^2)
\] 
\begin{prop} \label{SupportProp}
	If $t>0$ then $(p,t)\in \spt(\mathcal{M})\iff (p,t)\in M$.
\end{prop}
\begin{proof}
	As $\Omega_i$ is strictly monotone and $\mathcal{M}_i(s)=\mathcal{H}^n\llcorner M_i(s)$ for $s\geq 0$, the proof of \cite[Theorem 5.3]{WhiteMCSize} also implies that when $s>0$
	\begin{equation} \label{SupportExhaustEqn}
		(q,s)\in \spt(\mathcal{M}_i) \iff (q,s)\in M_i.
	\end{equation}
	As $\spt(\mathcal{M})$ is a weak $X$-flow $\spt(\mathcal{M})\subseteq \Omega$.  Suppose that $(p,t)\in\spt(\mathcal{M})$, but $(p,t) \in \mathrm{int}_{C_{\R^{n+1},0}}(\Omega)$. By \eqref{IntClosureEqn}, this means  $(p,t)\in \Omega_{i_0}$ for some $i_0$ and so, by  \eqref{MonotoneExhaustEqn}, $(p,t)\in \mathrm{int}_{\Real^{n+1}, 0}(\Omega_{j})$ for all $j> i_0$. As $(p,t) \in \spt(\mathcal{M})$ and $\mathcal{M}_i\to\mathcal{M}$ as Brakke $X$-flows, there is a sequence $(p_i,t_i)\in\spt(\mathcal{M}_i)$ converging to $(p,t)$. By \cref{SupportExhaustEqn}, $(p_i,t_i)\in M_i$, however, for $i> i_0+1$, $ \mathrm{int}_{\Real^{n+1}, 0}(\Omega_{i_0+1})$ is disjoint from $M_i$, which yields a contradiction.
	
	For the other direction, suppose $(p,t)\in M$. 
 By the strict monotonicity shown in \cref{monotone-prop}, we may choose $p_i'\in \mathrm{int}(\Omega(\frac{t}{2}))\setminus \Omega(t)$ so $p_i'\to p$. By construction, $p_i'\not\in \Omega_j(t)$ for all $j$, but, for large enough $j$, $p_i'\in \Omega_j(\frac{t}{2})$ and so the partition property of the $M_j$ implies there is a $t_i'\in (\frac{t}{2},t)$ so $(p_i',t_i')\in M_{j'}$ where $j'$ depends on $i$.  As $p_i'\to p$ and $[\frac{t}{2}, t]$ is compact, up to passing to a subsequence, $(p_i', t_i')\to(p,t') \in \Omega$. By the partition property of $M$ shown in \cref{monotone-prop},   $t'=t$ as otherwise $(p,t')\in \mathrm{int}(\Omega)$ which contradicts choice of $(p_i', t_i')$ for large $i$. Hence,  \eqref{SupportExhaustEqn} gives $(p'_i,t'_i)\in \spt(\mathcal{M}_{j'})$ and so, by monotonicity \cite[Section 11]{WhiteStrat}, $\Theta(\mathcal{M}_{j'}, (p'_i,t'_i))\geq 1$. By the upper semicontinuity of Gaussian density, we have $\Theta(\mathcal{M}, (p,t))\geq 1$ and so, by construction, $(p,t)\in\spt(\mathcal{M})$.
\end{proof}	
Next we establish a multiplicity bound for $\mathcal{M}$. This uses the following one-sided minimization property with respect to the functional $E$, \cref{expander-functional}, for strictly monotone ample $X$-flows -- and is a simple adaptation of \cite[Theorem 3.5]{WhiteMCSize}.  
\begin{lem} \label{OneSideMinimizeLem}
	Let $Z\subseteq C_{\mathbb{R}^{n+1},0}$ be a strictly monotone ample $X$-flow with starting time $0$ and assume $Z(0)$ is strictly $X$-mean-convex. When $t>0$:  
	\begin{enumerate}
		\item Any locally $E$-minimizing compact hypersurface,  $\Lambda\subseteq \mathrm{int}(Z(0))$ with boundary $\gamma\subseteq Z(t)$ satisfies $\Lambda\subseteq Z(t)$. 
		Here $\Lambda$ is locally $E$-minimizing means $\Lambda\setminus \gamma$ can be covered by balls $B$ so that $\Lambda\cap B$ is minimizing for the functional $E$. 
		\item A closed set, $V$, with $Z(t)\subseteq V\subseteq \mathrm{int}(Z(0))$ and $\mathrm{cl}(V\setminus Z(t))$ contained in a ball $B_R$, satisfies $E[\partial V\cap\bar{B}_R]\geq E[\partial Z(t)\cap \bar{B}_R]$.
	\end{enumerate}
\end{lem}
\begin{proof}
	Fix an $R>0$, and observe $\bar{B}_R$ is strictly $X$-mean-convex. By \cref{strictly-monotone-lem}, as $Z$ is strictly monotone and ample, $Z\cap C_{\bar{B}_{R},0}$ is an $X$-mean-convex flow generated by $Z\cap\partial C_{\bar{B}_{R},0}$ with starting time $0$. Let 
	\[
	Z'=\partial_{C_{\bar{B}_{R},0}} (Z\cap C_{\bar{B}_{R},0}) \mbox{ and } \Gamma'=\partial_{\partial C_{\bar{B}_{R},0}} (Z\cap\partial C_{\bar{B}_{R},0}).
	\]
	Then, by \cref{MCGenFlowProp} and \cref{SliceLem}, $Z'$ is a weak $X$-flow generated by $\Gamma'$ and $\{Z'(s)\cap B_R\}_{s\in [0,t)}$ forms a partition of $(Z(0)\setminus Z(t))\cap B_{R}$.  As $X=-\frac{\mathbf{x}}{2}$, the closed set $\bigcup_{t\geq 0} \Lambda\times\{t\}$ is a weak $X$-flow generated by $(\Lambda\times\{0\})\cup(\bigcup_{t>0}\gamma\times\{t\})$ with starting time $0$. Enlarging $R$ if needed, the hypotheses ensure $\gamma \cap Z'(s)=\Gamma'(s)\cap \Lambda=\emptyset$ for all $s\in [0,t)$. It follows from \cref{AvoidanceThm} that $\Lambda\cap Z'(s)=\emptyset$ for all $s\in [0,t)$ and so $\Lambda\subseteq Z(t)$, proving Item (1). 
	
	To see Item (2), suppose $V'$ minimizes $E[\partial V'\cap\bar{B}_R]$ subject to $Z(t)\subseteq V'\subseteq Z(0)$ and $\mathrm{cl}(V'\setminus Z(t))\subseteq \bar{B}_R$. As $\bar{B}_R$ and $Z(0)$ are both smooth and strictly $X$-mean-convex, the strong maximum principle of Solomon--White \cite{WhiteSolomon} implies $V'\setminus Z(t)\subseteq\mathrm{int}(Z(0))\cap B_R$. Thus, by Item (1), $V'\subseteq Z(t)$ and so 
	\[
	E[\partial V\cap\bar{B}_R] \geq E[\partial V'\cap\bar{B}_R]=E[\partial Z(t)\cap \bar{B}_R]
	\]
	proving Item (2).
\end{proof}	
\begin{prop} \label{MultiplicityProp}
	Given $\beta\in (1,2)$ and $p\in \mathrm{int}(\Omega(0))$, there is an $r_0>0$ so that $B_{r_0}(p)\subseteq\mathrm{int}(\Omega(0))$ and that for $r<r_0$, $0<\eps<\frac{1}{4}$ and $t>0$, if $M(t)\cap B_r(p)$ is contained in a slab $S\subseteq B_{r}(p)$ of thickness $2\eps r$ passing through $p$, then $\Omega(t)\cap (B_{r}(p)\setminus S)$ consists of $0\leq k\leq 2$ of the two connected components of $B_r(p)\setminus S$ and 
	\[
	\mathcal{M}(t)(B_r(p))\leq \beta(2-k+2n\eps) \omega_nr^n
	\]
	where $\omega_n$ is the $n$-dimensional volume of a unit $n$-ball.
\end{prop}
\begin{proof}
	By \cref{SliceLem}, \eqref{MonotoneExhaustEqn} and \eqref{IntClosureEqn} ensure that $\Omega_i(t)\subseteq \mathrm{int}(\Omega_{i+1}(t))$ and $\mathrm{int}(\Omega(t))=\bigcup_{i=1}^\infty \Omega_i(t)$. Thus, up to passing to a subsequence, $\partial \Omega_i(t)$ converges to a closed subset of $M(t)$ as $i \to \infty$. As $M(t)\cap B_r(p)\subseteq S$, the same holds for $\partial \Omega_i(t)$ when $i$ is sufficiently large. For $r_0$ with $B_{r_0}(p)\subseteq\mathrm{int}(\Omega(0))$, when $i$ is large $B_{r_0}(p)\subseteq\mathrm{int}(\Omega_i(0))$. For $r<r_0$, as $\Omega_i$ is an $X$-mean-convex flow, \cref{OneSideMinimizeLem} implies 
	\begin{align} \label{expander-energy-bound}
		E[\partial \Omega_i(t)\cap \bar{B}_r(p)] \leq E[\partial (\Omega_i(t)\cup S)\cap\bar{B}_r(p)].
	\end{align}
	Shrinking $r_0$ if needed, we can make sure 
	\[
	\sup_{x,y\in \bar{B}_{r_0}(p)}e^{\frac{|x|^2-|y|^2}{4}} \leq {\beta},
	\]
	and so that, by \cref{expander-energy-bound},
	\begin{align*}
		\mathcal{H}^n (\partial \Omega_i(t)\cap \bar{B}_r(p)) \leq \beta \mathcal{H}^n (\partial (\Omega_i(t)\cup S)\cap\bar{B}_r(p)).
	\end{align*}
	Clearly,  $\Omega_i(t)\cap (B_r(p)\setminus S)$ consists of $k$ of the two  components of $B_r(p)\setminus S$ and
	\[
	\mathcal{H}^n (\partial (\Omega_i(t)\cup S)\cap\bar{B}_r(p)) \leq (2-k+2n\eps) \omega_n r^n.
	\]
	Combining the above two inequalities gives the bound for $\mathcal{M}_i(t)$. By construction, $k$ is independent of $i$ for large $i$. As $\partial \Omega_i(t)=M_i(t)$ by \cref{SliceLem} and $\mathcal{M}_i(t)=\mathcal{H}^n\llcorner M_i(t)\to\mathcal{M}(t)$ as Radon measures, the bound for $\mathcal{M}(t)(B_r(p))$ follows by taking $i \to \infty$. 
\end{proof}
Next we will prove a sheeting theorem (see \cref{SheetingThm}) for blow-up sequences that resemble a static multiplicity-two hyperplane as analogous to \cite[Theorem 8.2]{WhiteMCSize}.
\begin{defn} \label{BlowupDefn}
	Given a point $P\in\mathbb{R}^{n+1}\times (0,\infty)$, a \emph{blow-up sequence} to $\Omega$ (resp. $\mathcal{M}$) at $P$ is a sequence $\tilde{\Omega}_i$ (resp. $\tilde{\mathcal{M}}_i$) of closed spacetime sets (resp. families of Radon measures) obtained by translating $\Omega$ (resp. $\mathcal{M}$) by $-P_i$ and then dilating parabolically by $\rho_i$ (i.e., $(\mathbf{x},t)\mapsto (\rho_i \mathbf{x},\rho_i^2t)$) for $P_i\to P$ in spacetime and $\rho_i\to\infty$. 
\end{defn}	
It is convenient to consider the pair $(\tilde{\Omega}_i,\tilde{\mathcal{M}}_i)$ as a blow-up sequence to $(\Omega,\mathcal{M})$ at $P$. By construction, $\tilde{\Omega}_i$ is a weak $X_i$-flow with starting time $T_i$.
It is readily checked that $X_i$ converges smoothly on compact sets to the zero vector field, and $T_i$ converges to $-\infty$. A straightforward modification of the proof of \cite[Theorem 24]{HershkovitsWhite} gives that a subsequence of $\tilde{\Omega}_i$ converges to a monotone weak flow $\tilde{\Omega}$ that exists for all times. Likewise, by \cite[Chapter 4]{Brakke}, a further subsequence of $\tilde{\mathcal{M}}_i$ converges to an integral Brakke flow $\tilde{\mathcal{M}}$. Such a pair $(\tilde{\Omega},\tilde{\mathcal{M}})$ is called a \emph{limit flow} to $(\Omega,\mathcal{M})$ at $P$. When $P_i\equiv P$, $(\tilde{\Omega},\tilde{\mathcal{M}})$ is called a \emph{tangent flow} to $(\Omega,\mathcal{M})$ at $P$. In particular, \cref{MultiplicityProp} implies that multiplicities higher than $2$ cannot occur for any tangent flow to $(\Omega,\mathcal{M})$. \par 
A crucial tool is the following generalization of \cite[Corollary 4.2]{WhiteMCSize} to monotone weak $X'$-flows for a smooth vector field $X'$ on $\mathbb{R}^{n+1}$. For a set $V\subseteq\mathbb{R}^{n+1}$, a point $p\in\mathbb{R}^{n+1}$ and $r>0$, the \emph{relative thickness} of $V$ in $B_r(p)$ is 
\[
\mathrm{Th}(V,p,r)=\frac{1}{r} \inf_{|\mathbf{v}|=1}\left(\sup_{p'\in V\cap B_r(p)} |\mathbf{v}\cdot (p-p')|\right).
\] 
The relative thickness is 0 if and only if $V$ is contained in a hyperplane passing through $p$.
\begin{prop} \label{ExpandHoleProp}
	Given $\eta>0$, there is a $\delta=\delta(\eta)>0$ with the following property. If $Z$ is a monotone weak $X'$-flow for a smooth vector field $X'$ on $\mathbb{R}^{n+1}$, and if $(p,t)$ is a spacetime point such that 
	\begin{equation} \label{RelThickEqn}
		\limsup_{r\to 0} \mathrm{Th}(Z(t), p,r)<\delta
	\end{equation}
	and such that $p\notin Z(t+h)$ when $h>0$, then 
	\begin{equation} \label{HoleSizeEqn}
		\liminf_{r\to 0} \frac{\dist(Z(t+r^2),p)}{r} \geq \eta.
	\end{equation}
\end{prop}
\begin{proof}
	We argue by contradiction. Suppose the proposition is false for a certain $\eta>0$. Then for each $\delta>0$ there is a monotone weak $X_\delta$-flow $Z_\delta$ in $\mathbb{R}^{n+1}$ and a spacetime point $(p_\delta,t_\delta)$ with $p_\delta\notin Z(t_\delta+h)$ when $h>0$, such that \eqref{RelThickEqn} holds for $Z=Z_\delta$ and $(p,t)=(p_\delta,t_\delta)$, but \eqref{HoleSizeEqn} fails. By translating, we may assume $p_\delta=\mathbf{0}$ and $t_\delta=0$. Thus there is an $R_\delta>0$ so that for each $0<r<R_\delta$, $\mathrm{Th}(Z_\delta(0),\mathbf{0},r)<\delta$ and $\dist(Z_\delta(r^2),\mathbf{0})<\eta r$. Combined with the monotonicity of $Z_\delta$, this implies that if $h_\delta=\delta^2R_\delta^2$ then 
	\[
	\mathrm{Th}(Z_\delta(h_\delta),\mathbf{0},r)<\delta \mbox{ for $0<r<R_\delta$}
	\]
	and 
	\[
	\dist(Z_\delta(h_\delta+r^2),\mathbf{0})<2\eta r \mbox{ for $\delta R_\delta\leq r<R_\delta$}.
	\]
	
	Thus, if
	\[
	\rho_\delta=\inf\set{r>0 \mid \dist(Z_\delta(h_\delta+r^2),\mathbf{0})\leq 2\eta r}
	\]
	then $\rho_\delta\leq\delta R_\delta$. As $\mathbf{0}\notin Z_{\delta}(h_\delta)$ one has $\rho_\delta>0$. Translating $Z_\delta$ in time by $-h_\delta$ followed by parabolically dilating by $\rho_\delta^{-1}$ produces a weak $X'_\delta$-flow $Z'_\delta$ satisfying 
	\[
	\mathrm{Th}(Z'_\delta(0),\mathbf{0},r)<\delta \mbox{ for $0<r<\delta^{-1}$}
	\]
	where we used $R_\delta\rho_\delta^{-1}\geq \delta^{-1}$, $\dist(Z'_\delta(1),\mathbf{0})=2\eta$, and
	\[
	\dist(Z'_\delta(r^2),\mathbf{0})>2\eta r \mbox{ for $0<r<1$}. 
	\]
	In addition, up to shrinking $R_\delta$, $X'_\delta$ converges smoothly on compact sets to the zero vector field as $\delta\to 0$. 
	
	If $Z'$ is a subsequential limit of $Z'_\delta\cap C_{\mathbb{R}^{n+1},0}$ as $\delta\to 0$, then $Z'$ is a weak flow, $Z'(0)$ is contained in a hyperplane and $Z'(\frac{1}{2})$ is a proper subset of the hyperplane. This implies that $Z'(t)=\emptyset$ for $t>\frac{1}{2}$, which contradicts $\dist(Z'(1),\mathbf{0})=2\eta$ and completes the proof.
\end{proof}
If $P=(p,t)$ and $P'=(p',t')$ are points in spacetime, let $\dist(P,P')$ denote the parabolic distance:
\[
\dist(P,P') = \dist((p,t),(p',t'))=\max\set{|p-p'|,|t-t'|^{\frac{1}{2}}}.
\]
Let $B^{n+1,1}_r(P)=\{P'\in \mathbb{R}^{n+1}\times\mathbb{R}\mid \dist(P,P')<r\}$ denote the spacetime (open) ball centered at $P$ with radius $r$. We omit the center $P$ when it is the spacetime origin. For two spacetime sets $Y ,Y'$,  $\dist_H(Y,Y')$ denotes their (parabolic) Hausdorff distance.
\begin{thm} \label{SheetingThm}
	Suppose $(\tilde{\Omega}_i,\tilde{\mathcal{M}}_i)$ is a blow-up sequence to $(\Omega,\mathcal{M})$ such that 
	\[
	\dist_H(\tilde{\Omega}_i\cap B^{n+1,1}_4,\mathbb{H}\cap B^{n+1,1}_4)\to 0
	\]
	where $\mathbb{H}$ is the spacetime half-space $\{x_{n+1}=0\}$. Then for sufficiently large $i$ there exist $f_i,g_i\colon B^{n,1}_1\to\mathbb{R}$ such that 
	\begin{enumerate}
		\item $f_i\leq g_i$;
		\item $f_i,g_i\to 0$ in $C^\infty(B^{n,1}_1)$ as $i\to\infty$;
		\item For $(x,t)\in B^{n+1,1}_2$ with $(x_1,\dots,x_n,0,t)\in B^{n+1,1}_1$, 
		\[
		(x,t)\in\tilde{\Omega}_i \iff f_i(x_1,\dots,x_n,t)\leq x_{n+1}\leq g_i(x_1,\dots,x_n,t);
		\]
		\item For each $x'\in B^{n}_1$, $f_i(x',t)$ and $g_i(x',t)$ are increasing and decreasing, respectively, as functions of $t$.
	\end{enumerate}
	In particular, for large $i$, $\tilde{\mathcal{M}}_i(t)\llcorner B^{n+1}_2=\mathcal{H}^n\llcorner(\partial \tilde{\Omega}_i(t)\cap B^{n+1}_2)$ for each $t\in (-4,4)$.
\end{thm}
\begin{proof}
	Let $S_i$ be the set of points $(x,t)$ such that $x$ is the center of a closed ball in $\tilde{\Omega}_i(t)$ that contains at least two points of $\partial\tilde{\Omega}_i(t)$. For large  enough $i$, $S_i\cap B^{n+1,1}_3$ is a smooth properly embedded hypersurface that separates $B^{n+1,1}_3$ into two open sets, and $\tilde{M}_i=\partial\tilde{\Omega}_i\cap B^{n+1,1}_3$ has exactly two components, $\tilde{M}_i^\pm$, each contained in one of the open sets. 
	
	To see this, we prove the Bernstein-type theorem \cite[Theorem 7.5]{WhiteMCSize} for limit flows to $(\Omega,\mathcal{M})$. With this and \cref{ExpandHoleProp}, the proof of \cite[Lemma 8.3]{WhiteMCSize} gives the claim. Indeed, any limit flow to $(\Omega,\mathcal{M})$ consists of a monotone weak flow and an integral Brakke flow as limit flows to a mean-convex mean curvature flow. The proof of \cite[Theorem 7.5]{WhiteMCSize} depends only on the properties for limit flows \cite[Theorems 5.4, 5.5, 6.1]{WhiteMCSize}, which, in our case, follow readily from \cref{SupportProp}, \cref{OneSideMinimizeLem} and \cref{MultiplicityProp} and so the conclusion of \cite[Theorem 7.5]{WhiteMCSize} is true for limit flows to $(\Omega,\mathcal{M})$. 
	
	By \cref{SupportProp}, $S_i$ separates $\tilde{\mathcal{M}}_i\llcorner B^{n+1,1}_3$ into two parts, $\tilde{\mathcal{M}}_i^\pm$. The hypotheses yield
	\[
	\spt(\tilde{\mathcal{M}}^\pm_i) = \tilde{M}^\pm_i \to \mathbb{H}\cap B^{n+1,1}_3.
	\]
	Hence, the $\tilde{\mathcal{M}}^\pm_i$ converge to  static hyperplanes with multiplicity $m_\pm$ respectively. By \cref{MultiplicityProp}, $m_+ + m_-\leq 2$ and so $m_\pm=1$. The claim follows easily from Brakke's local regularity theorem \cite[6.11]{Brakke} (see also \cite{WhiteRegularity}).
\end{proof}
We are now ready to prove the regularity and long-time behavior for the flow $\Omega$. Combined with \cref{monotone-prop}, this completes the proof of \cref{weak-theorem}.
\begin{thm} \label{StrongRegularThm}
	The boundary $M$ is smooth except for a closed set of parabolic Hausdorff dimension at most $n-1$, and  $\mathcal{M}(t)=\mathcal{H}^n\llcorner M(t)$ for each $t\geq 0$. Moreover, as $t\to\infty$, $M(t)$ converges smoothly away from a set of codimension at most 7 to a stable (possibly singular) self-expander asymptotic to $\mathcal{C}$.
	
	Furthermore, for $n\leq 6$, if $(\tilde{\Omega},\tilde{\mathcal{M}})$ is a limit flow to $(\Omega,\mathcal{M})$, then $\tilde{\Omega}(t)$ is convex for each $t\in\mathbb{R}$, and there is a $T\in [-\infty,\infty]$ such that 
	\begin{enumerate}
		\item $\mathrm{int}(\tilde{\Omega}(t))\neq\emptyset$ when $t<T$, while $\mathrm{int}(\tilde{\Omega}(t))=\emptyset$ when $t=T$;
		\item $\tilde{\mathcal{M}}$ is smooth with multiplicity-one in $\mathbb{R}^{n+1}\times (-\infty,T)$;
		\item $\tilde{\Omega}(t)=\emptyset$ when $t>T$. 
	\end{enumerate}
	If, in addition, $(\tilde{\Omega},\tilde{\mathcal{M}})$ is a tangent flow, then it is either a static multiplicity-one hyperplane, or a self-similarly shrinking sphere or generalized cylinder.
\end{thm}
\begin{proof}
	First, using Propositions \ref{SupportProp} and \ref{MultiplicityProp} in place of \cite[Theorems 5.3 and 3.9]{WhiteMCSize}, the proof of \cite[Theorem 5.4]{WhiteMCSize} implies the following: If $(\tilde{\Omega},\tilde{\mathcal{M}})$ is a limit flow to $(\Omega,\mathcal{M})$ and $T' \in (-\infty,\infty]$ is a time such that for all $t \le T'$, $\spt \tilde{\mathcal{M}}(t) = \mathbb{P}$ for some hyperplane $\mathbb{P}\subseteq\mathbb{R}^{n+1}$, then one of the following is true:
	\begin{enumerate}
		\item $\tilde{\Omega}$ is a static half-space with boundary $\mathbb{P}$ and $\tilde{\mathcal{M}}$ is the static hyperplane $\mathbb{P}$ with multiplicity one;
		\item $\tilde{\Omega}=\mathbb{P}\times (-\infty,T'']$ for some $T''\geq T'$ and, for every $t<T''$, $\tilde{\mathcal{M}}(t)$ is $\mathbb{P}$ with multiplicity two while, for every $t>T''$, $\tilde{\mathcal{M}}(t)=0$.
	\end{enumerate}
Suppose now that $(\tilde{\Omega},\tilde{\mathcal{M}})$ is a tangent flow instead of a general limit flow.   By \cref{SheetingThm}, static -- i.e., with $T''=\infty$ -- multiplicity-two hyperplanes cannot arise as tangent flows -- see \cite[Corollary 8.5]{WhiteMCSize} . Furthermore, as zoom-ins of $(\Omega,\mathcal{M})$ about a spacetime point are small perturbations of MCFs, the proof of \cite[Theorem 9.2]{WhiteMCSize}, with minor modifications, gives that quasi-static --  i.e., with $T''<\infty$ -- multiplicity-two hyperplanes cannot arise as tangent flows. Thus, only case (1) holds for tangent flows.
	
	Suppose $(\tilde{\Omega},\tilde{\mathcal{M}})$ is a static or quasi-static limit flow to $(\Omega,\mathcal{M})$. Without loss of generality, we may assume $\tilde{\mathcal{M}}(-1)\neq 0$. By \cref{SheetingThm} and Brakke's local regularity theorem \cite[6.11]{Brakke}, $\tilde{\mathcal{M}}(-1)$ is smooth at every point with a tangent hyperplane (possibly with high multiplicity). Thus, by a dimension reduction argument, e.g., \cite[Appendix A]{SimonBook}, $\tilde{\mathcal{M}}(-1)$ is smooth almost everywhere. Hence, using \cref{SupportProp} and \cref{OneSideMinimizeLem} in place of \cite[Theorems 5.3 and 3.5]{WhiteMCSize}, the proof of \cite[Theorem 6.1]{WhiteMCSize} gives that $(\tilde{\Omega},\tilde{\mathcal{M}})$ has the same one-sided minimizing property as in \cite[Corollary 8.6]{WhiteMCSize}. Combining this with a dimension reduction argument gives that $\tilde{\mathcal{M}}(-1)$ is a stable minimal hypersurface with singular set of Hausdorff dimension $\leq n-3$ -- cf. \cite[Theorem 7.2]{WhiteMCSize} for analogous arguments. It then follows from Schoen-Simon \cite{SchoenSimon} that the singular set of $\tilde{\mathcal{M}}(-1)$ has Hausdorff dimension $\leq n-7$. In particular, $\tilde{\mathcal{M}}(-1)$ cannot be a polyhedral minimal cone.  
	
	By what we have shown above, White's stratification theorem \cite[Section 11]{WhiteStrat} implies that the singular set of $\mathcal{M}$ (or equivalently $M$) has parabolic Hausdorff dimension $\leq n-1$. Thus, by Propositions \ref{monotone-prop} and \ref{SupportProp}, $\mathcal{M}(t)=\mathcal{H}^n\llcorner M(t)$ for each $t\geq 0$. The arguments above also show that $M(t)$ converges to a stable (possibly singular) self-expander asymptotic to $\mathcal{C}$ as $t \to \infty$, and the convergence is smooth away from a set of codimension at most 7. When $n\leq 6$, the convexity and regularity for limit flows follow readily from the proof of \cite[Theorem 1]{WhiteMCNature}.
\end{proof}

\section{Existence of monotone Morse flow line}
\label{MonotoneMorseFlowSec}
We show the existence of the monotone Morse flow line -- i.e., \cref{main-theorem} -- using Theorem \ref{weak-theorem} and by constructing an appropriate ancient expander mean convex flow.  This will be done by showing we can take limits of solutions constructed in \cite{BWTopologicalUniqueness}.  Throughout,  we will fix $\Sigma \subseteq \R^{n+1}$,  a smooth self-expander asymptotic to a $C^3$-regular cone $\cC$. As remarked in \cref{IntroSec}, $\lim_{\rho\to 0^+} \rho\Sigma=\mathcal{C}$ in $C^{2,\alpha}_{loc}(\mathbb{R}^{n+1} \setminus \{\mathbf{0}\})$ for any $\alpha\in (0,1)$.

Before proceeding, we introduce some conventions and notations we use in this section. The time variable $s$ is used for MCFs, while the time variable $t = \log s$ is used for EMCFs. The hypersurfaces along a MCF will be denoted with tildes.  We use the following notations for norms on functions $g\in C^\infty(\Sigma)$.
For an integer $k\geq 0$, let
\begin{align*}
\Vert g \Vert_{W, k}^2 =\sum_{i=0}^k \int_{\Sigma} |\nabla_\Sigma^i g|^2 e^{\frac{|\mathbf{x}|^2}{4}} \, d\mathcal{H}^n
\end{align*}
be a weighted integral norm.  Denote the $C^k$ and, for $\alpha \in (0,1)$,  the H\"{o}lder norm by
$$
\Vert g \Vert_{C^k(\Sigma)} =\sum_{i=0}^k \sup_{\Sigma} |\nabla_{\Sigma}^i g| \mbox{ and } \norm{g}_{C^{k,\alpha}(\Sigma)} = \norm{g}_{C^k(\Sigma)}  + [\nabla_\Sigma^k g]_{\alpha;\Sigma}.
$$
Here $[g]_{\alpha;\Sigma}$ is the (unweighted) H\"{o}lder seminorm
\begin{align*}
[g]_{\alpha; \Sigma} = \sup_{{p \ne q \in \Sigma, q \in B_\delta^\Sigma(p)}} \frac{\abs{g(p) - g(q)}}{d_\Sigma(p,q)^{\alpha}}
\end{align*}
where $d_\Sigma$ denotes the geodesic distance on $\Sigma$ and $B^\Sigma_\delta(p)$ is the geodesic ball in $\Sigma$ centered at $p$ with radius $\delta$. 
For simplicity, we also write $C^{k,\alpha} = C^{k,\alpha}(\Sigma)$ and $[\cdot]_{\alpha;\Sigma} =[\cdot]_{\alpha}$. The parabolic H\"{o}lder space $C^{k,\alpha}_P(I \times \Sigma)$ consists of functions $f\colon I\times\Sigma\to\mathbb{R}$ such that 
\begin{align*}
\norm{f}_{C^{k,\alpha}_P(I \times \Sigma)} = \sum_{i+2j \le k} \sup_{I \times \Sigma} \abs{\nabla_\Sigma^i \partial_t^j f} + \sum_{i+2j = k} [\nabla_\Sigma^i \partial_t^j f]_{\alpha; I \times \Sigma}<\infty,
\end{align*}
where $[f]_{\alpha;I \times \Sigma}$ is the H\"{o}lder seminorm 
\begin{align*}
[f]_{\alpha; I \times \Sigma} = \sup_{\substack{(t_1, p) \ne (t_2,q) \in I \times \Sigma \\ |t_1-t_2|<\delta^2, q \in B_\delta^\Sigma(p)}} \frac{\abs{f(p) - f(q)}}{d_\Sigma(p,q)^{\alpha} + \abs{t_1 - t_2}^{\frac{\alpha}{2}}}.
\end{align*}
We write $C^{k,\alpha}_P(I \times \Sigma) = C^{k,\alpha}_P$ if $I$ is understood. It is straightforward to extend the above notations to tensor fields.
In what follows, we will fix $\delta$ and $\alpha$.

Given a function $u\in C^\infty(\Sigma)$, let $\Gamma_u$ denote the normal graph of $u$ over $\Sigma$, i.e.
\begin{align*}
\Gamma_u =\{\xX(p) + u(p)\nN_\Sigma(p) \mid p \in \Sigma\}.
\end{align*}
We need the following result ensuring the expander mean curvature of $\Gamma_u$ can be computed from the function $u$ in a controlled manner.
\begin{lem}\label{GraphMCLem}
	There is an $\eta_0=\eta_0(\Sigma)>0$ so that if $\Vert u \Vert_{C^{2, \alpha}}\leq \eta_0$, then 
	$$	
	\mathbf{F}_u\colon\Sigma\to \Real^{n+1}, p \mapsto \mathbf{x}(p)+u(p) \mathbf{n}_\Sigma(p)
	$$
	is a $C^{2, \alpha}$ embedding.  Moreover,  the expander mean curvature of $\Gamma_u$ satisfies
	$$
	H_{\mathbf{F}_u}+\frac{\mathbf{F}_u\cdot \mathbf{n}_{\mathbf{F}_u}}{2}= -L_{\Sigma} u +Q(u,\mathbf{x}\cdot \nabla_\Sigma u, \nabla_\Sigma u, \nabla_\Sigma^2 u)
	$$
	where 
	$$
	L_\Sigma=\Delta_\Sigma+\frac{\mathbf{x}}{2}\cdot\nabla_\Sigma+|A_\Sigma|^2-\frac{1}{2}
	$$
	and $Q$ is a homogeneous degree-two polynomial of the form
	$$
	Q(\tau, \rho , \mathbf{d}, \mathbf{T})=\mathbf{a}(\tau, \rho , \mathbf{d}, \mathbf{T})\cdot \mathbf{d}+b(\tau , \mathbf{d}, \mathbf{T})\tau.
	$$
	Moreover, the $\mathbf{a}$ and $b$ are homogeneous degree-one polynomials whose coefficients are functions on $\Sigma$ that are bounded in $C^{0, \alpha}(\Sigma)$ by $\bar{M}=\bar{M}(\Sigma)>0$. 
\end{lem}
\begin{proof}
	First observe that, as $\Sigma$ is $C^{2,\alpha}$-asymptotically conical, it has the property that, in a uniform sized ball around each point $p\in \Sigma$, the unit normal has uniform $C^{2, \alpha}$ estimates.  For any fixed $R_1\geq 1$, this is true for $p\in B_{R_1}\cap \Sigma$ by standard interior elliptic estimates.  For $p\in \Sigma \setminus B_{R_1}$, as $\Sigma$ is $C^{2,\alpha}$-asymptotically conical,  it follows essentially from the definition -- see \cite[Proposition 4.3]{BWSmoothCompactness} -- that there are constants $\gamma>0$ and $C_1>0$ so that, with $R=|\mathbf{x}(p)|\geq R_1\geq 1$, on $B_{\gamma R}(p)\cap \Sigma$ one has the estimates
	$$
	\Vert \mathbf{n}_\Sigma \Vert_{C^0}+R \Vert\nabla_\Sigma \mathbf{n}_\Sigma\Vert_{C^0} +R^{1+\alpha }[\nabla_\Sigma \mathbf{n}_\Sigma]_{\alpha} +R \Vert\mathbf{A}_\Sigma\Vert_{C^0}+R^{1+\alpha} [\mathbf{A}_\Sigma]_\alpha \leq C_1
	$$
	A consequence is that, on $B_{\gamma R}(p)\cap \Sigma$,  
	$$
	\Vert\mathbf{x}\cdot \nabla_\Sigma \mathbf{n}_\Sigma\Vert_{C^0}+R^\alpha [\mathbf{x}\cdot \nabla_\Sigma \mathbf{n}_\Sigma]_{\alpha}\leq 2 C_1.
	$$
	Using, $L_\Sigma \mathbf{n}_\Sigma=-\frac{1}{2}\mathbf{n}_\Sigma$ -- see \cite[Lemma 5.9]{BWSpace} --
	it follows that on $B_{\gamma R}(p)\cap \Sigma$ 
	$$
	\Vert\Delta_\Sigma \mathbf{n}_\Sigma\Vert_{C^0}+	[\Delta_\Sigma \mathbf{n}_\Sigma]_\alpha \leq C_1'.
	$$
	Hence, applying standard interior Schauder estimates \cite[15, Theorem
	6.2]{GilbargTrudinger} on a fixed scale around $p$ one obtains the desired bounds on the unit normal. 
	Hence,  for $\eta_0$ sufficiently small depending only on $\Sigma$, the map $\mathbf{F}_u$
	is a $C^{2, \alpha}$  embedding of $\Gamma_u=\mathbf{F}_u(\Sigma)$. 
	
	Finally, the hypotheses ensure that \cite[Lemma A.2]{BWRelEnt} gives the desired form of the $Q$, $\mathbf{a}$ and $b$ and the boundedness of the coefficients of $\mathbf{a}$ and $b$. However, by using the established uniform H\"{o}lder estimates on $\mathbf{x}|_{\Sigma}, \mathbf{n}_\Sigma, \mathbf{A}_\Sigma$ one readily checks that the proof of \cite[Lemma A.2]{BWRelEnt} also gives the required H\"{o}lder norm estimates.
\end{proof}

Next we establish a lemma that shows a self-expander cannot be written as a one-sided normal graph with sufficiently small $C^{2}$ norm over a strictly unstable expander -- i.e., each connected component of the expander is unstable. 
\begin{lem}\label{NonStationaryLem}  Suppose $\Sigma$ is strictly unstable. 
	There is an $\eta_1=\eta_1(\Sigma)>0$ depending on $\Sigma$ so that if $u\in C^\infty(\Sigma)$ is positive and $\Vert u\Vert_{C^{2}}\leq \eta_1$, then $\Sigma'=\Gamma_u$, is not a self-expander.
\end{lem}
\begin{proof}
	Suppose the graph of $u>0$ is a self-expander, by \cite[Lemma A.2]{BWRelEnt}, there is a constant $\bar{M}_0>0$ depending  on $\Sigma$  so that if $\eps<(8\bar{M}_0)^{-1}$ and $\Vert u\Vert_{C^2}\leq \eps$, then $u$ satisfies
	$$
	|L_\Sigma u|\leq C_1\eps ( | \mathbf{x}\cdot \nabla_{\Sigma } u|+|\nabla_\Sigma u| +|u|).
	$$
	where $C_1$ depends only on $\Sigma$.
	
	For  $\mathcal{L}_\Sigma = \Delta_\Sigma + \frac{1}{2}\xX \cdot \nabla_\Sigma $, the drift Laplacian, the Cauchy-Schwarz inequality implies, 
	$$
	-\mathcal{L}_\Sigma \log u \geq (1-C_1\eps ) |\nabla_\Sigma \log u|^2 +|A_\Sigma|^2-\frac{1}{2}-\frac{3}{2} C_1\eps -\frac{1}{2} C_1 \eps|\mathbf{x}|^2,
	$$
Using the Poincar\'{e} inequality of \cite[Lemma A.1]{BWIntegerDegree} one has 
	\begin{align*}
	\int_{\Sigma}\psi^2 &(-\mathcal{L}_\Sigma \log u)e^{\frac{|\mathbf{x}|^2}{4}} \, d\mathcal{H}^n\geq   (1-C_1\eps ) \int_{\Sigma}\psi^2 |\nabla_\Sigma \log u|^2e^{\frac{|\mathbf{x}|^2}{4}} \, d\mathcal{H}^n\\
	&+\int_{\Sigma} \left(|A_\Sigma|^2-\frac{1}{2}-\frac{3}{2} C_1\eps\right) \psi^2e^{\frac{|\mathbf{x}|^2}{4}} \, d\mathcal{H}^n-8C_1 \eps\int_{\Sigma} |\nabla_\Sigma \psi|^2 e^{\frac{|\mathbf{x}|^2}{4}} \, d\mathcal{H}^n.
	\end{align*}
	
	If $\psi$ is a compactly supported smooth function on $\Sigma$ it follows from an integration by parts and the Cauchy-Schwarz inequality that when $C_1 \eps<1$,
	\begin{align*}
	(1-C_1\eps)^{-1}\int_{\Sigma} |\nabla_\Sigma \psi|^2 e^{\frac{|\mathbf{x}|^2}{4}} \, d\mathcal{H}^n& + (1-C_1\eps)\int_{\Sigma}\psi^2 |\nabla_\Sigma \log u|^2 e^{\frac{|\mathbf{x}|^2}{4}} \, d\mathcal{H}^n \\
	\geq \int_{\Sigma}\psi^2 (-\mathcal{L}_\Sigma \log u)e^{\frac{|\mathbf{x}|^2}{4}} \, d\mathcal{H}^n.
	\end{align*}	
	Combining these two inequalities yields
	\begin{align*}
	(1+8C_1 \eps) \int_{\Sigma} |\nabla_\Sigma \psi|^2 e^{\frac{|\mathbf{x}|^2}{4}} \, d\mathcal{H}^n \geq (1-C_1 \eps) \int_{\Sigma} \left(|A_\Sigma|^2-\frac{1}{2}-\frac{3}{2} C_1\eps\right) \psi^2e^{\frac{|\mathbf{x}|^2}{4}} \, d\mathcal{H}^n.
	\end{align*}
	Hence, the boundedness of the second fundamental form, implies that, up to shrinking $\eps$, 
	$$
	\int_{\Sigma} (|\nabla_\Sigma \psi|^2 - |A_\Sigma|^2 \psi^2 +\frac{1}{2}\psi^2)  e^{\frac{|\mathbf{x}|^2}{4}} \, d\mathcal{H}^n\geq -K \eps  \int_{\Sigma}\psi^2e^{\frac{|\mathbf{x}|^2}{4}} \, d\mathcal{H}^n
	$$
	for some $K>0$ that depends only on $\Sigma.$
	
	By \cite[Lemma 4.1]{BWIntegerDegree}, $L_\Sigma$ has a discrete spectrum and it follows from the above that 
	the lowest eigenvalue satisfies $\lambda_0(L_{\Sigma})\geq -K\eps$.  
	However, as $\Sigma$ is strictly unstable,  $\lambda_0(L_\Sigma)<0$.  Hence, when  $\eta_1=\frac{1}{2}\min \set{-K^{-1} \lambda_0(L_\Sigma), 8(\bar{M}_0)^{-1}}>0$ there can be no such $u$ satisfying $\Vert u\Vert_{C^2} \leq \eta_1$.
\end{proof}

We now begin the construction of the flow line.  It will be convenient to have some notion for EMCFs that are near $\Sigma$ in a controlled way and that have certain good properties.
\begin{defn}
	Let $\eta_0$ be the constant given by \cref{GraphMCLem}.
	For $T>0$ and $0<\eta<\eta_0$, 
	$v\in C^\infty([0, T] \times \Sigma)$ is an  \emph{$\eta$-uniform EMCF on $[0, T]$} if 
	\begin{enumerate}
		\item $\Vert v\Vert_{C^{2, \alpha}_P} \leq \eta$;
		\item $\Sigma_t=\Gamma_{v(t, \cdot)}$ is a smooth hypersurface for all $t\in [0, T]$;
		\item $t\mapsto \Sigma_t$ is a EMCF for $t\in [0, T]$.
	\end{enumerate}
	We say such a $v$ satisfies a \emph{reverse Poincar\'{e} inequality with constant $\kappa>1$}, if, in addition,
	\begin{enumerate}[resume]
		\item  For all $t\in [0, T]$, $\Vert v(t, \cdot)\Vert_{W, 1}\leq \kappa\Vert v(t, \cdot)\Vert_{W, 0}<\infty$.
	\end{enumerate}
	We say such a $v$ is \emph{monotone} if, in addition,
	\begin{enumerate}[resume]
		\item  For all $0\leq t\leq t'\leq T$, $0\leq v(t,\cdot)\leq v(t', \cdot)$.
	\end{enumerate}
\end{defn}
By \cref{GraphMCLem},  $v$,  an $\eta$-uniform EMCF on $[0, T]$,  satisfies a parabolic PDE of the form
\begin{align}
\label{graphical-rmcf-equation}
\frac{\partial v}{\partial t} = L_\Sigma v + Q(v,\xX\cdot \nabla_\Sigma v  , \nabla_\Sigma v, \nabla^2_\Sigma v).
\end{align}
Here $Q$ is the homogeneous degree-two polynomial given by \cref{GraphMCLem} which we will sometimes shorten to $Q[v]$. 
We may treat the equation \cref{graphical-rmcf-equation} as a quasilinear parabolic PDE with time independent coefficients bounded in $C^{0,\alpha}(\Sigma)$.   By the strict maximum principle if $v$ is monotone and $\Sigma$ is connected, then either $v$ is strictly monotone, i.e., the inequality in Item (5) above is strict, or is a static solution to \eqref{graphical-rmcf-equation}, i.e., each $\Sigma_t$ is a self-expander.

Standard existence results for MCF establish a short time existence and uniqueness in the class of $\eta$-uniform EMCF for hypersurfaces close to $\Sigma$.
\begin{prop}\label{IterativeRegProp}
	Given $0<\eta<\eta_0$, there is an $\alpha_0=\alpha_0(\Sigma, \eta)>0$ and  $\delta_0 = \delta_0(\Sigma,\eta)>0$ so that, if $v_0\in C^{\infty}(\Sigma)$ satisfies $\Vert v_0\Vert_{C^{2, \alpha}}\leq \alpha_0$, then there is a unique $\eta$-uniform EMCF on $[0, \delta_0]$, $v$, with $v(0,\cdot)=v_0$. 
\end{prop}
\begin{proof}
	By \cref{GraphMCLem},  for $\alpha_0$ sufficiently small depending only on $\Sigma$, the map
	$$
	\mathbf{F}_0\colon\Sigma\to \Real^{n+1}, p \mapsto \mathbf{x}(p)+v_0(p) \mathbf{n}_\Sigma(p)
	$$
	is a $C^{2, \alpha}$ diffeomorphism onto $\Sigma_0=\mathbf{F}_0(\Sigma)=\Gamma_{v_0}$.  It follows from the interior estimates of Ecker--Huisken  \cite{EckerHuiskenInterior} -- specifically \cite[Theorem 4.2]{EckerHuiskenInterior} -- that there is short time existence of the MCF starting from $\Sigma_0$ and hence also of the EMCF. This flow is parameterized by a map $\mathbf{F}(t, \cdot)\colon [0, \delta_0]\times \Sigma\to \Real^{n+1}$ with $\mathbf{F}(0, \cdot)=\mathbf{F}_0$. Moreover, the map $\mathbf{F}$ is $\frac{1}{2}$-H\"older continuous in time at $t=0$ and smooth for $t>0$.  It follows that there exists a $C^{\frac{1}{2}}_P$ function $v\colon [0,\delta_0] \times \Sigma \to \R$ such that
	\begin{align*}
	\Sigma_t=\mathbf{F}(t, \Sigma) = \Gamma_{v(t,\cdot)}
	\end{align*}
	and that $v$ solves \cref{graphical-rmcf-equation}. Here one uses only that $\Sigma_0$ is expressible locally as a Lipschitz graph on a uniform scale. Standard parabolic Schauder estimates give the bound on $\Vert v \Vert_{C^{2,\alpha}_P}$ as long as $\alpha_0$ is taken sufficiently small -- see the discussion in the proof of \cite[Corollary 4.4]{EckerHuiskenInterior}.  The uniqueness is straightforward -- see also \cite{ChenYin}.
\end{proof}

We now establish estimates that will allow us to iterate Proposition \ref{IterativeRegProp}.
First, we recall an interior $L^2$ to $C^{2,\alpha}_P$ estimate:
\begin{prop}
	\label{L2-schaulder-estimates}
	Suppose $v$ is smooth and satisfies 
	\begin{align}
	\frac{\partial v}{\partial t} =a^{ij} \nabla_{ij} v+ b^j \nabla_jv+c v.
	\end{align}
	on $[-R^2, 0]\times \bar{B}_{R}(\mathbf{0})\subset \Real\times \Real^n$, with $\{a^{ij}\}$ uniformly elliptic. If the coefficients satisfy
	$$
	\Vert a^{ij}\Vert_{C_P^{0, \alpha}}+R \Vert b^j\Vert_{C_P^{0, \alpha}}+ R^2 \Vert c\Vert_{C_P^{0, \alpha}} \leq C_0
	$$
	on  $[-R^2, 0]\times \bar{B}_{R}(\mathbf{0})$.
	Then there is a constant $C>0$ depending on $C_0, n$,  and the ellipticity constant so that, after setting, 	$$
	I_R[v]=R^{-n-2} \int_{-R^2}^0 \int_{\bar{B}_R(\mathbf{0})} |v|^2 \, dx dt,
	$$
	\begin{gather*}
	\abs{v(0,\mathbf{0})}^2 + R^2\abs{\nabla_\Sigma v(0,\mathbf{0})}^2 + R^{4} \abs{\nabla^2_\Sigma v(0,\mathbf{0})}^2 +
	R^4\abs{\partial_t v(0,\mathbf{0})}^2 \leq C I_R[v] \mbox{ and }
	\end{gather*}
	\begin{gather*}
	R^{4+2\alpha} \left( [\nabla^2_\Sigma v]_{\alpha;  [-R^2/4,0] \times \bar{B}_{R/2}(\mathbf{0})}^2+[\partial_t v]_{\alpha;  [-R^2/4,0] \times \bar{B}_{R/2}(\mathbf{0})}^2 \right) \leq C I_R[v].
	\end{gather*} 

\end{prop}
\begin{proof}
	By rescaling the domain, this follows from \cite[Theorem C.2]{ChoiMantoulidis}. 
\end{proof}
Using this we establish a weighted interior estimate for appropriate $\eta$-uniform ECMF.
\begin{prop}\label{GrowthL2WProp}
	Fix $\kappa_1>1, T_1>0$.  There are $\eta_2=\eta_2(\Sigma)>0$, $C_1=C_1(\Sigma, \kappa_1)>0$ and $C_2=C_2(\Sigma,  \kappa_1, T_1)>0$ so that if $v$ is an $\eta_2$-uniform EMCF on $[0, T_1]$, that satisfies a reverse Poincar\'{e} inequality with constant $\kappa_1$, then
	$$
	\Vert v(t, \cdot )\Vert_{W,0} \leq e^{C_1 t} \Vert v(0, \cdot) \Vert_{W,0} \mbox{ and } 	\norm{v}_{C^{2,\alpha}_P ([\frac{1}{2}T_1 ,T_1] \times \Sigma)} \le C_2\norm{v(T_1,\cdot)}_{W,0}.
	$$
\end{prop}

\begin{proof}
	First we establish the existence of a constant $K_1=K_1(\Sigma, T_1)>0$ so that
	\begin{equation} \label{L2toC2aEst}
	\Vert v\Vert_{C^{2,\alpha}_P([\frac{1}{2} T_1,T_1] \times \Sigma)}^2\leq K_1 \int_{0}^{T_1} \Vert v(t,\cdot)\Vert_{W,0}^2\, dt.
	\end{equation}
	Fix $p \in \Sigma$ and $T\in [\frac{1}{2} T_1, T_1]$ and consider the localization at scale $r$ around $(T,p)$:
	$$
	\Sigma_r(T, p)= \left(\left[\frac{1}{2}T_1, T_1\right]\cap [T-r^2,T]\right)\times B_{r}^\Sigma(p)
	$$
	Note the restriction on the time.  Take $r=r(p)$ to be
	\begin{align*}
	r = \frac{1}{2}\min\{\sqrt{T_1}, \abs{\xX(p)}^{-1}, r_0\}.
	\end{align*}
	Here $r_0=r_0(\Sigma)>0$ is the scale where $B^\Sigma_{r_0}(p)$ can be written as a graph over a region in $ T_p \Sigma$ of a function with $C^{2, \alpha}$ norm smaller than $\frac{1}{4}$ -- this scale exists by Lemma \ref{GraphMCLem}. 
	
	By Lemma \ref{GraphMCLem} and the $\eta_2$-uniformity of $v$, with $\eta_2<\eta_0$ sufficiently small,  the  equation \cref{graphical-rmcf-equation}, satisfies all the conditions of \cref{L2-schaulder-estimates} and so, up to increasing $C$, 
	\begin{align*}
	\phantom{\le}\abs{v(T,p)}^2 &+ r^2\abs{\nabla_\Sigma v(T,p)}^2 + r^{4} \abs{\nabla^2_\Sigma v(T,p)}^2 \le Cr^{-n-2} \int_{T-r^2}^{T} \int_{B_r^\Sigma(p)} \abs{v}^2 \, dt \\
	&\le Cr^{-n-2}e^{-\frac{\abs{\xX(p)}^2+r^2-2r \abs{\xX(p)}}{4}} \int_{0}^{T_1} \int_{B_r^\Sigma (p)} \abs{v}^2 e^{\frac{\abs{\xX}^2}{4}}\,  dt\\
	&\le  C r^{-n-2}e^{-\frac{|\mathbf{x}(p)|^2}{4}} \int_{0}^{T_1} \Vert v\Vert_{W,0}^2 \, dt.
	\end{align*}
	Likewise, \cref{L2-schaulder-estimates} and above reasoning implies that,
	\begin{align*}
	r^{4+2\alpha}\left(  [\nabla^2_\Sigma v]_{\alpha; \Sigma_r(T, p)}^2 +[\partial_t v]_{\alpha;\Sigma_r(T,p)}^2\right)\leq Cr^{-n-2}e^{-\frac{\abs{\xX(p)}^2}{4}} \int_{0}^{T_1} \Vert v\Vert_{W,0}^2 \, dt.
	\end{align*}
	When $|\mathbf{x}(p)|\leq R_1=R_1(\Sigma,T_1)=\max\set{T_1^{-\frac{1}{2}}, r_0^{-1}}$,  then $r=\frac{1}{2} R_0^{-1}>0$ is independent of $p$.
	Hence,  the preceding estimates give the bound on scale $r$
	\begin{equation}\label{ScalerEst}
	\Vert v\Vert_{C^{2, \alpha}_P(\Sigma_r(T, p))}^2 \leq C K(p)
	\int_{0}^{T_1} \Vert v\Vert_{W,0}^2\, dt
	\end{equation}
	where, we have increased $C=C(\Sigma, T_1)$ and set
	$$
	K(p)=\left\{\begin{array}{cc} 1 & |\mathbf{x}(p)|\leq R_1 \\  	r^{-n-6-2\alpha}e^{-\frac{\abs{\xX(p)}^2}{4}} & |\mathbf{x}(p)|> R_1.\end{array}\right.
	$$
	
	It remains to establish a bound on scale $\delta$.  
	To do so one observes that up to shrinking $r$ and increasing $C$ if needed, for any
	$ (T', p')\in 	\Sigma_\delta(T, p)$,  \eqref{ScalerEst} holds with $K(p)$ in place of $K(p')$. Moreover, if $\Sigma_r(T', p')$ and $\Sigma_r(T'', p'')$ intersect, then the triangle inequality and properties of the H\"{o}lder seminorms imply
	$$
	\Vert v \Vert_{C^{2, \alpha}_P(\Sigma_r(T', p')\cup \Sigma_r(T'', p''))}\leq 	\Vert v\Vert_{C^{2, \alpha}_P(\Sigma_r(T', p'))} +	\Vert v\Vert_{C^{2, \alpha}_P(\Sigma_r(T'', p''))} 
	$$
	As we can connect any point in $\Sigma_{\delta}(T,p)$ to $(T, p)$ by a piece-wise smooth path consisting of a minimizing geodesic of length at most $\delta$ in $\Sigma$ and a stationary path in time with temporal length at most $\delta^2$, we obtain
	\begin{align*}
	\Vert v \Vert_{C^{2, \alpha}_P(\Sigma_\delta(T,p))}^2& \leq 2(\delta r^{-1})^6 	\Vert v \Vert_{C^{2, \alpha}_P(\Sigma_r(T,p))}^2\leq 	2 C\delta^6 r^{-6}K(p)\int_0^{T_1}\Vert v\Vert_{W,0}^2 \, dt.
	\end{align*}
	When $|\mathbf{x}(p)|\leq R_1$ as $r$ and $K(p)$ are depend only on $T_1$ and $\Sigma$,  the bound of \eqref{L2toC2aEst} holds with an appropriately chosen $K_1$.  When $|\mathbf{x}(p)|\geq R_1$, one uses the fact that on $\rho\geq R_1$,  the functions $E_N \colon \rho\mapsto \rho^N e^{-\frac{\rho^2}{4}}$ have maxima that depend only on $N$ and $R_1$, \eqref{L2toC2aEst} holds after possibly increasing $K_1$.
	
	Differentiating equation \cref{graphical-rmcf-equation} and integrating by parts, both justified as $v$ satisfies the reverse Poincar\'{e} inequality, yields
	\begin{align*}
	\frac{d}{dt}\Vert v(t, \cdot )\Vert_{W, 0}^2&= 2 \int_\Sigma v \left(L_\Sigma v + Q[v]\right) e^{\frac{\abs{\xX}^2}{4}} \, d\mathcal{H}^n \\
	&= 2\int_\Sigma \left(-\abs{\nabla_\Sigma v}^2 + \left( \abs{A_\Sigma}^2 - \frac{1}{2} \right) \abs{v}^2 + vQ[v]\right)  e^{\frac{\abs{\xX}^2}{4}} \, d\mathcal{H}^n \\
	&\le K_2 \Vert v(t,\cdot)\Vert_{W,0}^2
	\end{align*}
	where $K_2=K_2(\Sigma, \eta_1)$.  Here we used the hypotheses that $v$ is an $\eta_1$-uniform EMCF and satisfies the reverse Poincar\'{e} inequality and \cref{GraphMCLem} to obtain
	$$
	\int_\Sigma |v Q[v]| e^{\frac{\abs{\xX}^2}{4}} \, d \mathcal{H}^n \leq  \bar{M}_0 \norm{v(t, \cdot)}_{C^{2,\alpha}} \norm{v(t,\cdot)}_{W,1}^2\leq \eta_1 \kappa_1^2 \bar{M}_0 \Vert v(t, \cdot) \Vert_{W,0}^2.
	$$
	Hence, for sufficiently large $C_1=C_1(\Sigma, \eta_1, \kappa_1)$,
	\begin{align*}
	f_1(t) = e^{-2C_1 t} \norm{v(t,\cdot)}_{W,0}^2
	\end{align*}
	is non-increasing.  This proves the first bound.
	The same reasoning implies
	$$
	\frac{d}{dt}\Vert v(t, \cdot )\Vert_{W, 0}^2	\ge -K_3 \norm{v(t,\cdot)}_{W,0}^2.
	$$
	where $K_3=K_3(\Sigma, \eta_1 ,\kappa_1)$.
	Hence, 
	\begin{align*}
	f_2(t) = e^{K_3 t} \norm{v(t,\cdot)}_{W,0}^2
	\end{align*}
	is non-decreasing. Thus,  by \cref{L2toC2aEst}
	\begin{align*}
	\norm{v}_{C^{2,\alpha}_P([\frac{1}{2}T_1,T_1] \times \Sigma)}^2 &\le K_1 \int_{0}^{T_1} \norm{v(t,\cdot)}_{W,0}^2 \, dt \\
	&\le K_1e^{K_3 T_1} T_1 \norm{v(T_1,\cdot)}_{W,0}^2 \le C_2\norm{v(T_1,\cdot)}_{W,0}^2,
	\end{align*}
	where $C_2 = C_1(S_1, K_1, K_3)$.
\end{proof}

We now assume $\Sigma$ is strictly unstable and trapped, and use construction of \cite{BWTopologicalUniqueness} to produce the flow line from $\Sigma$. Recall, an asymptotically conical hypersurface is \textit{trapped} if it is contained in the region bounded by the outermost and innermost self-expanders -- see e.g., \cite[Section 4]{BWTopologicalUniqueness} for the partial ordering of self-expanders and \cite{BWRelEnt} for the trapping of flows.   By treating each connected component separately, we may assume $\Sigma$ is connected  and unstable -- see \cite[Theorem 6.1]{BWTopologicalUniqueness} for more details.

Let $f$ be the unique positive eigenfunction of $-L_\Sigma$ corresponding to the lowest eigenvalue, normalized so $\Vert f \Vert_{C^{2,\alpha}}= 1$.  As in \cite{BWTopologicalUniqueness}, we consider the map
\begin{align*}
\fF^\eps = \xX|_{\Sigma} + \eps f \nN_\Sigma,
\end{align*}
and put $\Sigma^\eps = \fF^\eps(\Sigma)=\Gamma_{\eps f}$. By \cref{GraphMCLem}, for $\abs{\eps}$ sufficiently small, $\fF^\eps$ is an embedding and $\Sigma^\eps$ is a hypersurface. The asymptotic structure of $f$, \cite[Proposition 3.2]{BWMountainPass}, implies that for for $\abs{\eps}$ sufficiently small $\Sigma^\eps$ is a trapped hypersurface that is asymptotic to $\mathcal{C}$, the asymptotic cone of $\Sigma$.  By \cite[Proposition 5.1]{BWTopologicalUniqueness}, up to shrinking $\abs{\eps}$, there is a unique, smooth MCF starting from $\Sigma^\eps$, $\tilde{\mathcal{M}}^\eps = \{\tilde{\Sigma}^\eps_s\}_{s \in [1,S^\eps)}$, where $S^\eps\in (1, \infty]$ is the maximal time of smooth existence, so that certain rescaled mean curvature of the $\tilde{\Sigma}^\eps_s$ vanishes nowhere. Let $\mathcal{M}^\eps=\{\Sigma_t^\eps\}_{t\in [0, T^\eps)}$ denote the corresponding EMCFs defined on $[0,T^\eps)$.  The construction ensures the $\Sigma_t^\eps$ lie on one side of $\Sigma$ for $\eps\neq 0$ and remain trapped and asymptotic to $\mathcal{C}$ in the sense of Item (3) of \cref{strictly-monotone-flow}.  

We construct an ancient flow by taking appropriate limits of the $\mathcal{M}^\eps$. To do so, we need more quantitative control on the amount of time they remain well-behaved. To that end, 
let
$$
I^{\eps}(\eta)=\set{T\in (0, T^\eps) \mid  \exists  v, \mbox{an $\eta$-uniform EMCF on $[0, T]$ so $\Sigma_t^\eps=\Gamma_{v(t, \cdot)}$}}
$$
where we take this interval to be $[0]$  if there is no such $v$.  Set
$$
T^{\eps}(\eta)=\sup I^{\eps}(\eta).
$$
When $T^{\eps}(\eta)>0$,  denote by $v^\eps\in C^\infty([0, T^\eps(\eta))\times \Sigma)$ the function such that $\Sigma_t^\eps = \Gamma_{v^\eps(t,\cdot)}$.  Observe that when $\alpha_0=\alpha_0(\Sigma, \eta)$ is given by \cref{IterativeRegProp}, then, when $\abs{\eps} <\alpha_0$, $T^{\eps}(\eta)\geq \delta_0$ since we have normalized $\norm{f}_{C^{2, \alpha}}=1$.  

The definition ensures that $v^\eps$ is an $\eta$-uniform ECMF on $[0, T]$ for every $T\in (0, T^\eps)$.  In particular, for $|\eps'|\leq |\eps|$, $T^{\eps}(\eta)\leq T^{\eps'}(\eta)$ and $T^{0}(\eta)=\infty$ as in this case the EMCF is the static flow of ${\Sigma}$ and $v^0$ vanishes.  
For $\eps> 0$, we may choose the unit normal to $\Sigma$ so $v^\eps(0, \cdot) >0$ and hence, by the expander mean convexity,  $v^\eps$ is monotone -- see the proof of  \cite[Proposition 5.1 (3)]{BWTopologicalUniqueness} for a more detailed justification.   

We now establish that for $\eps>0$ sufficiently small the flows $\mathcal{M}^\eps$ exist with good properties for a long time after which they have moved away from $\Sigma$ by a definite amount.
\begin{prop}\label{TimeExistProp} 
	There are constants $\omega_0=\omega_0(\Sigma)>0$,  $\eps_0=\eps_0(\Sigma)>0$ and $\kappa=\kappa(\Sigma)>1$ so that when $0<\eps<\eps_0$, there is a $T^\eps_0\in [1, T^\eps)$ so that
	\begin{enumerate}
		\item $ -\kappa^{-1}\log \eps-\kappa\leq T^\eps_0<T^\eps(\eta_1)<\infty$, where $\eta_1$ is given by Lemma \ref{NonStationaryLem};
		\item On $[0, T^\eps_0]$, $v^\eps$ satisfies a reverse Poincar\'{e} inequality with constant $\kappa$;
		\item  For  $t\in [0, T^\eps_0]$, $\Vert v^\eps(t,\cdot)\Vert_{W,0}\leq \omega_0=\Vert v^\eps(T^\eps_0,\cdot)\Vert_{W,0}$.
	\end{enumerate} 
\end{prop}
\begin{proof}
	Let us write $v_0^\eps=\eps f$ for the initial perturbation from above.  
	As $v^\eps_0$ is a small perturbation by the first eigenfuction normalized to be positive, for $\eps_0=\eps_0(\Sigma)>0$ sufficiently small, one has $\Vert v^\eps_0\Vert _{W,0}^2 =\eps^2 \norm{f}_{W,0}^2 $ and 
	$$E_0 =E_{rel}[\Sigma^\eps_0;\Sigma ] = E_{rel}[\Gamma_{v^\eps_0};\Sigma ]<-\delta \eps^2,$$
	where $\delta = \delta(\Sigma)>0$. The forward monotonicity of \cite{BWRelEnt} and fact that $\Sigma_t^\eps$ are trapped implies this bound holds for all $t\in [0, T^\eps)$.  
	
	Set  $\eta=\min\set{\eta_0(\Sigma), \eta_1(\Sigma), \eta_2(\Sigma)}>0$
	where $\eta_0$ is given by Lemma \ref{GraphMCLem},  Lemma \ref{NonStationaryLem} and  $\eta_1$ is given by Proposition \ref{GrowthL2WProp} and $\eta_2$ is given by .  Let $\alpha_0$ and $\delta_0$ be the constant obtained in \cref{IterativeRegProp} from this $\eta$.  As observed above, by taking $\eps_0<\alpha_0$ we have
	$$
	\Vert v^\eps_0\Vert_{C^{2,\alpha}} =\eps <\alpha_0
	$$
	and hence $T^\eps(\eta )>\delta_0>0$.  In particular, we may take $v^\eps$ to be defined at least on $[0, \delta_0]$. 
	
	As the relative entropy of the $\Sigma_t$ is negative,  we may appeal to \cite[Corollary 4.3]{ChenAncient} to see that for any $t\in [0, T^\eps(\eta)]$, there is a constant $\kappa=\kappa(\Sigma)>1$ so that
	\begin{align*}
	\norm{v^\eps(t,\cdot)}_{W,1}^2 \le \kappa^2\norm{v^\eps(t,\cdot)}_{W,0}^2.
	\end{align*}
	That is, $v^\eps$ satisfies the reverse Poincar\'{e} inequality with constant $\kappa$ for any $t\in [0, T^\eps(\eta)]$. Our choice of $\eta$ allows us to apply \cref{GrowthL2WProp} and obtain
	\begin{align}
	\label{shorttimel2}
	\norm{v^\eps(t,\cdot)}_{W,0}\le \eps C_3 e^{ C_1 t}, \mbox{ for }	t\in [0, T^\eps(\eta)]. 
	\end{align} 
	Using $T^\eps(\eta)\geq \delta_0$ and the time translation of EMCF, one also obtains from \cref{GrowthL2WProp} that, for $t\in [\delta_0, T^\eps(\eta)]$,
	\begin{align}\label{shorttimeCP}
	\norm{v^\eps(t,\cdot)}_{C^{2,\alpha}} \le \norm{v^\eps}_{C^{2,\alpha}_P([\frac{1}{2}\delta_0,\delta_0]\times \Sigma)} \le  C_2\norm{v^\eps(t,\cdot)}_{W,0}.
	\end{align}
	Now fix a $\omega_0=\omega_0(\Sigma)>0$ small enough so $C_2\omega_0 <  \alpha_0 $, and set 
	\begin{align*}
	T' = \sup \{T \in (0,T^\eps(\eta)) \mid  \forall t\in [0, T], \norm{v^\eps(t,\cdot)}_{W,0} \leq  \omega_0\}.
	\end{align*}
	Fix $\eps_0=\eps_0(\Sigma)>0$ small enough so $\eps_0 C_3 e^{C_1 \delta_0}<\omega_0$.  For $\eps<\eps_0$, \cref{shorttimel2} implies that $T' \geq \delta_0$.   Indeed, by \cref{shorttimel2} we see that, 
	after possibly increasing $\kappa$,
	$$
	T'\geq -C_1^{-1} \log \eps+C_1^{-1} \log( C_3^{-1} \omega_0)\geq -\kappa^{-1} \log \eps-\kappa.
	$$

	We now argue that $T' < T^\eps(\eta)$ for every $\eps < \eps_0$ -- in particular $T'<\infty$ and 
	$$
	\Vert v^\eps(T', \cdot) \Vert_{W,0}=\omega_0.
	$$
	Suppose that $T' = T^\eps(\eta)\geq \delta_0$.  If $T^\eps(\eta) < \infty$, then by \cref{shorttimeCP},  for $t\in [\delta_0, T']$, 
	$$\norm{v^\eps(t,\cdot)}_{C^{2,\alpha}} <\alpha_0 $$ 
	and so $ T'<T'+\delta_0\leq T^\eps(\eta)$ which is a contradiction.  Hence,  it remains to consider the case $T' =T^\eps(\eta)= \infty$.  Take a sequence of times $t_i \to \infty$ and consider the translated graphical EMCF given by $v_i(t,\cdot) = v^\eps(t + t_i, \cdot)$. Clearly $\norm{v_i}_{C^{2,\alpha}_P(I\times \Sigma)} \le \eta$ for any compact interval $I\subset (-t_i,  \infty)$, and so, by the Arzel\`{a}--Ascoli theorem, there exists a subsequence, still denoted by $v_i$, such that $v_i\to \bar{v}$ in $C^2_{loc}(I\times \Sigma)$.  Consequently $\bar{v}$ satisfies \eqref{graphical-rmcf-equation} on $(-\infty, \infty) \times \Sigma$ and is thus a smooth function by standard parabolic estimates.  Moreover, it follows from the uniform bounds on the $v_i$'s that
	$$\norm{\bar{v}}_{C^{2,\alpha}_P((-\infty, \infty)\times \Sigma)} \le \eta.$$
	By construction and the fact that the $v^\eps$ are monotone, for any $t\geq t_i$, 
	$$0<\eps f = v^\eps_0  \leq  v_i(t_i,\cdot) \leq  v_i(t, \cdot) \leq \bar{v}(t, \cdot)$$ 
	and so $\bar{v}$ is not identically zero.    In fact, as $\bar{v}$ is the increasing limit of monotone flows it must be static; that is,  the graph of $\bar{v}$ is a smooth self-expander satisfying $\norm{\bar{v}}_{C^{2,\alpha}} \le \eta$. As $\bar{v}$ is non-negative and positive somewhere, this contradicts \cref{NonStationaryLem} and our choice of $\eta$.   The result follows by taking $T_0^\eps=T'$.
\end{proof}
\begin{cor}\label{AncientCor}
	Suppose $\Sigma$ is strictly unstable and trapped.  There is an ancient smooth EMCF $\mathcal{M}=\{\Sigma_t\}_{t\in (-\infty, 0]}$ that satisfies
	\begin{enumerate}
		\item Each $\Sigma_t$ lies on one side of $\Sigma$ and is asymptotic to $\mathcal{C}$ in the sense of Item (3) of \cref{strictly-monotone-flow};
		\item Each $\Sigma_t$ has expander mean curvature nowhere vanishing;
		\item $\lim_{t\to -\infty}\Sigma_t = \Sigma$ in $C^\infty_{loc}(\mathbb{R}^{n+1})$.	
	\end{enumerate} 
\end{cor}
\begin{proof}
	Let $\omega_0, \eps_0, \kappa_0$ and, for each $0<\eps<\eps_0$ small, let $T^\eps_0$ be given by Proposition \ref{TimeExistProp}. By the lower bound on $T^\eps_0$, we have $T^\eps_0 \to \infty$ as $\eps \to 0^+$.  
	Let $\mathcal{M}^{\eps,T_0^\eps}$ be the EMCF obtained from time translating $\mathcal{M}^\eps$ by $-T^\eps_0$, and let $w^{\eps}$ be the corresponding graphical solutions over $\Sigma$. Then, by definition of $T^\eps_0$,
	\begin{align*}
	\norm{w^{\eps}(0,\cdot)}_{W,0} = \omega_0>0,
	\end{align*}
	and $\norm{w^{\eps}(0,\cdot)}_{C^{2,\alpha}} \le \eta$.  As each $v^\eps(t,\cdot)$ satisfies a reverse Poincar\'{e} inequality with constant $\kappa$ for $t\in [0, T_0^\eps]$, the same is true for $w^\eps$ on $[-T_0^\eps, 0]$, 
	
	By the Arzel\`{a}--Ascoli theorem we may take a subsequential limit $\eps _i\to 0^+$ to obtain a graphical EMCF $\mathcal{M}=\{\Sigma_t\}_{t\in (-\infty, 0]}$ so that $\Sigma_t=\Gamma_{{w}(t,\cdot)}$ where $\Vert w(0, \cdot) \Vert_{C^{2, \alpha}}<\eta_1$ and
	\begin{align*}
	\norm{w(0,\cdot)}_{W,0} = \omega_0>0.
	\end{align*} 
	The strong convergence in $\Vert \cdot\Vert_{W, 0}$ holds because of the reverse Poincar\'{e} inequality.
	
	This means that $\mathcal{M}$ is not the static flow of $\Sigma$. As $\mathcal{M}$ is obtained as the smooth limit of flows with nowhere vanishing expander mean curvature, the strong maximum principle implies each component of $\mathcal{M}$ either has expander mean curvature nowhere vanishing or is static.  However, the latter contradicts \cref{NonStationaryLem}.   Finally, one can argue as in the proof of Proposition \ref{TimeExistProp} that the limit as $t\to -\infty$ of the $\Sigma_t$ is a self-expander that is a graph over $\Sigma$ of a non-negative function.  By \cref{NonStationaryLem}, this expander must be $\Sigma$, i.e., the backwards limit in time of $w$ vanishes.
\end{proof}

\begin{rem}
	When the cone $\mathcal{C}$ is generic, this is the unique one-sided EMCF by \cite[Proposition 4.9]{ChenAncient}. The genericity hypothesis can likely be dropped.
\end{rem}
To finish the proof of \cref{main-theorem} we will couple the smooth ancient flow constructed above with an expander weak flow constructed as in \cref{weak-theorem}. 
\begin{proof}[Proof of \cref{main-theorem}]
	Let $\mathcal{M}=\{\Sigma_t\}_{t\in (-\infty,0]}$ be the ancient smooth flow constructed in Corollary \ref{AncientCor} with $\Sigma=\Sigma_-$. By symmetry, we may also assume the $\Omega(t)$ are the closed sets with boundary $\Sigma_t$ so that $\bigcup_{t\leq 0} \Omega(t)=\mathrm{int}(\Omega_-)$. Observe that our construction ensures $\Sigma_t$ is asymptotic to $\mathcal{C}$ in a uniform-in-time manner as in Item (3) of \cref{strictly-monotone-flow}. As $\Omega(0)$ is strictly expander mean convex and asymptotically conical, there is a regular strictly monotone expander weak flow $\Omega(t)$ for $t \ge 0$ starting from $\Omega(0)$ by \cref{weak-theorem}.  It is readily checked that $\Omega(t)$ for $t\in\mathbb{R}$ satisfies the properties in \cref{strictly-monotone-flow}. 
\end{proof}

\appendix

\section{Proofs of some preliminary results} \label{ProvePrelim}
In this appendix, we give detailed proofs of some results from \cref{PrelimSec}.

\subsection{Proof of \cref{AvoidanceThm}} \label{ProveAvoidanceSec}
We adapt the arguments of \cite[Theorem 16]{HershkovitsWhite} to weak $X$-flows generated by spacetime sets. First we need a couple of auxiliary propositions.

\begin{prop}
	\label{ContCor}
	Suppose $Z$ is a weak $X$-flow generated by a closed set $\Gamma\subseteq C_{N,T_0}$ with starting time $T_0$ and $T>T_0$. 
	\begin{enumerate}
		\item \label{DistItem} For any $p\in Z(T)\setminus \Gamma(T)$, $\dist(p,Z(t))^2\leq 2c_0(T-t)$ for $t<T$ close to $T$. Here $c_0$ is the constant given by \cref{spherical-barrier}.
		\item If $u\colon N\times \Real \to \Real$ is continuous, then,  when $ \tilde{u}(t)=\inf_{x\in Z(t)} u(x,t)$, 
		\[
		\tilde{u}(T)\geq \min\set{\limsup_{t\uparrow T} \tilde{u}(t), \inf_{y\in \Gamma(T)} u(y,T)}.
		\]
	\end{enumerate}
\end{prop}

\begin{proof}
	As $p\notin\Gamma(T)$ and $\Gamma$ is closed, there is a radius $r>0$ and an $\eps>0$ with $T-\eps>T_0$ such that $B_r(p)\cap \Gamma(t)=\emptyset$ for $T-\eps\leq t\le T$. Thus, the restriction, $Z^\prime$, of $Z$ to $B_r(p)\times [T-\eps,T]$ is a weak $X$-flow in $B_r(p)$ with starting time $T-\eps$. Appealing to \cref{spherical-barrier}, one sees that Theorem 3(ii) and consequently Corollary 4(i) of \cite{HershkovitsWhite} hold for the weak $X$-flow $Z'$. That is, $\dist(p,Z^\prime(t))^2\le 2c_0(T-t)$ for $t<T$ close to $T$. As $Z^\prime\subseteq Z$, the inequality holds with $Z^\prime$ replaced by $Z$. This proves the first claim. 
	
	For the second claim, we note that if $x \in \Gamma(T)$, clearly
	\[
	u(x,T)\geq \inf_{y\in\Gamma(T)} u(y,T).
	\]
	If $x \in Z(T) \setminus \Gamma(T)$, the by the first claim, for each $t<T$ close to $T$ we can find $x_t \in Z(t)$ with $\dist(x_t,x) \le \sqrt{2c_0(T-t)}$. Hence $\tilde{u}(t)\leq u(x_t,t)$ and so
	\[
	\limsup_{t \uparrow T} \tilde{u}(t) \le \limsup_{t \uparrow T} u(x_t,t) = u(x,T).
	\]
	Combining the two cases gives that, for each $x\in Z(T)$,
	\[
	u(x,T)\geq\min\set{\limsup_{t\uparrow T} \tilde{u}(t), \inf_{y\in\Gamma(T)} u(y,T)}.
	\]
	Taking the infimum yields the desired inequality.
\end{proof}

For $v\in T_p N$, let
\[
\mathrm{Ric}^X(v,v)=\mathrm{Ric}(v,v)+v\cdot\nabla_v X.
\]

\begin{prop} 
	\label{distance-gap}
	Let $t\in [0,T]\mapsto K(t)$ be a smooth barrier in $N$ such that $\Phi_K^X(p,t)\leq 0$ for all $t\in [0,T]$ and $p\in\partial K(t)$. Suppose $\eta>0$ and $\lambda\in\mathbb{R}$ are such that the set 
	\[
	Q=\set{(p,t) \mid t\in [0,T], \dist(p,\partial K(t))\leq e^{\lambda t} \eta}
	\]
	is compact, and $\mathrm{Ric}^X > \lambda$ on $\bigcup_{t\in [0,T]} Q(t)$. Let $Z$ be a weak $X$-flow generated by a closed subset $\Gamma\subseteq C_{N,0}$. Assume either $\bigcup_{t\in [0,T]} K(t)$ or $\bigcup_{t\in [0,T]} Z(t)$ is compact. If, for all $t\in [0,T]$,
	\[
	\dist(\Gamma(t), K(t))\geq e^{\lambda t} \eta,
	\]
	then,  for all $t\in [0,T]$,
	\[
	\dist(Z(t), K(t)) \geq e^{\lambda t} \eta.
	\]
\end{prop}

\begin{proof}
	It suffices to show that given $\eta^\prime\in (0,\eta)$, one has $\dist (Z(t),K(t))>e^{\lambda t}\eta^\prime$ for all $t\in [0,T]$. Suppose this is not true for some $\eta^\prime$. If 
	\[
	Q^\prime=\set{(p,t) \mid t\in [0,T], \dist(p,K(t))\leq e^{\lambda t} \eta^\prime},
	\]
	then $Q^\prime\cap Z$ is a non-empty compact spacetime set. Thus, there is a first time $t_0$ such that $Q^\prime(t_0)\cap Z(t_0)$ is non-empty. The hypotheses ensure $t_0>0$. Suppose $q\in  Q^\prime(t_0)\cap Z(t_0)$. Then, by the definition of $Q^\prime$, 
	\[
	\dist(q, K(t_0))\leq e^{\lambda t_0} \eta^\prime.
	\]
	As $u(x,t)=\dist(x,K(t))$ is continuous, \cref{ContCor} applied to $u$ implies that
	\[
	\dist(q, K(t_0))\geq \min\set{e^{\lambda t_0} \eta^\prime, \dist(\Gamma(t_0), K(t_0))} \geq e^{\lambda t_0} \eta^\prime,
	\]
	where the second inequality uses the hypotheses. Thus, $\dist(q, K(t_0))=e^{\lambda t_0} \eta^\prime$. 
	
	As $Q$ is compact and $K$ is closed, we may assume $\bigcup_{t\in [0,t_0]} K(t)$ is compact by replacing $K$ by $Q\cap K$; this does not decrease the distance to $Z(t)$ for $t\in [0,t_0]$ and keeps it the same for $t=t_0$. Thus, there is a unit-speed, minimizing geodesic $\gamma$ joining $q$ to a point $p\in\partial K(t_0)$ with $\dist(q,p)=\dist(q,K(t_0))$. As $\gamma\subseteq Q$ one has $\mathrm{Ric}^X(\gamma',\gamma')>\lambda$. As $\dist(\Gamma(t_0),K(t_0))>e^{\lambda t_0}\eta^\prime$,  $q\notin\Gamma(t_0)$ and so, as remarked in Section 11 of \cite{HershkovitsWhite}, the proof of \cite[Lemma 10]{HershkovitsWhite} may be adapted straightforwardly to weak $X$-flows. It follows that $\Phi_K^X(p,t_0)>0$. This contradicts the hypotheses and so proves the claim.
\end{proof}

\begin{thm} 
	\label{distance-gap-2}
	For $i=1,2$, let $Z_i$ be a weak $X$-flow generated by a closed set $\Gamma_i\subseteq C_{N,0}$. Suppose $\eta>0$ and $\lambda\in\mathbb
	R$ are such that the set
	\[
	Z_{1,\lambda}^\eta =\bigcup_{t\in [0,T]} \set{p\in N \mid \dist(p, Z_1(t))\leq e^{\lambda t}\eta}
	\]
	is compact. Suppose also that $\mathrm{Ric}^X\geq \lambda$ on $Z_{1, \lambda}^\eta$. If
	\[
	\min\set{\dist(\Gamma_1(t), Z_2(t)),\dist(\Gamma_2(t), Z_1(t))}\geq e^{\lambda t} \eta 
	\]
	for all $t\in [0,T]$, then 
	\[
	\dist(Z_1(t), Z_2(t))\geq e^{\lambda t} \eta 
	\]
	holds for all $t\in [0, T]$.
\end{thm}

\begin{proof}
	Suppose first that $\lambda$ is a strict lower bound of $\mathrm{Ric}^X$ on $Z_{1,\lambda}^\eta$. Given $\eta^\prime\in (0,\eta)$, let 
	\[
	\mathcal{T} = \set{\tau \in [0,T] \mid \dist(Z_1(t),Z_2(t)) \ge e^{\lambda t}\eta^\prime \; \forall t \in [0,\tau]}.
	\]
	By the hypotheses, $0 \in \mathcal{T}$. 
	
	First we show $\mathcal{T}$ is closed. Indeed, let $b = \sup \mathcal{T}$. Suppose $z_1 \in Z_1(b)$ and $z_2 \in Z_2(b)$. We claim $\dist (z_1,z_2)\geq e^{\lambda b} \eta^\prime$. To see this, if $z_1 \in \Gamma_1(b)$ or $z_2\in\Gamma_2(b)$, then it follows from our hypotheses that 
	\[
	\dist(z_1,z_2) \ge \min\set{\dist(\Gamma_1(b),Z_2(b)),\dist(\Gamma_2(b),Z_1(b))} \ge e^{\lambda b} \eta > e^{\lambda b} \eta^\prime.
	\]
	So we may assume that $z_1 \in Z_1(b) \setminus \Gamma_1(b)$ and $z_2 \in Z_2(b) \setminus \Gamma_2(b)$. For $t < b$, the triangle inequality gives
	\begin{equation}
	\label{eq-2-2}
	e^{\lambda t}\eta^\prime \le \dist(Z_1(t),Z_2(t)) \le \dist(Z_1(t),z_1) +\dist(z_1,z_2) + \dist(z_2,Z_2(t)).
	\end{equation}
	By \cref{ContCor} applied to $u_i(x,t) = \dist(x,z_i)$, $i=1,2$, as $\dist(z_i,\Gamma_i(b))>0$, 
	\[
	\limsup_{t \uparrow b} \dist(z_i,Z_i(t)) \le \inf_{x \in Z_i(b)} \dist(x,Z_i(b)) = 0.
	\]
	So taking $t \uparrow b$ in \cref{eq-2-2} gives that
	\[
	\dist(z_1,z_2) \ge e^{\lambda b} \eta^\prime
	\]
	proving the claim. Since this holds for all $z_1\in Z_1(b)$ and $z_2\in Z_2(b)$, we conclude that
	\[
	\dist(Z_1(b),Z_2(b)) \ge e^{\lambda b}\eta^\prime.
	\]
	Hence $b\in\mathcal{T}$ and $\mathcal{T}$ is closed. 
	
	Next we show that $\mathcal{T}$ is open. Let $\tau \in \mathcal{T}$ with $\tau<T$. For $\delta>0$, let 
	\[
	Z_1^\prime(\tau)=Z_1(\tau)\cup \set{p\in N \mid \dist(p,\Gamma_1(\tau))\leq \delta}.
	\]
	By the hypotheses and triangle inequality, we may choose $\delta$ small so that $Z_1^\prime(\tau)$ is compact and so that
	\[
	\dist(Z^\prime_1(\tau),Z_2(\tau))\geq e^{\lambda \tau} \eta^\prime \mbox{ and } \dist(\Gamma_2(\tau),Z_1^\prime(\tau))>e^{\lambda \tau} \eta^\prime.
	\]
	Let $J = \set{p \in N \mid \dist(p,Z^\prime_1(\tau)) \ge e^{\lambda \tau}\eta^\prime}$. Note that $Z_2(\tau) \subseteq J$ and $\Gamma_2(\tau)\subseteq \mathrm{int} J$. By \cite[Theorem A.1]{HershkovitsWhite}, there exists a closed $C^1$ hypersurface $M\subseteq N$ separating $Z^\prime_1(\tau)$ and $J$ such that 
	\begin{equation} \label{EquaSeparate}
	\dist(Z_1^\prime (\tau),M) = \dist(M,J) = \frac{1}{2} \dist (Z_1^\prime (\tau), J) = \frac{1}{2}e^{\lambda\tau} \eta^\prime.
	\end{equation}
	By construction
	\[
	\min\set{\dist(\Gamma_1(\tau),M), \dist(\Gamma_2(\tau),M)}>\frac{1}{2} e^{\lambda\tau} \eta^\prime.
	\]
	
	By White's regularity theorem \cite{WhiteRegularity}, there exists a smooth $X$-flow of closed hypersurfaces, $t\in (\tau,\tau+\eps]\mapsto M(t)$ such that $M(t)$ converges in the $C^1$ topology to $M$ as $t\to\tau$. Denote $M(\tau)=M$. For sufficiently small $\eps$ such that $[\tau,\tau+\eps] \subseteq [0,T]$, we let
	\[
	\tilde{M} = \bigcup_{t \in [\tau,\tau+\eps]} \set{p \in N \mid \dist(p,M(t)) \le \frac{1}{2}e^{\lambda t}\eta^\prime}.
	\]
	Shrinking $\eps$ if necessary, we can also make sure that $\tilde{M}$ is compact, that $\lambda$ is a strict lower bound of $\mathrm{Ric}^X$ on $\tilde{M}$, and that $M(t)$ separates $\Gamma_1(t)$ and $\Gamma_2(t)$ with
	\[
	\min\set{\dist (\Gamma_1(t),M(t)),\dist(\Gamma_2(t),M(t))}>\frac{1}{2} e^{\lambda t} \eta^\prime
	\]
	for $t\in [\tau,\tau+\eps]$. Combined these with \cref{EquaSeparate}, it follows from \cref{distance-gap} applied to the barriers bounded by $M(t)$ and the flows $Z_i$, with $\frac{1}{2} e^{\lambda\tau} \eta^\prime$ in place of $\eta$, that for $t \in [\tau,\tau+\eps]$
	\[
	\min\set{\dist(Z_1(t),M(t)),\dist(Z_2(t),M(t))} \ge \frac{1}{2}e^{\lambda t}\eta^\prime.
	\]
	Since $M(t)$ separates $Z_1(t)$ and $Z_2(t)$, we have $\dist(Z_1(t),Z_2(t)) \ge e^{\lambda t}\eta^\prime$, meaning $\tau + \eps \in \mathcal{T}$. Hence $\mathcal{T}$ is open. 
	
	As $[0,T]$ is connected, it follows that $\mathcal{T}=[0,T]$. That is, $\dist(Z_1(t),Z_2(t))\geq e^{\lambda t}\eta^\prime$ for all $t\in [0,T]$. Sending $\eta^\prime\to\eta$, the inequality also holds for $\eta^\prime=\eta$. 
	
	When $\lambda$ is not a strict lower bound for $\mathrm{Ric}^X$ on $Z_{1,\lambda}^\eta$, we observe that $Z_{1,\lambda'}^\eta \subset Z_{1,\lambda}^\eta$ for all $\lambda' < \lambda$, and we may apply the above to $Z_{1,\lambda'}^\eta$ to conclude
	\[
	\dist(Z_1(t),Z_2(t)) \ge e^{\lambda' t}\eta
	\]
	for all $t \in [0,T]$. Since $\lambda' <\lambda$ is arbitrary, the inequality holds for $\lambda$ as well.
\end{proof}

\begin{proof}[Proof of \cref{AvoidanceThm}]
	Without loss of generality we may assume $T_0=0$. Let $U\subseteq N$ be an open set containing $\bigcup_{t\in [0,T]} Z_1(t)$ with $\overline{U}$ compact and let $\lambda$ be a lower bound of $\mathrm{Ric}^X$ on $U$. Choose $\eta>0$ so small that
	\[
	\inf_{t \in [0,T]} \min \set{e^{-\lambda t}\dist(\Gamma_1(t), Z_2(t)), e^{-\lambda t}\dist(\Gamma_2(t),Z_1(t))} > \eta,
	\]
	and that
	\[
	\bigcup_{t \in [0,T]} \set{p\in N \mid \dist(p,Z_1(t)) \le e^{\lambda t}\eta} \subseteq U.
	\]
	It follows from \cref{distance-gap-2} that $Z_1(t)$ and $Z_2(t)$ are disjoint for all $t \in [0,T]$.
\end{proof}

\subsection{Proof of \cref{CompactFlowProp}} \label{ProveCompactFlowSec}
Without loss of generality we may assume $T_0=0$. By our hypotheses, there is a point $x_0\in N$ and $r>1$ such that $\bigcup_{t\in [0,T]}\Gamma(t)\subseteq B_r(x_0)$. Let $p\in N\setminus B_{4r}(x_0)$. Define $R$ by $\dist(p,x_0)=r+3R$, so $R \ge r>1$. Observe by the triangle inequality $\bar{B}_{2R}(p)\subseteq \bar{B}_{6R}(x_0)$. As $|\nabla X|$ is bounded,
\[
\sup_{\bar{B}_{2R}(p)} |X| \leq 6RC
\]
where $C=|X(x_0)|+\sup_{N} |\nabla X|$. Let $\kappa\in\mathbb{R}$ be a lower bound of the Ricci curvature of $N$. As $\bigcup_{t\in [0,T]}\Gamma(t)$ is disjoint from $\bar{B}_{2R}(p)$, by \cite[Corollary 28]{HershkovitsWhite} there is $h=h(\kappa,n)>0$ such that 
\[
\dist(p,Z(t))>R \mbox{ for $0\leq t\leq\frac{R}{h+6RC}$}.
\]
As $R>1$,
\[
\frac{R}{h+6RC}=\frac{1}{(h/R)+6C}>\frac{1}{h+6C}.
\]
Thus, if $T^\prime=\frac{1}{h+6C}$, then $p\notin \bigcup_{t\in [0,T^\prime]} Z(t)$. This holds for all $p\in N\setminus B_{4r}(x_0)$, so
\[
\bigcup_{t\in [0,T^\prime]} Z(t) \subseteq B_{4r}(x_0).
\]
Hence, by iteration, 
\[
\bigcup_{t\in [0,T]} Z(t)\subseteq B_{4^k r}(x_0)
\]
that is, $\bigcup_{t\in [0,T]} Z(t)$ is compact.

\subsection{Proof of \cref{equivalent-defn}} \label{ProveEquivDefnSec}
We first show (1) and (2) are equivalent. Clearly, (1) implies (2). To see (2) implies (1), suppose (2) holds, but $Z$ is not a weak $X$-flow generated by $\Gamma$. This means that there exists a smooth compact barrier $\tilde{K}(t)$ on $[a,b]$ such that 
\begin{itemize}
	\item $\tilde{K}(t) \cap Z(t) = \emptyset$ for all $t \in [a,b)$. 
	\item $(\tilde{K}(b) \cap Z(b)) \cap \Gamma(b) = \emptyset$. 
	\item For $p \in \tilde{K}(b) \cap Z(b)$, either $p \in \mathrm{int}(\tilde{K}(b))$ or $p \in \partial \tilde{K}(b)$ but $\Phi_{\tilde{K}}^X(p,b) < 0$. 
\end{itemize}
By \cref{ContCor}\footnote{The proof of \cref{ContCor} uses only the spherical strong barriers, so the corollary remains true for flows satisfying (2).} applied to $u(x,t)=\dist(x,N\setminus \tilde{K}(t))$, as $\tilde{K}(t)\cap Z(t)=\emptyset$ for $t\in [a,b)$ one sees that $p$ cannot be in the interior of $\tilde{K}(b)$. Thus necessarily $p \in \partial \tilde{K}(b)$ and $\Phi_{\tilde{K}}^X(p,b) < 0$. As remarked in Section 11 of \cite{HershkovitsWhite}, the barrier modification theorem \cite[Theorem 9]{HershkovitsWhite} may be adapted to produce a strong $X$-flow barrier $\hat{K}$ with $p \in \partial \hat{K}(b)$ and $\hat{K}(t) \subseteq \tilde{K}(t)$ for all $t \in [\hat{a},b]$, where $\hat{a}$ is sufficiently close to $b$. This means $\partial \hat{K}(b) \cap Z(b) \ne \emptyset$, a contradiction. This verifies the claim. 

Next, it is clear that (2) implies (3). Since every strong barrier can be extended slightly due to the strict inequality while remaining disjoint from $\Gamma$, (3) implies (2). This proves (2) and (3) are equivalent. 

By \cref{AvoidanceThm}, as any smooth $X$-flow is automatically a weak $X$-flow, (1) implies (4). To conclude the proof, it suffices to show (4) implies (2). Indeed, suppose (4) holds, but (2) does not; that is, there exists a strong $X$-flow barrier $K(t)$ on $[a,b]$ such that $K(t) \cap \Gamma(t) = \emptyset$ for all $t \in [a,b]$, and that $K(a) \cap Z(a) = \emptyset$ but $K(t) \cap Z(t) \ne \emptyset$ for some $t \in (a,b]$. Shrinking $b$ if necessary we may assume $t = b$ is the first time $K(t)$ intersects $Z(t)$ nontrivially. Since
\[
\min_{t \in [a,b]} \dist(\partial K(t), \Gamma(t)) > 0
\]
by increasing $a$ if necessary, we may assume that the classical $X$-flow $M(t)$ starting from $\partial K(a)$ is smooth on $[a,b]$, and that $M(t)\cap\Gamma(t)=\emptyset$ for all $t \in [a,b]$. Let $t\in [a,b]\mapsto K'(t)$ be the family of closed regions such that $\partial K'(t)=M(t)$ and $K'(t) \cap \Gamma(t) = \emptyset$ for all $t \in [a,b]$. Since $K(t)$ is a strong barrier (i.e., $\Phi_K^X(x,t) < 0$), we have $K(t) \subseteq K'(t)$ for all $t \in [a,b]$. It follows that $K'(b) \cap Z(b) \ne \emptyset$. As remarked in Section 11 of \cite{HershkovitsWhite}, we may adapt \cite[Lemma 18]{HershkovitsWhite} to $X$-flows and get the first contact of $K'(t)$ and $Z(t)$ occurs at a point in $\partial K'(t)$. This is a contradiction to (4). Hence, (4) implies (2).

\subsection{Proof of \cref{ExistBiggestFlow}} \label{ProveExistBiggestFlowSec}
Without loss of generality, assume $T_0=0$. Let 
\[
\mathcal{Y}=\set{\mbox{$Y$ is a weak $X$-flow generated by a closed set $\Upsilon\subseteq\Gamma$}}
\]
and let $Z$ be the closure of $\bigcup_{Y\in\mathcal{Y}} Y$. We show that $Z$ is a weak $X$-flow generated by $\Gamma$. First, as $\Gamma$ is a weak $X$-flow generated by itself, $\Gamma\in\mathcal{Y}$ and so $\Gamma\subseteq Z$. In particular, $\Gamma(0)\subseteq Z(0)$. To see the reverse inclusion, we argue by contradiction. Suppose $p\in Z(0)\setminus\Gamma(0)$. Observe there is a small radius $\delta>0$ and a time $s>0$ so that $B_\delta(p)\cap \Gamma(t)=\emptyset$ for $t\in [0,s]$. Thus, by \cref{spherical-barrier}, there is $c>0$ so that, shrinking $\delta$ and $s$ if necessary, for all $Y\in\mathcal{Y}$
\[
\dist(p,Y(t))>\sqrt{\delta^2-ct} 
\]
for $t\in [0,s]$. This contradicts the construction of $Z$ and so $Z(0)=\Gamma(0)$. 

Next, suppose $t\in [a,b]\mapsto K(t)$ is a strong barrier with $a\ge 0$ such that $K(a)\cap Z(a)=\emptyset$ and $K(t)\cap \Gamma(t)=\emptyset$ for $t\in [a,b]$. For $\eps>0$ sufficiently small, the flow 
\[
t \in [a,b] \mapsto K_\eps(t)=\set{x\in N \mid \dist(x,K(t))\leq \eps}
\]
is a strong barrier such that $K_\eps(a)\cap Z(a)=\emptyset$ and $K_\eps(t)\cap\Gamma(t)=\emptyset$ for $t\in [a,b]$. Let $Y\in\mathcal{Y}$ be a weak $X$-flow generated by $\Upsilon$. Then $K_\eps(a)\cap Y(a)=\emptyset$ and $K_\eps(t)\cap\Upsilon(t)=\emptyset$ for $t\in [a,b]$. Thus $K_\eps(t)\cap Y(t)=\emptyset$ for $t\in [a,b]$. In particular, $\dist(K(t),Y(t))\geq \eps$ for $t\in [a,b]$. As $Y$ is an arbitrary element of $\mathcal{Y}$, it follows that $\dist(K(t),Z(t))\geq \eps$ for $t\in [a,b]$. By \cref{equivalent-defn}, this proves that $Z$ is a weak $X$-flow generated by $\Gamma$. 

Moreover, suppose $Z'$ is a weak $X$-flow generated by a closed set $\Gamma'\subseteq Z$ with starting time $T'_0\geq 0$. It is readily checked that $Z\cup Z'$ is a weak $X$-flow generated by $\Gamma$. Thus $Z\cup Z' \in\mathcal{Y}$ and so $Z'\subseteq Z$, and $Z$ is the biggest $X$-flow we are looking for, which we will denote by $Z = F^X(\Gamma,0)$. Uniqueness of $Z$ follows immediately from the definition. \par 
Finally, suppose $\hat{\Gamma}\subseteq C_{N,0}$ is a closed subset with $\hat{\Gamma}|_{[0,T_1]}=\Gamma|_{[0,T_1]}$. As $F^X(\hat{\Gamma}; 0)|_{[0,T_1]}$ is a weak $X$-flow generated by $\hat{\Gamma}|_{[0,T_1]}\subseteq\Gamma$,
\[
F^X(\hat{\Gamma};0)|_{[0,T_1]}\subseteq F^X(\Gamma;0)|_{[0,T_1]}.  
\]
By symmetry, 
\[
F^X(\Gamma;0)|_{[0,T_1]}\subseteq F^X(\hat{\Gamma};0)|_{[0,T_1]}.
\]
Combining these inclusions gives that $F^X(\hat{\Gamma};0)$ coincides with $F^X(\Gamma;0)$ on $[0,T_1]$.

\section{Regularizing certain singular $X$-mean-convex subsets} \label{RegularXMCSec}
We wish to formalize the notion that an appropriate class of singular $X$-mean-convex domains can be approximated by smooth $X$-mean-convex domains. Specifically, we show:
\begin{prop} \label{RoundingProp}
Suppose $(N,X)$ is tame.  Let  $U_1,U_2\subseteq N$ be two closed sets satisfying:
	\begin{enumerate}
		\item $U_1,U_2$ are both smooth and strictly $X$-mean-convex;
		\item $U=U_1\cap U_2$ is non-empty and compact;
		\item $\partial U_1$ intersects $\partial U_2$ transversally.
	\end{enumerate}
	Then for any sufficiently small $\eps>0$ there are compact sets $U^\pm_\eps\subseteq N$ satisfying:
	\begin{enumerate}
		\item $U^\pm_\eps$ is smooth and strictly $X$-mean-convex;
		\item $U^-_\eps\subseteq \mathrm{int}(U)\subseteq U\subseteq \mathrm{int}(U_\eps^+)$;
		\item As $\eps\to 0$ the $\partial U^\pm_\eps$ converge as closed sets to $\partial U$, and the convergence is smooth on compact subsets of $N\setminus (\partial U_1\cap \partial U_2)$.
	\end{enumerate}
\end{prop}	
To prove this we use the classical $X$-flow to show smooth out the boundary of certain less singular $X$-mean-convex domains while preserving $X$-mean-convexity.  
\begin{defn} \label{ViscosityXMCDefn}
	Let $V\subseteq N$ be a closed set with non-empty interior and $C^{1,1}$-regular boundary. We say $V$ is \emph{strictly viscosity $X$-mean-convex} if every closed subset $K\subseteq V$ with $\partial K$ a $C^{2}$-regular hypersurface has the property that, if $p\in \partial V\cap \partial K$, then the $X$-mean-curvature of $\partial K$ at $p$ is non-vanishing and points into $K$, i.e., $H^X_K(p)=-\Div_{\partial K}\mathbf{n}_K(p)+X(p)\cdot\mathbf{n}_{K}(p)<0$ where $\mathbf{n}_{K}$ is the unit outward normal on $\partial K$.
\end{defn}
We have the following approximation result:
\begin{thm} \label{ApproxThm}
Suppose $(N,X)$ is tame. Let $V\subseteq N$ be a strictly viscosity $X$-mean-convex compact set with $C^{1,1}$-regular boundary. For every sufficiently small $\eps>0$, there exists a compact subset, $V_\eps\subseteq \mathrm{int}(V)$, so that:
	\begin{enumerate}
		\item $V_\eps$ is strictly $X$-mean-convex with smooth boundary;
		\item $\partial V_\eps$ is a $C^{1}$ exponential graph over $\partial V$ with $C^1$ norm bounded above by $\eps$.
	\end{enumerate}
	In fact, there is $\eps_1$ small such that, for $0< \eps <\eps_1$, the sets $V_\eps$ can be constructed so that $V_{\eps'}\subseteq \mathrm{int}(V_\eps)$ for $\eps < \eps'$, and 
	$$
	\mathrm{int}(V) \setminus V_{\eps_1} =\bigcup_{0<\eps<\eps_1} V_\eps.
	$$ 
\end{thm}
In order to prove this we will need the following uniqueness result that seems to be not quite standard; see \cite{SchulzeMetzger, HershkovitsReifenberg} for related results.
\begin{prop} \label{UniquenessXFlowProp}
Suppose $(N,X)$ is tame. If $\Sigma\subseteq N$ is a (two-sided)  $C^{1,1}$-regular closed hypersurface, then there is a unique solution, $t\in [0,T)\mapsto\Sigma(t)$, of the $X$-mean-curvature flow which is smooth in maximal existence interval $(0,T)$ and so $\Sigma(t)\to \Sigma(0)=\Sigma$ in $C^1$ as $t\to 0^+$.
\end{prop}
\begin{proof}
	The existence follows readily from the local regularity of White \cite{WhiteRegularity}.  Hence, it remains only to show that it is the unique smooth flow.
	
	To that end, pick a small $\delta>0$ and use a convolution argument to approximate $\Sigma$ by a smooth hypersurface, $\Sigma'$, so that $\Sigma$ can be written as a normal exponential graph over $\Sigma'$ by a function whose $C^1$ norm is bounded above by $\delta$. Moreover, there is a constant $C>0$ independent of $\delta$ so that the $C^{1,1}$ norm of the function is bounded above by $C$. 
	
	For two $X$-mean-curvature flows, $t\in [0,T)\mapsto \Sigma_i(t)$ with $\Sigma_i(0)=\Sigma$ for $i=1,2$, there is $T' < T$ so that, for all $t \in (0,T')$, one can express both of these as normal graphs, $u_i(t,\cdot)$, $i=1,2$, over $\Gamma$ with $C^1$ norm bounded by $2\delta$ and $C^{1,1}$ norm bounded by $2C$. In fact, for $t\in (0,T')$ both functions are smooth and will have $C^2$ norm bounded by $2C$. 	If $v=u_1-u_2$, then $v$ solves a uniformly parabolic equation in divergence form with coefficients that are $C^1$. Moreover, at time $t=0$, $v$ identically vanishes. It follows from the strong maximum principle that $v$ identically vanishes and so $\Sigma_1(t)=\Sigma_2(t)$ for all $t\in [0,T')$. As the flows remain compact (by \cite[Theorem 29]{HershkovitsWhite}), uniqueness up to $T$ follows by iteration.
\end{proof}
\begin{proof}[Proof of \cref{ApproxThm}]
	By standard convolution arguments and the hypothesis that $\partial V$ is $C^{1,1}$-regular,  there is $C>0$ so that, for every sufficiently small $\eps>0$,  there is a smooth compact set, $K_\eps\subseteq N$, so that:
	\begin{enumerate}
		\item $K_\eps \subseteq\mathrm{int}(V)$;
		\item $\sup_{\partial K_\eps} |A_{\partial K_\eps}|\leq C$;
		\item $\partial K_\eps$ can be written as a $C^{1}$-graph over $\partial V$ with $C^1$ norm bounded above by $\eps$.
	\end{enumerate}
	Let $t\in [0,T_\eps)\mapsto K_\eps(t)$ be a continuous family of compact sets with $K_\eps(0)=K_\eps$ and whose boundary is a smooth $X$-mean-curvature flow. By Items (2) and (3) and \cite{WhiteRegularity}, there is a $T>0$ so that the maximum time of existence $T_\eps>T$ for all small $\eps$.  
	
	We claim that $K_\eps(t)\subseteq \mathrm{int}(V)$ for all $t\in (0, T_\eps)$.  To see this, we observe that $K_\eps(0)=K_\eps\subseteq \mathrm{int}(V)$ by construction. The curvature bound implies that there is a $T_\eps'\in (0, T_\eps)$ so that for $t\in [0, T_\eps']$, $K_\eps(t)\subseteq \mathrm{int}(V)$.  If the claim is not true, then there is a first time $t'\in [0, T_\eps)$ for which the inclusion fails.  Clearly, $t'>T_\eps'$, $K_{\eps}(t')\subseteq V$ and there is a $p\in \partial K_{\eps}(t')\cap \partial V$. As $V$ is strictly viscosity $X$-mean-convex, this implies that the $X$-mean-curvature of $\partial K_{\eps}(t')$ at $p$ points into $K_\eps(t')$ which is impossible, as it would mean at some previous time $K_\eps(t)$ was not contained in $V$.
	
	Let $t\in [0,T)\mapsto K(t)$ be a subsequential limit of the $t\in [0,T)\mapsto K_\eps (t)$ as $\eps\to 0$. Thus, $t\in [0,T)\mapsto \partial K(t)$ is a solution of the $X$-mean-curvature flow which is smooth on $(0,T)$ and $\partial K(t)\to\partial V$ in $C^1$ as $t\to 0^+$. As $K_\eps(t)\subseteq V$ for all $t\in [0,T)$ and small $\eps> 0$, the nature of the convergence ensures that $K(t)\subseteq V$ for all $t\in [0,T)$ with equality at time $t=0$.  Because $V$ is strictly viscosity $X$-mean-convex and $\partial K(t)$ is smooth for all $t\in (0,T)$, arguing as above implies that $K(t)\subseteq \mathrm{int}(V)$ for all $t \in (0,T)$.
	
	By the uniqueness result of Proposition \ref{UniquenessXFlowProp}, one appeals to \cite[Theorem 31]{HershkovitsWhite} to see that the biggest $X$-flow of $V$ agrees with the domain bounded by the smooth $X$-mean-curvature flow as long as the latter exists, i.e., $K(t)=F^X_t(V\times\{0\}; 0)$ for $t\in (0,T)$. As $F^X_t(V\times\{0\}; 0)\subseteq \mathrm{int}(V)$ for $0<t<T$, it follows that $F^X_{t+h}(V\times\{0\};0)\subseteq \mathrm{int}(F_t^X(V\times\{0\};0))$ for all $t,h>0$ -- see \cite[Remark 33]{HershkovitsWhite}. In particular, for $t\in (0,T)$ the $X$-mean curvature vector of $\partial K(t)$ either vanishes at a point or points into $K(t)$.  However, if the $X$-mean-curvature vector vanishes at a point $p\in \partial K(t)$, then, by the strong maximum principle, one must have it vanish on the entire connected component of $\partial K(t)$ containing $p$. This means that there is a component of $\partial V$ that is smooth and on which the $X$-mean-curvature vector vanishes. This would contradict the assumption that $V$ is strictly viscosity $X$-mean-convex and so the $X$-mean curvature vector of $\partial K(t)$ is non-vanishing everywhere and points into $K(t)$ for all $t\in (0,T)$.  
	
	To conclude the proof we observe, for sufficiently small $t>0$, $K(t)$ satisfies the conclusions of the theorem and so we may take $V_\eps=K(t)$ for an appropriate choice of $t$.
\end{proof}
We can now prove Proposition \ref{RoundingProp}.
\begin{proof}[Proof of Proposition \ref{RoundingProp}]
	We first construct $U^+_\eps$. For a set $V\subseteq\mathbb{R}^{n+1}$, denote by $\mathcal{T}_\eps(V)$ the $\eps$-tubular neighborhood of $V$. Fix a sufficiently large $R$ so $U=U_1\cap U_2\subseteq B_R(p)$. By taking $\eps$ sufficiently small one can ensure that the $\partial\mathcal{T}_\eps(U_i)\cap B_R(p)$ are smooth and have $X$-mean-curvature vector pointing into $\mathcal{T}_\eps(U_i)$. 
	
	It is not hard to see that, for $\eps>0$ sufficiently small, $\mathcal{T}_\eps(U)$ has $C^{1,1}$-regular boundary that is actually piecewise smooth. Moreover, $\mathcal{T}_\eps(U)$ is strictly viscosity $X$-mean-convex. Indeed, by taking $\eps$ small enough $\partial\mathcal{T}_\eps(U)$ can be decomposed into pieces lying in $\partial \mathcal{T}_\eps(U_1)$, $\partial \mathcal{T}_\eps(U_2)$ and $\partial \mathcal{T}_\eps(\partial U_1\cap \partial U_2)$ and on each of these pieces  the $X$-mean-curvature vector strictly points into $\mathcal{T}_\eps(U)$.  This last claim follows from a computation similar to that in the proof of \cref{spherical-barrier}. The non-smooth points of $\partial\mathcal{T}_\eps(U)$ lie on $\partial\mathcal{T}_\eps(U_i)\cap \partial \mathcal{T}_\eps(\partial U_1\cap \partial U_2)\subseteq B_R(p)$ and one readily checks that \cref{ViscosityXMCDefn} holds at these points. We conclude the existence of $U_\eps^+$ by applying \cref{ApproxThm} to $\mathcal{T}_\eps(U)$.
	
	In order to construct $U^-_{\eps}$ we consider the inward parallel domains
	\[
	U^{\eps'}_i = U_i\setminus \mathcal{T}_{\eps'}(U_i^c), i=1,2,
	\]
	for $\eps'$ sufficiently small, so $U^{\eps'}_1$ and $U^{\eps'}_2$ satisfy the same hypotheses as $U_1$ and $U_2$. By choosing $\eps<\eps'/4$ and possibly shrinking $\eps'$, we can construct a domain $U^-_{\eps}$ by applying the first part of the proof to $U^{\eps'}_1$ and $U^{\eps'}_2$ in place of $U_1$ and $U_2$. Moreover, $U^-_{\eps}$ satisfies 
	$$
	U^{\eps'}_1\cap U^{\eps'}_2\subseteq U_\eps^-\subseteq \mathrm{int}(U).
	$$
	Items (1)-(3) immediately follow from the construction and \cref{ApproxThm}.
\end{proof}

\section{A property of the biggest flow of an asymptotically conical hypersurface} \label{BigFlowACSurfaceSec}
Throughout this section, we take $X = -\frac{\xX}{2}$. We show that the biggest $X$-flow of a asymptotic conical hypersurface agrees with the smooth EMCF for a short time.
\begin{prop}[{cf. \cite[Theorem 31]{HershkovitsWhite}}]
	\label{BigFlowProp} 
	Let $\Sigma\subseteq \Real^{n+1}$ be a smooth hypersurface that is $C^2$-asymptotic to a $C^2$-regular cone $\cC$. There is an $\eps>0$ so that the smooth $X$-mean-curvature flow starting from $\Sigma$, $t\mapsto\Sigma(t)$ exists for $t\in [0,\eps_0]$ and agrees with the biggest $X$-flow $F^X(\Sigma\times\{0\};0)$. Moreover, if $\Sigma=\partial U$ for a closed subset $U$, then for $t\in [0,\eps_0]$, $\Sigma(t)=\partial F^X_t(U\times\{0\};0)$.
\end{prop}
\begin{proof}
	 As $\Sigma$ is an embedded hypersurface, there are disjoint open sets, $U^\pm\subseteq \Real^{n+1}\setminus \Sigma$ so $\Sigma=\partial U^\pm$.  
	 As $\Sigma$ is smooth and $C^2$-asymptotic to $\cC$, there is a uniform upper bound on the norm of the second fundamental form, and so there are sequences of smooth compact sets $V_i^\pm\subseteq U^\pm$ so that $U^\pm=\bigcup_{i=1}^\infty V_i^\pm$ and $\Sigma_i^\pm=\partial V_i^\pm$ converge with multiplicity one to $\Sigma$ in $C^\infty_{loc}(\Real^{n+1})$ as $i \to \infty$. As $X=-\frac{\mathbf{x}}{2}$, by  the change of variable \cref{ChangeOfVariables}, we may directly apply Ilmanen's elliptic regularization scheme \cite{IlmanenMAMS} for the usual MCF. In particular, there are small perturbations of $V_i^\pm$ so that the biggest flows of the $\Sigma_i^\pm$ do not fatten, but all other properties are preserved.  Thus, there are unit-regular integral Brakke $X$-flows, $\mathcal{M}_i^\pm$  from $\Sigma_i^\pm$ with
	$$
	\Sigma_i^\pm(t)=\spt(\mathcal{M}_i(t))\subseteq F_t^X(\Sigma_i^\pm\times \set{0}; 0)\subseteq F_t^X(V_i^\pm\times \set{0}; 0).
	$$
	By \cite{WhiteRegularity} and standard compactness results, up to passing to a subsequence,  the $\Sigma_i^\pm$ converge in $C^\infty_{loc}(\mathbb{R}^{n+1}\times [0,\eps_0])$ to smooth $X$-mean-curvature flows $t\in [0,\eps_0]\mapsto \Sigma^\pm(t)$ that satisfy $\Sigma^\pm(0)=\Sigma$.  As this convergence is locally smooth
 when $t\in [0, \eps_0]$, there are open subsets $U^\pm(t)\subseteq \Real^{n+1}$ so that, for such $t$, $\Sigma^\pm=\partial U^\pm(t)$ and 
	$$
U^\pm(t)=\bigcup_{i=1}^\infty F_t(V_i^\pm\times \set{0}; 0).
	$$
 Moreover, as $\Sigma$ is $C^2$-asymptotic to $\cC$, the uniqueness result of \cite{ChenYin} applies, and so there is only one smooth $X$-mean-curvature flow out of $\Sigma$, i.e., $\Sigma(t)=\Sigma^\pm(t)$ for all $t \in [0,\eps_0]$. 
	By \cite[Theorem 30]{HershkovitsWhite}, for $t\in [0,\eps_0]$, $F^X_t(\Sigma\times\{0\};0)$ is disjoint from $F^X_t(V_i^\pm\times \set{0}; 0)$, and so $F_t^X(\Sigma\times\{0\};0)\subseteq \Sigma(t)$. As the set $\set{(x,t) \mid x\in \Sigma(t), t\in [0,\eps_0]}$ is a weak $X$-flow, it follows that $F^X_t(\Sigma\times\{0\};0)=\Sigma(t)$ for $t\in [0, \eps_0]$. 
	
	Finally, set $U^+=\mathrm{int}(U)$. By \cite[Theorems 30]{HershkovitsWhite}, $F^X(U\times\{0\};0)$ is disjoint from the biggest $X$-flow of the $V_i^-$ and , likewise,  contains the biggest flow of $V_i^+$.  Hence, $F^X_t(U\times\{0\};0)\cap U^-=\emptyset $ and $U^+\subseteq F^X_t(U\times\{0\};0)$. As $F^X(\Sigma\times\{0\};0)\subseteq F^X(U\times\{0\};0)$ by the definition of the biggest flow and $\Sigma(t)=\partial U^\pm(t)$, the last claim follows.
\end{proof}

\begin{cor} \label{ExpanderBigFlowCor}
	Let $U\subseteq \Real^{n+1}$ be a closed set so that $\partial U$ is smooth self-expander that is $C^2$-asymptotic to a $C^2$-regular cone. For all $t\geq 0$, 
	$$
	F_t^X(\partial U\times\{0\};0)=\Sigma \mbox{ and } F_t^X(U\times\{0\};0)=U.
	$$ 
\end{cor}

\bibliographystyle{alpha}
\bibliography{EMCref.bib}

\end{document}